\documentclass[12pt,  oneside]{amsart}

\usepackage{inputenc}

\usepackage{stackengine}[2013-09-11]
\usepackage{scalerel}
\stackMath
\def\ptt{\mathrel{\ThisStyle{\raisebox{-.2ex}{$\SavedStyle\scalerel*%
    {\stackinset{c}{}{t}{.5ex}{\smash{-}}{\pitchfork}}{\pitchfork}$}}}}

\usepackage{amscd, amsmath, amssymb, colonequals}

\usepackage[usenames]{xcolor}

\usepackage{amsrefs}

\usepackage[hyperfootnotes=false]{hyperref}

\hypersetup{hidelinks}

\title[]{Equivalence of neighborhoods of embedded compact complex manifolds and higher codimension foliations}
\author[]{Xianghong Gong$^{\dag}$}

\address{Department of Mathematics,
 University of Wisconsin-Madison, Madison, WI 53706, U.S.A.}
 \email{gong@math.wisc.edu}

\author{Laurent Stolovitch$^{\dag\dag}$}
\address{CNRS and Laboratoire J.-A. Dieudonn\'e
U.M.R. 7351, Universit\'e C\^ote d'Azur, Parc Valrose
06108 Nice Cedex 02, France}
\email{stolo@unice.fr}
\thanks{$^{\dag}$Partially supported by a grant from the Simons
Foundation (award number: 505027). $^{\dag\dag}$Research of L. Stolovitch was supported by ANR grant "ANR-14-CE34-0002-01" for the project "Dynamics
and CR geometry".  }


\newcommand{\dist}{\operatorname{dist}}
\newcommand{\DD}[2]{\frac{\partial #1}{\partial #2}}

\newtheorem{thm}{Theorem}[section]
\newtheorem{cor}[thm]{Corollary}
\newtheorem{prop}[thm]{Proposition}
\newtheorem{lemma}[thm]{Lemma}

\theoremstyle{definition}

\newtheorem{defn}[thm]{Definition}
\newtheorem{exmp}[thm]{Example}

\newtheorem{rem}[thm]{Remark}

\renewcommand{\th}[1]{\begin{thm}\label{#1}}
\newcommand{\co}[1]{\begin{cor}\label{#1}}
\newcommand{\eco}{\end{cor}}
\renewcommand{\le}[1]{\begin{lemma}\label{#1}}
\newcommand{\ele}{\end{lemma}}
\newcommand{\pr}[1]{\begin{prop}\label{#1}}
\newcommand{\epr}{\end{prop}}
\newcommand{\pf}[1]{\begin{proof}#1 \end{proof}}

\newcommand{\ga}{\begin{gather}}
\newcommand{\ega}{\end{gather}}
\newcommand{\gan}{\begin{gather*}}
\newcommand{\egan}{\end{gather*}}
\newcommand{\al}{\begin{align}}
\newcommand{\eal}{\end{align}}
\newcommand{\aln}{\begin{align*}}
\newcommand{\ealn}{\end{align*}}
\newcommand{\eq}[1]{\begin{equation}\label{#1}}
\newcommand{\eeq}{\end{equation}}

\newcommand{\ci}{~\cite}

\newcommand{\sk}{\vspace{2ex}\noindent}

\newcommand{\f}[2]{\frac{#1}{#2}}

\newcommand{\cc}{{\bf C}}
\newcommand{\nn}{{\bf N}}
\newcommand{\zz}{{\bf Z}}
\newcommand{\rr}{{\bf R}}

\newcommand{\ov}{\overline}
\newcommand{\ord}{\operatorname{ord}}
\newcommand{\codim}{\operatorname{codim}}
\newcommand{\id}{\operatorname{I}}

\newcommand{\IM}{\operatorname{Im}}

\newcommand{\rank}{\operatorname{rank}}

\newcommand{\cL}{\mathcal}

\newcommand{\I}{\operatorname{I}}

\newcommand{\T}{ T}

\newcommand{\all}{\alpha}

\newcommand{\gaa}{\gamma}
\newcommand{\Gaa}{\Gamma}
\newcommand{\del}{\delta}
\newcommand{\Del}{\Delta}
\newcommand{\var}{\varphi}
\newcommand{\e}{\epsilon}
\newcommand{\om}{\omega}
\newcommand{\Om}{\Omega}

\newcommand{\la}{\lambda}
\newcommand{\ta}{\tau}

\newcommand{\pd}{\partial}

\newcommand{\re}[1]{(\ref{#1})}
\newcommand{\rea}[1]{$(\ref{#1})$}
\newcommand{\rl}[1]{Lemma~\ref{#1}}
\newcommand{\nrc}[1]{Corollary~\ref{#1}}
\newcommand{\rp}[1]{Proposition~\ref{#1}}
\newcommand{\rt}[1]{Theorem~\ref{#1}}
\newcommand{\rd}[1]{Definition~\ref{#1}}
\newcommand{\rrem}[1]{Remark~\ref{#1}}

\newcommand{\rla}[1]{Lemma~$\ref{#1}$}

\newcommand{\rpa}[1]{Proposition~$\ref{#1}$}
\newcommand{\rta}[1]{Theorem~$\ref{#1}$}

\newcounter{pp}
\newcommand{\bpp}{\begin{list}{$\hspace{-1em}(\alph{pp})$}{\usecounter{pp}}}
\newcommand{\epp}{\end{list}}

\newcounter{ppp}
\newcommand{\bppp}{\begin{list}{$\hspace{-1em}(\roman{ppp})$}{\usecounter{ppp}}}
\newcommand{\eppp}{\end{list}}

\def\beq{\begin{equation}}
\def\eeq{\end{equation}}
\newcommand{\dindice}[2]{{\stackrel{\scriptstyle #1}{\scriptstyle #2}}}

\newcommand{\scalxx}[2]{\left \langle #1, #2 \right   \rangle}
\newcommand{\mf}{\text{mf}}

\begin{document}
\begin{abstract}
We consider an embedded $n$-dimensional compact complex manifold
in $n+d$ dimensional complex manifolds.
We are interested in the holomorphic classification of neighborhoods as part of Grauert's formal principle program.
We will give conditions ensuring that a neighborhood of $C_n$ in $M_{n+d}$ is biholomorphic to a neighborhood of the zero section of its normal bundle. This extends Arnold's result about neighborhoods of a complex torus
in a surface.
We also prove the existence of a holomorphic foliation in $M_{n+d}$ having $C_n$ as a compact leaf, extending Ueda's theory to the high codimension case.
Both problems appear as a kind linearization problem involving {\it small divisors condition} arising from solutions to their cohomological equations.
\end{abstract}

\date{\today}
 \maketitle

\tableofcontents


\setcounter{thm}{0}\setcounter{equation}{0}
\setcounter{section}{0}
\section{Introduction}

We are interested in the classification of the germs of neighborhood of an embedded compact complex manifold $C$ in a complex manifold $M$. Here two germs $(M,C)$ and $(\tilde M,C)$ are holomorphically equivalent if there is a biholomorphic mapping $F$ fixing $C$ pointwise and  sending a neighborhood $V$ of $C$ in $M$ into a neighborhood $\tilde V$ of $C$ in $\tilde M$.  These considerations can be useful 
to
 extend
  holomorphic objects such as cohomology classes of holomorphic sections of bundles over $C$ or functions on $C$ to a neighborhood of $C$ in $M$. Indeed, it might be that such an extension problem is much easy to  %
  solve on an equivalent neighborhood. We are also interested in the existence of a non-singular holomorphic foliation of the germ of neighborhood of $C$ in a complex manifold having $C$ as a compact leaf.  We refer to it as a ``horizontal foliation".

A neighborhood $V$ of an embedded complex  manifold  $C_n$ in    $M_{n+d}$ has local holomorphic charts $(h_j,v_j)=\Phi_j$ mapping $V_j$ onto   $\hat V_j$ in $\cc^{n+d}$ with $n=\dim C$. Here $\cup V_j$ is a neighborhood of $C$ and  $U_j:=V_j\cap C$ is defined by $v_j=0$. The above-mentioned classification of the germs of neighborhoods of $C$ is then the classification of transition functions $\Phi_{kj}:=\Phi_k\Phi_{j}^{-1}$ under
holomorphic conjugacy $F_k^{-1}\Phi_{kj}F_j$.
To such an embedding, one can associate the normal bundle $N_C(M)$
of $C$ in $M$, which has
the transition matrices $g_{kj}(p)$, $p\in U_k\cap U_j$. To this embedding one can associate another natural embedding, namely the embedding of $C$ as the zero section of $N_C(M)$.
Under a mild assumption, this last embedding $(N_C(M), C)$ naturally
serves as
a first order approximation of $(M,C)$.
Let $\var_j=\Phi_j|_{U_j}$ and let $\var_{kj}=\var_k\var_j^{-1}$  be the transition functions of $C$.
To have a neighborhood of $C$ in $M$ equivalent to a neighborhood of the zero
section in $N_C(M)$ is equivalent to 
seeking
 $F_j$ such that $\hat\Phi_{kj}=F_k^{-1}\Phi_{kj}F_j$ are of the form $N_{kj}(h_j,v_j)=(\var_{kj}(h_j),t_{kj}(h_j)v_j)$
with $t_{kj}(h_j)=g_{kj}$, the latter being  %
regarded as the transition functions of a neighborhood of the zero
section of $N_C(M)$. 
We call this process a ``full linearization" of the neighborhood.
The   
above-mentioned
 ``horizontal foliation" will be obtained as a consequence of a ``vertical linearization" of the neighborhood which amounts to 
 seeking  $F_j$ such that $\hat\Phi_{kj}=(\var_{kj}(h_j)+\hat\phi_{kj}^h(h_j,v_j), t_{kj}(h_j)v_j)$. 

Without even considering holomorphic equivalence problem, it is known that there are formal obstructions to
linearizing \cite{NS60,Gr66} or to
 linearizing
vertically~\ci{Ue82} 
a neighborhood; see section~\ref{sec:formal}.  Part of the Grauert formal principle \cite{Gr62, HR64, Ko88, CMS03} is to seek geometry conditions that ensure a holomorphic linearization when the formal obstructions are absent.
In this paper, we will obtain linearizations of a neighborhood of an embedded compact complex manifold $C_n$ at the  absence of formal obstructions under {\it small divisor conditions} in the form of bounds of solutions of cohomology equations involving all symmetric powers of $N_C^*$, the dual of the normal bundle $N_C$ of $C_n$ in $M_{n+d}$.  Because of the very nonlinear nature of the problem, we need to work with a {\it family of nested domains}
on which we
solve and eventually bound the solutions of $1$-cohomological equations. Indeed, we are naturally led to consider shrinking of the domains as we need to get estimates of derivatives of sections (by Cauchy estimates for instance). To be more precise, assume that a $1$-cocycle $f$ with value in the sheaf of sections of holomorphic bundle (involving symmetric power $S^mN_C^*$ for some $m\geq 2$) on $C$   
vanishes
in the $1$st cohomology group over a covering $\cL W$. Then there is
a $0$-cochain $w$ over $\cL W$ such that $\del w=f$. Nevertheless, we need to prove the existence of
a (possibly different) solution $u$ satisfying the linear equation $\del u=f$ and a ``linear" estimate of the form
$\|u\|_{
\cL W}\leq K \|f\|_{\cL W}$ (the norm is either $L^2$ or the sup-norm).
 Because of the nonlinear nature of our problem, we need to solve the linear equation
 iteratively and estimate
solutions
 of the form $\del u_m=
 F_m(f_2,\dots, f_m,u_2,\ldots, u_{m-1})$, $m\geq 2$. Here $F_m(f_2,\dots, f_{m}, u_2,\ldots, u_{m-1})$ is a nonlinear function 
 and vanishes
 in a first cohomology group.  Therefore the bound $K$, depending on $m$, will compound,  which leads to a problem on non-linear estimates.   Here come some of the main issues : we need that, at the limit, the sequence of nested domains, over which the solutions are estimated iteratively,  remains to cover the manifold. 
 And we need to control the growth of the bound $K$ with respect to $m$, that gives rise to the so-called {\it small divisors condition}.
  Therefore,  the existence of {any} bound $K$ for  linear
  solutions $u$ {\it without} shrinking the covering $\cL W$ is a basic
  question. The latter was
  solved affirmatively by Kodaira-Spencer~\cite[eq.~(9),  p.~499]{KS59} for the case of line bundles for a general covering.
For higher rank vector bundles,
  we provide a positive solution in
 the following result~:
%
\begin{prop}
	\label{KnDn}Let $C$ be a compact complex manifold.
	There exists a family of   coverings $\cL U^r=\{U_j^r\}$, $r_*\leq r  
 < r^*$,  of $C$ such that  for any holomorphic vector bundle $E$ over $C$, and
	each $f\in C^1(\cL U
	^{r'}, E)$, the space of $1$-cochains on $\cL U^{r'}$ of holomorphic sections with values in $E$, 
satisfying $
	f=\del u_0$ for some $u_0\in  C^0(\cL U^{r'},E)$, there	exist $u\in  C^0(\cL U^{r'}, E)$   
 and $v\in C^0(\cL U^{r''},E)
	$ such that $\delta u=f$ and $\del v=f$,  and
	\ga
	\label{bKn}
	|u|_{r'}\leq K(E)|f|_{r'},\\
	|v|_{r''}\leq \frac{D(E)}{(r'-r'')^\tau}|f|_{r'}.
	\label{bDn}
\end{gather}
Here   
$r',r''$ are any numbers satisfying
$r_*<r''<r'\leq \tilde r<r^*$ and $r'-r''\leq r^*- %
\tilde r$, and   
$\tau, K(E), D(E)$ are independent of $ r',r''$.
\end{prop}
Here, we have used the sup-norm (or $L^2$-norm) of cochains of holomorphic sections of bundles (see section~\ref{secA.2} for specific notations).
We do not know if $K(E)$ and $D(E)$ are comparable when they are applied to the symmetric powers of $N_C^*$ except when $N_C$ is unitary.
H\"ormander~\ci{Ho65,ohsawa-book} obtained
solutions with bounds for cohomology groups with respect to the $\bar\partial$ operator acting on the sheaf of $(p,q)$-forms with $L^2$ coefficients on $\cc^n$. 

The  %
estimate \re{bDn} was proved by Donin~\ci{Do71} for a special family of
coverings
by the $L^2$ theory. He also raised the question if 
estimate \re{bKn} exists,
i.e.
the basic question mentioned above.
\rp{KnDn} gives us
 a more flexible kind of results and ultimately an estimate that holds without any shrinking for  higher rank vector bundles via the above mentioned nested coverings.
 We also use the $L^2$-theory. We first obtain \re{bDn} by \rt{DGR}. Then \re{bKn} is obtained by \rl{SD-p}. The constant $K(E)$ is defined for the kind of bundles we need in \rd{defK}. This is summarized in \rt{estim-cohom-summary}.
The main results of this paper are based on the existence of {\it  nested finite coverings} proved in subsection \ref{nested}.

 \rp{KnDn} will be a useful tool in this paper. We now formulate our main results.
We say that $T_CM=TM|_C$ {\it splits} if $T_CM=TC\oplus N_C$ holomorphically.
For instance, $T_CE$ splits for any holomorphic vector bundle $E$ over $C$. Here and in the sequel, we  identify $C$ with the zero section of $E$.
We say that $N_C$ is {\it flat} if the transition matrices of $N_C$ are locally constant. We say that $N_C$ is {\it unitary} if its transition matrices are unitary. Note that the maximum principle  implies that a unitary $N_C$ is flat.
We have the following ``vertical linearization" result:
\begin{thm}\label{vlin-int}
Let  $C_n$ be a compact submanifold of $M_{n+d}$  with splitting $T_CM$ and unitary $N_C$.
 Let $\eta_0=1$ and
\beq\label{def-eta-foliation-int}\nonumber
\eta_m
:=K(N_C\otimes S^m(N_C^*))\max_{m_1+\cdots +m_p+s=m} \eta_{m_1}\cdots \eta_{m_p},
\eeq
where  the maximum is taken in $1\leq m_i<m$ for all $i$ and $s\in\Bbb N$.
Assume that there are positive constants $L$, $L_0$ such that
\eq{eta-cov}\nonumber
\eta_m\leq L_0L^m, \quad
m=1,2\dots.
\eeq
Assume that  $H^0(C,N_C\otimes S^\ell(N_C^*))=0$ for
all $\ell>1$ . Assume that either $ H^1(\cL U,N_C\otimes S^\ell (N_C^*))=0$ for
all $\ell>1$ or a neighborhood of $C$ is formally vertically linearizable 
by a formal holomorphic mapping that is tangent to the identity 
and preserves the splitting of $T_CM$.
Then the  embedding is actually holomorphically vertically linearizable.
\end{thm}
When $C$ is  %
a   compact holomorphic curve embedded in a complex surface $M$ with a unitary normal bundle $N_C$, the above vertical linearization is one of main results in Ueda~\cite{Ue82} where $H^0(C,N_C\otimes S^\ell(N_C^*))=0$  for all $\ell>1$ follows from his small-divisor condition. This has been generalized by T. Koike in higher codimension  %
case under a strong assumption that $N_C$ is a direct sum of unitary line bundles~\cite{Ko15}.
The Ueda theory for codimension-one foliations has also been extended by Claudon-Loray-Pereira-Touzet~\cite{CLPT18} and Loray-Thom-Touzet~\ci{LTT17}.
   %
We remark that
\rt{vlin-int} via the flatness of $N_C$
 ensures the  existence of a ``horizontal" foliation~:
\begin{cor}\label{exist-of-foli}
Under assumptions of \rta{vlin-int}, there exists  a neighborhood of $C_n$ in $M_{n+d}$ that admits an
$n$-dimensional  smooth holomorphic
foliation having $C_n$ as a leaf.
\end{cor}

The following results can been understood in the context of  %
the Grauert formal principle for rigidity: If $(M, C)$ is formal equivalent to $(N_C,C)$, then they are holomorphically equivalent under suitable assumptions. We first consider the unitary case.
\begin{thm}\label{lin-unit}
	Let  $C_n$ be a compact submanifold of $M_{n+d}$
	$N_C$ is unitary. Let $\eta_0=1$ and
	\beq\label{def-eta-foliation-int}\nonumber
	\eta_m
	:=\max\left( K(N_C\otimes S^m(N_C^*)),K(T_C\otimes S^m(N_C^*)) \right)\max_{m_1+\cdots +m_p+s=m} \eta_{m_1}\cdots \eta_{m_p},
	\eeq
	where  the maximum is taken in $1\leq m_i<m$ for all $i$ and $s\in\Bbb N$.
	Assume that there are positive constants $L$, $L_0$ such that
	\eq{eta-cov}
	\eta_m\leq L_0L^m, \quad
	m=1,2\dots.
	\eeq If $T_CM$ splits and  $ H^1(\cL U,T_CM\otimes S^\ell (N_C^*))=0$ for
	all $\ell>1$ or more generally if
	a neighborhood of $C$ in $M$ is 
	linearizable 
	by a formal holomorphic mapping which is tangent to the identity
	and preserves the splitting of $T_CM$, then there exists a neighborhood of $C$ in $M$ which is holomorphically equivalent to a neighborhood of $C$ (i.e the $0$th section) in $N_C$
	In that case, we say that the embedding $C\hookrightarrow M$ is holomorphically linearizable.
\end{thm}
More generally, the following result treats two more general cases.
\begin{thm}\label{lin-int}
Let $C_n$ be  a compact submanifold of $M_{n+d}$. 
 Suppose that
\eq{bruno-cond-int}
\sum_{k\geq 1} \frac{\log D_*(2^{k+1})}
{2^{k}}<+\infty,
\eeq
where $
D_* (2^{k+1})
$
is defined by \rea{om2m}.
Suppose that either 
$H^0(C,TC\otimes S^{\ell}(N_C^{*}))=0$   for
all $\ell>1$, or 
$N_C$ is  flat. Assume further that either $T_CM$ splits and
$H^1(\cL U,T_CM\otimes S^\ell (N_C^*))=0$ for all $\ell>1$ or  $(M,C)$ and $(N_C,C)$ are   equivalent 
by a formal holomorphic mapping which is tangent to the identity
and preserves the splitting of $T_CM$. Then $(M,C)$ and $(N_C,C)$ are actually holomorphically equivalent.
\end{thm}
The previous results can be seen as a ``full linearization" results. \rt{lin-unit} is proved by using a majorant method while \rt{lin-int} is based on a Newton scheme. It is not clear how to compare the two "small divisors conditions" \re{eta-cov} and \re{bruno-cond-int} althought the counterparts in theory of dynamical systems are equivalent \cite{Br71}.
 The formal principle holds in the following cases:
  $(a)$  
 negative $N_C$ in the sense of Grauert, by results of Grauert~\ci{Gr62} and Hironaka-Rossi~\cite{HR64}. In Grauert's case, $C_n$ has a system of strictly pseudoconvex neighborhoods and consequently $C_n$ is the only compact  %
 $n$-submanifold  near $C_n$. In the same spirit, Savelev proved that all neighborhoods of embeddings of $\Bbb P^1$ in   complex  %
  surfaces with a unitary flat normal bundle  %
   are holomorphically equivalent~\cite{Sa82}. $(b)$ sufficiently positive $N_C$ and $\dim C>2$, by a result of Griffiths~\cite[Thm II (i)]{Gr66} showing that a neighborhood is determined by a
    finite-order neighborhood. In other words,  %
   under this condition the holomorphic classification of neighborhoods
  is ``finitely determined". $(c)$ $H^1(C, N_C)=0$ and  the case that for each $x\in C$ there is $x'\in C$ such that the fiber of $N_C$ at $x$ is generated by global sections of $N_C$ vanishing at $x'$, by a result of Hirschowitz (see \cite{Hi81} for more general results)\footnote{Recently, Jun-Muk Hwang proved instances of Hirschowitz's conjecture on the Formal Principle \cite{Hw19}.
  The authors thank Takeo Ohsawa for acknowledging this work.}.
$(d)$ $1$-positive $N_C$, by a result of Commichau-Grauert~\ci{CG81}.

We should remark that the above ``full linearization" result was obtained by Arnol'd when $C$ is an elliptic curve  and $M$ is 
a surface, where the vanishing of $H^0(X,T_CM\otimes S^\ell M)$ follows from the non vanishing of ``small divisors"~\cite{Ar76,Ar88}.  Ilyashenko and Pyartli~\ci{IP79} proved an
analogous result for special embeddings of 
the product flat tori under a strong assumption that $N_C$ is a direct sum of flat line bundles. We emphasize that in our linearization \rt{lin-int}, for general compact manifolds $C_n$, we impose the vanishing of $H^0(X,T_CM\otimes S^\ell M)$ for all integers $\ell\geq 2$ whereas there is no restriction on $H^0$ when $C$ is affine and  $N_C$ is flat.

As a simple consequence, we have the following

\begin{cor}
Under assumptions of \rta{lin-int}
on $C$ and $M$, any holomorphic section of a
holomorphic vector bundle $E$ over $C$ extends 
 to
a holomorphic section of  %
a holomorphic-vector-bundle extension of $E$ 
over a neighborhood of $C$ in $M$.
\end{cor}
\begin{cor}\label{tri-bundle} 
Let $C$ be 
a compact complex manifold.
Let $(M,C)$ be   equivalent to $(C\times \cc^d,C)$ 
by a formal holomorphic mapping which is tangent to the identity 
and preserves the splitting of $T_CM$. Suppose that the small-divisor condition in \rta{lin-int} is satisfied. Then $(M,C)$ is holomorphically equivalent to $(C\times \cc^d,C)$.
\end{cor}

\medskip

We now give an outline of the paper.

In section~\ref{sec:formal} we study the formal obstructions to the full linearization and vertical linearization 
 problems. The formal obstructions are known from    
  work of Nirenberg-Spencer~\cite{NS60}, Griffiths~\ci{Gr66}, Morrow-Rossi~\ci{MR78}, for the  
 the full linearization problem and by Ueda~\cite{Ue82} (see also Neeman \cite{Ne89}
and among others) for the
 vertical linearization problem. 
  The obstructions are described
  in  $H^1(C,E\otimes S^\ell N_C^*)$ for a natural vector bundle $E$ that is either  $T_CM$ or $N_C$.
   In this paper we emphasize the role of $H^0(C, T_CM\otimes S^\ell N_C^*)$. In  local dynamical systems, the elements in the analogous group appear as finite symmetries in the Ecalle-Voronin theory~\ci{AG05} and centralizers for the linearizations~\cite{GS16}.
The small divisors in local dynamics   
emerge 
in the form of the bounds $K(N_C\otimes S^\ell N_C^*)$ and $D(T_CM\otimes S^\ell N_C^*)$ in \rp{KnDn}. In   work of Arnol'd~\ci{Ar76} and Ueda~\ci{Ue82}, the  vanishing condition of the corresponding zero-th cohomology groups is not explicit; however it follows from their small-divisor conditions.

In section~\ref{sec:verlin}, we prove \rt{vlin-int} by using Ueda's majorization method~\ci{Ue82}. In our case the majorization relies on an important tool of the (modified) Fischer norm 
which is invariant under a unitary change of coordinates. The invariance allows us to overcome
 the main difficulty in our majorization proof to deal with the transition functions of $N_C^*$ when  
they are unitary, but not necessarily
 diagonal. The (modified) Fischer norms have also been useful in other convergence proofs \cite{Sh89,IL05, LS10}. In section~\ref{sec:majorlin}, we also extend the majorant method to the full linearization problem for the special case   
  where $N_C$ is unitary.  In section~\ref{sec:fulllin}, we obtain the full linearization in the general case by introducing a Newton scheme, i.e. a rapid convergence scheme as in Brjuno's work \cite{Br71}; see also~\ci{St00, Ru02}.
However, we must cope with the domains of transition functions which are not so regular. These domains, when carefully chosen, have nevertheless a disc structure. 
This allows us to obtain a proof by using sup-norm estimates.

 Finally, the paper contains an appendix which has interests in its own right. It has two results, namely the existence of the two bounds stated in \rp{KnDn} and the existence of nested coverings (see Definition~\ref{nested-cov}). The existence of bound $K(E)$ was
employed
  by Ueda~\ci{Ue82}  
 through the complete system of Kodaira-Spencer \cite{KS59} when
 $\dim C=1$ and
 $\codim_MC=1$.
 We will prove \rp{KnDn} by using some techniques developed by Donin~\ci{Do71}.
  Our proof also relies on a "quantified" version of Grauert-Remmert  finiteness theorem \cite{GR04}. The existence of bound $D(E'\otimes S^\ell E'')$ was proved by Donin~\ci{Do71} for the so-called
  ``normal" coverings. We have used nested coverings in the proof of  \rp{KnDn} as well as the convergence proof in \rt{lin-int}. 
  We believe that the methods and tools developed in this article
  will be
  useful for
   other kinds of
    problems.

\setcounter{thm}{0}\setcounter{equation}{0}
\section{
Full linearizations, horizontal foliations,
and vertical linearizations
}
\label{sec:formal}

In this section, we describe the  problem of equivalence of a neighborhood of a complex compact submanifold $C$ of $M$ with a neighborhood of the zero section in the normal bundle of $C$ in $M$ as  a ``full'' linearization problem of the transition functions of this neighborhood.
We also describe the existence of a holomorphic foliation of a neighborhood of $C$ having $C$ as a leave
as a consequence of a {\it vertical} linearization problem of the transition functions of this neighborhood.


We will first describe the formal coordinate changes in terms of cohomological groups of holomorphic sections of a 
suitable 
sequence of holomorphic vector bundles.


\subsection{Transition functions}
We recall basic facts on vector bundles,   
which we refer to \cite[Chap.~0, Sect.~5]{GH94}.

We first set up notation.  If a vector space $E$ has a basis $e=\{e_1,\dots, e_d\}$, then a vector $v$ in $E$ can be expressed as
$$
v=\xi^\mu e_\mu, \quad \xi=(\xi^1,\dots, \xi^d)^t.
$$
Here, we use the summation notation: $\xi^\mu e_\mu$ stands for $\sum_{\mu=1}^d\xi^\mu e_\mu$. The
$\xi^\mu$'s are the coordinates or components of $v$ in the basis $e$.

We recall that a holomorphic
vector bundle $\mathbf E$ over a complex manifold   
$X$  is defined by a projection $\pi\colon {\mathbf E}\to X$ and holomorphic trivializations   $\Psi_j\colon\pi^{-1}(D_j)\to D_j\times\cc^r$ such that each $\Psi_j\colon\pi^{-1}(D_j)\to D_j\times\cc^r$  is a biholomorphism, and $\Psi_j(\mathbf E_p)=\{p\}\times\cc^r$ for $\mathbf E_p:=\pi^{-1}(p)$. Furthermore $\{D_j\}$   is an open covering  of $X$ and the maps $\Psi_{kj}=\Psi_k\Psi_j^{-1}\colon D_k\cap D_j\times\cc^r\to D_k\cap D_j\times\cc^r$ satisfy
 \eq{Psig}
 \Psi_{kj}(p,\xi_j)=(p, g_{kj}(p)\xi_j)
 \eeq
 where $g_{kj}$ are transition matrices
 which are holomorphic and invertible.
 Thus for $\xi_k^\mu e_{k,\mu}=\xi_j^\mu e_{j,\mu}$, we have
 \ga
 \xi_{k}^\mu=g_{kj,\nu}^\mu \xi_j^\nu, \quad
 e_{j,\mu}=  g_{kj,\mu}^{\nu} e_{k,\nu},\\
 \xi_k=g_{kj}\xi_j, \quad e_k=(g_{kj}^{-1})^te_j.
 \label{exig}
 \end{gather}
  They satisfy the cocycle conditions,
\eq{cocy-C}
g_{kj}g_{jk}=I, \quad\text{on $D_k\cap D_j$}; \quad g_{ki}g_{ij}=g_{kj}, \quad
\text{on $D_k\cap D_j\cap D_i$}.
\eeq
 We also need to consider the dual bundle $E^*$.
 Let $e_j^*$ be the basis dual to $e_j$ so that $(e_{j,\mu}^*(e_{j,\nu}))$ is the identity matrix. Suppose $\zeta_j^{\mu}e^*_{j,\mu}=\zeta_k^{\mu}e^*_{k,\mu} \in E^*$. Corresponding to \re{exig}, we have
\eq{dualg}
e_k^*=g_{kj}e_j^*, \quad \zeta_k=(g_{kj}^{-1})^t\zeta_j.
\eeq

Let us also express   transition functions 
for various vector bundles in coordinate charts as above. Let $C_{n}$ be a compact complex manifold embedded in complex manifold $M_{n+d}$.    %
We cover a neighborhood of $C$ in $M$ by open sets $V_j$ so that we can choose  coordinate charts
 $(z_j,w_j)$ on $V_j$    %
  for $M$
 such that
$$
U_j:=C\cap V_j\colon w_j=0.
$$
Then ${\mathcal U}=\{U_i\}$ is a finite covering of $C$ by open sets on which the coordinate 
charts $z_i=(z_i^1,\dots, z_i^n)$ 
are defined.
Let
\beq\label{transitionC}
z_k=
\var_{kj}
(z_j)=
\var_k\var_j^{-1}(z_j)
\eeq
be the transition function of $C$ on $U_{kj}:=U_k\cap U_j$. It is a biholomorphic mapping from $
\var_j(U_{kj})$ onto $
\var_k(U_{kj})$ in $\cc^n$.
Then $TC$ has a basis
$$
e_{j,\alpha}:=\frac{\pd}{\pd z_j^\all}, \quad 1\leq\all\leq n
$$
 over $U_j$ and its transition  
matrices
  $s_{kj}$ have the form
\eq{xskj}
s_{kj,\beta}^\all
(z_j):=
\f{\pd z_k^\all}{\pd z_j^\beta}\bigg\vert_{U_j\cap U_k}.
\eeq
 Thus
 for $\eta_{k}^\all\f{\pd}{\pd z_k^\all}=\eta_{j}^\all\f{\pd}{\pd z_j^\all}$ on $U_j\cap U_k$, we have $\eta_k=s_{kj}(z_j)\eta_j$.
As to the normal bundle $N_C$,  its transition  
matrices
  $
t_{kj,\nu}^\mu
(z_j):=
\f{\pd w_k^\mu}{\pd w_j^\nu} |_{U_j\cap U_k}$ on $U_j\cap U_k$
are for the basis
$$
f_{j,\mu}:=\frac{\pd}{\pd w_j^\mu} \mod TC, \quad 1\leq\mu\leq d.
$$
Thus for $\xi_{k}^\mu f_{j,\mu}=\xi_{j}^\mu f_{k,\mu}$, we have $\xi_k=t_{kj}(z_j)\xi_j$.
With notation \re{Psig}, the transition  
matrices
 of $TM|_{C}$ are then of the form
$$
g_{kj}:=\left(\begin{matrix}s_{kj}& l_{kj}\\0&
t_{kj}\end{matrix}\right)
(z_j)\quad \text{on } U_j\cap U_k
$$
for some $n\times d$
matrices  $l_{jk}$.    %
 Note that
$
 \DD{w_j}{z_k} |_C=0.
$

Throughout the paper,  $
\tau_{kj}(z_j)$ are the transition  
matrices
  of $N_C^*$ for the base $dw_j$. Note that
$$
\tau_{kj}=(t_{kj}^{-1})^{t}.
$$
 More specifically,    if $w_{j,\mu}^{*}:=dw_j^\mu|_{U_j}$ and $\zeta_j^{\mu}w_{j,\mu}^{*}=\zeta_k^{\mu}w_{k,\mu}^{*}$, then \re{dualg} becomes
\eq{dualg+}
\zeta_k^*=(t_{kj}^{-1}(z_j))^t\zeta_j^*, \quad
  w_k^*=t_{kj}(z_j)w_j^*.
  \eeq
We remark that
the cocycle conditions \re{cocy-C} for $N_C$ now takes the form
\eq{}
t_{kj}(z_j)t_{jk}(z_k)=\text{Id}\  \text{on $U_j\cap U_k$},\quad
t_{kj}(z_j)t_{j\ell}(z_\ell)=t_{k\ell}(z_\ell)\  \text{on $U_j\cap U_k\cap U_\ell$}.
\label{vb-cocycle-}
\eeq

We say that $TM$ {\it splits} on $C$, if there is a (non-canonical) decomposition
\eq{tNc}
TM|_C=TC\oplus \tilde N_C, \quad\tilde N_C\cong N_C.
\eeq
Equivalently, there exists a system of coordinate charts such that on $C$,  the transitions   
 matrices of $TM|_C$ are of the form
$$
g_{kj}=\left(\begin{matrix}s_{kj}& 0\\0&
t_{kj}\end{matrix}\right)
(z_j)
\quad \text{on } U_j\cap U_k.
$$
In other words,
$
\left.\DD{z_j}{w_k}\right|_C=0.
$
Throughout the paper, we assume that $TM$ splits on $C$ and we fix a splitting \re{tNc}.
Then the change of bases of the normal bundle $N_C$ has a simple form
\beq\label{wktk-}\nonumber
z_k=
\var_{kj}(z_j),\quad \DD{}{w_k^\nu}=
t_{jk,\nu}^\mu(z_k)\DD{}{w_j^\mu}, \quad \text{on $U_j\cap U_k$}.
\eeq

In summary,  for a neighborhood of the embedded manifold $C$ in $M$ with splitting $T_CM$, we  can find a covering $\mathcal V=\{V_i\}$, with $
\Phi_j(V_j)=\tilde U_i\times \tilde W_i$, by open sets on $M$ and coordinates $(z_i,w_i)$ defined on $V_i$. We assume that $U_j:=C\cap V_i$ is defined by $\{w_i=0\}$. A neighborhood of $C$ will then be described by transition functions on $V_{kj}$ of the form
\eq{transitionN-}
\Phi_{kj}\colon
\begin{array}{rcl}
z_k &= & \Phi^h_{kj}(z_j,w_j):=
\var_{kj}(z_j)+\phi^h_{kj}(z_j,w_j),\vspace{.75ex}
\\
w_k &= &\Phi^v_{kj}(z_j,w_j):=
 t_{kj}(z_j)w_j+\phi^v_{kj}(z_j,w_j).
\end{array}
\eeq
Here, $\phi^h_{kj}$ (resp. $\phi^v_{kj}$) are holomorphic functions of vanishing order $\geq 2$ along $w_j=0$:
\eq{vord2}
\phi_{kj}^h(z_j,w_j)=O(|w_j|^2), \quad \phi_{kj}^v(z_j,w_j)=O(|w_j|^2).
\eeq
 That $\phi^h_{kj}$ vanishes at order $\geq 2$ follows from the fact that $TM|_C$ splits as $TC\oplus N_C$ (see above and
 \cite[proposition 2.9]{MR78}).
Define
\beq\nonumber
N_{kj}(h_j,v_j):=(
\var_{kj}(z_j),t_{kj}(h_j)v_j). 
\eeq

Our goals are to apply changes of coordinates to simplify $\phi_{kj}^h, \phi_{kj}^v$, or one of them, according to the problem we study.

\subsection{
The equivalence of transition functions}The germ of  neighborhood of an embedded manifold is well-defined. For the 
formal normalization, we need to introduce (semi) formal charts  and formal neighborhoods of an embedded manifold in a (semi)
formal manifold.

 \begin{defn}
 We call $\hat M$ an (admissible and splitting)   formal neighborhood of $C$ if   there are holomorphic coordinate charts $
 \var_j$ on $U_j$ where $\{U_j\}$ is a covering of $C$ and there are formal power series
$$
(z_j,w_j)=\hat\Phi_j(p,w):=(
\var_j(p),
t_j(p)w)+\sum_{|Q|\geq2}\Phi_{j,Q}(p)w^Q,
$$
where $\Phi_{j,Q}$ are holomorphic functions in $U_j$ and each $t_j$ is an invertible holomorphic  $d\times d$ matrix on $U_j$.  Note that the formal transition functions $\hat\Phi_{kj}=\hat \Phi_k\hat\Phi_j^{-1}$ have the form
\eq{}\nonumber
\hat\Phi_{kj}(z_j,w_j)=(
\var_{kj}(z_j),t_{kj}(z_j)w_j)+\sum_{|Q|>1}
\hat \Phi_{kj,Q}(z_j)w_j^Q, \quad z_j\in
\var_j(U_j\cap U_k).
\eeq
\bpp
\item
When all $\Phi_j$ are
holomorphic, the formal neighborhood $\hat M$ is called
the germ of a  (holomorphic)
 neighborhood of $C$.
\item $\hat M$ is called a 
 linear neighborhood of $C$ if additionally
\eq{initial000}   
\hat\Phi_{kj}(z_j,v_j)=(
\var_{kj}(z_j),t_{kj}(z_j)v_j)
\eeq
and each
 $t_{kj}$ is
 an invertible holomorphic matrix in $U_k\cap U_j$. The terminology is meaningful since the $\hat\Phi_{kj}$ can be realized as the transition functions of a holomorphic vector bundle over $C$,  namely the normal bundle of $C$ in $M$. 
\epp
\end{defn}
We are mainly interested in the classification of a neighborhood of $C$ for a given $C$.   Therefore, it is reasonable to assume that the local trivialization of $C$ are fixed. In other words, $
\var_{kj}$ are fixed and we will only consider mappings sending a neighborhood of $C$ into another neighborhood of $C$ that fix $C$ pointwise.


\begin{defn}
We shall say that $N_C$ is a flat (resp. unitary flat), if we can find constant (resp. with values in group of unitary matrices $U_d$) transition functions in a possibly refined covering.
If $T_CM:=(TM)|_C$ is {\it holomorphically flat}, or {\it flat}, i.e. in some coordinates both   transition functions  $N_C$ and  $TC$ are constant matrices, then by \re{xskj}
\eq{}\nonumber
\var_{kj}(z_j)=s_{kj}z_j+c_{kj}
\eeq
where $s_{kj}$ are constant matrices and $c_{kj}$ are constant vectors. 
Then, the transition functions of a neighborhood of the zero section of the normal bundle, $\hat \Phi_{kj}$ as defined in \re{initial000} 
read
$$
A_{kj}(z_j,w_j):=(s_{kj}z_j+c_{kj},t_{kj}w_j).
$$
\end{defn}
We will use the following notation: When $N_C$ is flat, we write its transition  
matrices
  $t_{kj}(z_j)$ as $t_{kj}$, indicating that they are independent of $z_j$.
\begin{defn}
We shall say that a change of coordinates $\{F_j\}$ preserves the germ of
a neighborhood of  the zero section of $N_C$ with transition maps $\{N_{kj}\}$ if
each $F_j$ is biholomorphic and fixes $v_j=0$ pointwise and
 $F_kN_{kj}=N_{kj}F_j$, in which case we say that $\{F_j\}$ preserves $\{N_{kj}\}$ for simplicity.
\end{defn}




We further observe the following. 
\le{dfj-}Let $M$, $\hat M$ be two $($admissible$)$ neighborhoods of $C$, of which coordinate charts are $\{\Phi_j\},\{\hat\Phi_j\}$, respectively. Let $\Phi_{kj}=\Phi_k\Phi_j^{-1}$ and $\hat\Phi_{kj}=\hat\Phi_k\hat\Phi_j^{-1}$.
\bpp
\item  There is a biholomorphic mapping $F\colon M\to\hat M$, defined near $C$ and fixing $C$, if and only if there are biholomorphic mappings $F_j$ satisfying
 \ga
\label{Fkhp}
F_k\hat\Phi_{kj}(z_j,w_j)=\Phi_{kj}F_j(z_j,w_j), \quad   F_j(z_j,0)=(z_j,0).
\end{gather}
\item  If  $F_j$ satisfies \rea{Fkhp},
then
\gan
F_j(z_j,w_j)=LF_j(z_j,w_j)
+O(|w_j|^2), \quad LF_j=(z_j+s_j(z_j)w_j,u_j(z_j)w_j),\\
s_k(
\var_{kj}(z_j))t_{kj}(z_j)=D
\var_{kj}(z_j)s_j(z_j),\\
u_k(
\var_{kj}(z_j))t_{kj}(z_j)=t_{kj}(z_j)u_j(z_j).
\end{gather*}
Assume further that $F$ preserves the splitting. Then $s_j=0$.
\item
Let $TC$ and $N_C$ be flat and let $F_j$ be $($semi$)$ formal biholomorphism fixing $C$ pointwise. Suppose that $F_k^{-1}\Phi_{kj}F_j=N_{kj}+O(|v|_j^2)$. Then $\{LF_j\}$ preserves $\{N_{kj}\}$, i.e.
$LF_kN_{kj}(LF_j)^{-1}=N_{kj}$,  where
$$
F_{j}(h_j,v_j)=LF_j(h_j,v_j)+O(|v_j|^2), \quad LF_j(h_j,v_j)=(h_j+s_j(h_j)v_j,u_j(h_j)v_j).
$$
\epp
\ele
\begin{proof} The points $(a),(b)$ can be verified easily. For $(c)$,  let us expand $F_k\Phi_{kj}(h_j,v_j)=N_{kj}\circ F_j(h_j,v_j)+O(|v_j|^2)$ and compare the constant and linear terms in $v_j$. We obtain
\gan
\var_{kj}(h_j)+s_j(\var_{kj}(h_j))t_{k 
j}v_j=\var_{kj}(h_j+s_j(h_j)v_j)+O(|v_j|^2),\\
u_k(\var_{kj}(h_j))t_{kj}v_j=t_{kj}u_j(h_j)v_j+O(|v_j|^2).
\end{gather*}
Here we have used the assumption that $t_{kj}$ are constant. Since $\var_{kj}$ are affine, the two identities still hold
if we drop $O(|v_j|^2)$ from them.
This shows that $LF_kN_{kj}=N_{kj}LF_j$, again using the fact that $t_{kj}$ are constant and $\var_{kj}$ are affine.
\end{proof}

Finally, we mention that we will choose the  atlas of  $C$ so that each $\var_j$ is a biholomorphism  from $U_j$ onto the unit polydisc $\Del_n$ in $\cc^n$ 
and  from a neighborhood $\tilde U_j$ of $\ov {U_j}$ onto another larger polydisc. When $C$ is embedded in a complex manifold $M$, we can extend $\var_j$ to $V_j$ to get a coordinate chart $\Phi_j$ on $V_j$ such that $\Phi_j$ maps $V_j$ onto $U_j\times\Del_\del^d$. This can be achieved since any holomorphic vector bundle over $\tilde U_j$ is holomorphically trivial. Thus $N_C|_{U_j}$ splits. Consequently, we can use a flow box of holomorphic normal vector fields to construct the required $\Phi_j$. Therefore, if   $C$ is embedded into another complex manifold $\tilde M$, we will choose the atlas of a neighborhood of $C$ in $\tilde M$ such that the restriction of the chart on $U_j$ agrees with $\var_j$.

Therefore, we introduce the following.
\begin{defn}We say that a formal  neighborhood  $\{\Phi_{kj}\}$ of $C$ is
  equivalent to a neighborhood $\{\hat\Phi_{kj}\}$ of $C$ in $M$ by a formal mapping  $F$
that is tangent to the identity and preserves the splitting of $T_CM$, if there are formal maps $F_j(z_j)=(z_j,w_j)+\sum_{|Q|>1} F_{j,Q}(z_j)w_j^Q$ such that $F_{j,Q}(z_j)$ are holomorphic functions in $  U_j$ and as power series in $w_j$
$$
F_k\hat\Phi_{kj}(z_j,w_j)=\Phi_{kj}F_j(z_j,w_j).
$$
 We take $F=\hat\Phi_j^{-1}F_j\Phi_j$, which is well-defined,
 when $\Phi_{kj}=\Phi_k\Phi_j^{-1}$ and $\hat\Phi_{kj}=\hat\Phi_k\hat\Phi_j^{-1}$.
\end{defn}

\subsection{
The full Linearization of a neighborhood} In this case, our goal is to seek new coordinates $(h_k,v_k)$ so that all $\phi_{kj}^h,\phi_{kj}^v$ are $0$.

Let us consider a change of coordinates in a neighborhood of $C$  by modifying the old coordinate charts $(z_k,w_k)$ via $F_k$.  We write it as
\eq{change}
F_k: \begin{array}{rcl}
z_k &= & F^h_{k}(h_k,v_k):=h_k+f^h_{k}(h_k,v_k),\vspace{.75ex}\nonumber\\
w_k &= & F^v_{k}(h_k,v_k):=v_k+f^v_{k}(h_k,v_k).
\end{array}
\eeq
Here, $f^h_{k}(h_k,v_k)$ and $f^v_{k}(h_k,v_k)$ are holomorphic functions vanishing at order $\geq 2$ at $v_k=0$. In particular, $C$ is pointwise fixed by the change as $z_k=h_k$ on $C$ (i.e. for $v_k=0$). We require that the inverse
of $F_k$ is defined in a possibly smaller open sets $\hat V_k\subset
\var_k(U_k)$
such that the union of $
\Phi_k^{-1}(\hat V_k)$ remains a neighborhood of $C$ in $M$.

We recall that the cocyle condition \re{vb-cocycle-} on the transition  
matrices
  $t_{kj}$ has the form
\begin{eqnarray}
t_{kj}(z_j)t_{jk}(
\var_{kj}(z_j))&=&\text{Id},\nonumber\\
t_{kj}(
\var_{j\ell}(z_\ell))t_{j\ell   }(z_\ell)&=&t_{k\ell   }(z_\ell).\label{vb-cocycle}
\end{eqnarray}

Let us assume that the ({\it a priori} formal) change of coordinates $\re{change}$, maps a neighborhood $C$ to a neighborhood of the zero section in the normal bundle. This means that, in these new coordinates, we have
\eq{transitionNC}
N_{kj}:=F_k^{-1}\Phi_{kj}F_j:
\begin{array}{rcl}
h_k &= &
\var_{kj}(h_j),\nonumber\\
v_k &= &  t_{kj}(z_j)v_j.
\end{array}
\eeq

Let us write down the above  ``conjugacy equations''.
We first consider the horizontal equation of
$$
F_kN_{kj}=\Phi_{kj}F_j.
$$
 On the left side of the equation, we have
$$
z_k=h_k+f_k^h(h_k,v_k)=
\var_{kj}(h_j)+f_k^h(
\var_{kj}(h_j),t_{kj}(h_j)v_j).
$$
On the other side, we have
$$
z_k=
\var_{kj}(h_j+f_j^h(h_j,v_j))+\phi^h_{kj}(h_j+f_j^h,v_j+f_j^v).
$$
Let us define the {\it horizontal cohomological operator} to be
\beq
{\mathcal L}_{kj}^h(f^h_j):=f_k^h(
\var_{kj}(h_j),t_{kj}(h_j)v_j)  -s_{kj}(h_j)f_j^h(h_j,v_j).\label{h-cohomo-op}
\eeq
Recall that $s_{kj}(h_j)=D
\var_{kj}(h_j)$ is the Jacobian matrix of $
\var_{kj}$.
Hence, we can write the previous horizontal equation as
\begin{eqnarray}
{\mathcal L}_{kj}^h(f^h_j)&=& \phi^h_{kj}(h_j+f_j^h,v_j+f_j^v)\label{h-conjug}\\
&&+\,
\var_{kj}(h_j+f_j^h(h_j,v_j))-
\var_{kj}(h_j)-D
\var_{kj}(h_j)f_j^h(h_j,v_j).\nonumber
\end{eqnarray}

Let us consider the vertical equation. We have, on one side of the equation,
$$
w_k = v_k+f^v_{k}(h_k,v_k)= t_{kj}(h_j)v_j+f_k^v(
\var_{kj}(h_j), t_{kj}(h_j)v_j).
$$
On the other side, we have
$$
w_k=t_{kj}(h_j+f_j^h)(v_j+f_j^v)+\phi^v_{kj}(h_j+f_j^h,v_j+f_j^v).
$$
Let us define the {\it vertical cohomological operator} to be
\beq
{\mathcal L}_{kj}^v(f^v_j):=f_k^v(
\var_{kj}(h_j),t_{kj}(h_j)v_j)-t_{kj}(h_j)f_j^v.\label{v-cohomo-op}
\eeq
Hence, we can write the previous vertical equation as
\begin{eqnarray}
{\mathcal L}_{kj}^v(f^v_j)&=& \phi^v_{kj}(h_j+f_j^h,v_j+f_j^v)\label{v-conjug}\\
&&+\left(t_{kj}(h_j+f^h_{j}(h_j,v_j))-t_{kj}(h_j)\right)f_j^v\nonumber\\
 &&+\left(t_{kj}(h_j+f^h_j(h_j,v_j))-t_{kj}(h_j)\right)v_j.\nonumber
\end{eqnarray}

\subsection{
Horizontal foliations and vertical trivializations}
Let us assume that there exists a non singular holomorphic foliation having $C$ as a leaf. We seek holomorphic functions $f_j=(f_{j,1},\dots, f_{j,d})$ defined in a neighborhood $V_j$ of $U_j$ such that $f_j=0$ on $U_j$ and $df_{j,1}\wedge\cdots\wedge df_{j,d}\neq0$. Then, we may use $(h_j,v_j)=(z_j,f_j(z_j,w_j))$ as a coordinate mapping on $V_j$, which changes variables in vertical components. We then prove that in these new coordinates, the transition functions of a neighborhood of $C$ are of the form $\hat\Phi_{kj}=(\hat\Phi_{kj}^h,\hat\Phi_{kj}^v)$
such that $\Phi_{kj}^v$ are independent of $h_j$. We remark that $N_C$ must be flat if  a horizontal foliation exists.
\pr{2-prob} Assume that there is smooth holomorphic
horizontal foliation defined in a neighborhood $V$ of $C$ in $M$.  By a refinement of $U_j$,
then there exists
 a change of variables of the form
$$
z_k = h_k\quad w_k=s(h_j)v_j+O(|v_j|^2)
$$
so that in the new variables, we have
\begin{eqnarray}
h_k &= &
\var_{kj}(h_j)+\phi
_{kj}^h(h_j,v_j),\nonumber\\
v_k &= & \tilde t_{kj}v_j+\sum_{|Q|>1}c_{
kj,Q}v_j^Q,
\label{retrC}\nonumber
\end{eqnarray}
where $\tilde t_{kj}, c_{kj,Q}$ are constants.
\epr
\begin{proof}By a refinement, we may assume that the foliation on $V_j$ is given
$W_{j}(h_j,v_j)=cst$ by holomorphic functions $W_j=(W_{j,1},\dots, W_{j,d})$ such that $W_j=0$ on $U_j$ and $dW
_{j,1}\wedge\cdots\wedge dW_{j,d}\neq0$. We have $W_k=\tilde\Phi
^v_{kj} W_j$, where $\tilde\Phi^v_{kj}$ is a biholomorphism of $(\Bbb C^d,0)$ with $\tilde\Phi_{kj}^{ v}(0)=0$. Then $\tilde W_j=(z_j,W_j)$ is a biholomorphism defined on $V_j$
 and fixing $C\cap V_j$ pointwise, by shrinking $V_j$ if necessary
in the vertical direction. Since $\tilde W_j$ is invertible, we can define $\tilde\Phi^h_{kj}=z_k\tilde W_j^{-1}$
Then we have
$
\tilde\Phi_{kj}^h\tilde W_j=z_k.
$
Therefore,
$$\tilde W_k\tilde W_j^{-1}(h_j,v_j)=(\tilde \Phi_{kj}^h(h_j,v_j),\tilde\Phi_{kj}^v(v_j)). $$
Set $F_j=\Phi_j\tilde W_j^{-1}$. We have $F_j^h(h_j,v_j)=h_j$. We now get
$$
F_k^{-1}\Phi_{k}\Phi_j^{-1}F_j=\tilde W_k\tilde W_j^{-1}=\tilde\Phi_{kj}.\qedhere
$$
\end{proof}
 In this paper, we will approach the horizontal foliation problem via the following vertical linearization when $N_C$ is unitary.

\subsection{
The vertical linearization}

Here we seek new coordinates $(h_j,v_j)$  from $(z_j,w_j)$ such that the vertical component of the new $\Phi_{kj}$ agrees with
the vertical component of $N_{kj}$. In \rl{linkv} we will show that if such
 {\it formal}
 coordinates  exist, then the vertical linearization can be achieved
by changing vertical coordinates only, i.e. a coordinate change of the form
$$
w_k =  F^v_{k}(h_k,v_k):=v_k+f^v_{k}(h_k,v_k),\quad z_k=h_k.
$$
For the vertical linearization,
we  only need to consider the {\it vertical part} of  transition functions
so that in the new variables, we have
\begin{eqnarray}
h_k &= & \hat \Phi_{kj}^h(h_j,v_j):=
\var_{kj}(h_j)+\hat \phi_{kj}^h(h_j,v_j)\nonumber\\
v_k &= & t_{kj}(h_j)v_j.
\label{extNC}\nonumber
\end{eqnarray}
Here, $\hat \phi_{kj}^h(h_j,v_j)$ vanishes up to order $2$ at $v_j=0$.  The vertical equation reads
$$
t_{kj}(h_j)(v_j+f_j^v)+\phi^v_{kj}(h_j,v_j+f_j^v)=w_k=t_{kj}(h_j)v_j+f_k^v(\hat\Phi^h_{kj}(h_j,v_j), t_{kj}(h_j)v_j).
$$
Using the previous notation, we finally obtain the following ``conjugacy equations''
\beq
{\mathcal L}_{kj}^v(f^v_j)= \phi^v_{kj}(h_j,v_j+f_j^v)-\left(f_k^v(\hat\Phi^h_{kj}(h_j,v_j), t_{kj}(h_j)v_j)-f_k^v(
\var_{kj}(h_j), t_{kj}(h_j)v_j)\right).\label{v-ext}
\eeq

Having determined the coordinate change, let us find the horizontal component $\hat\phi_{kj}^h$ from the horizontal equation
$$
\var_{kj}(h_j)+\phi^h_{kj}(h_j,v_j+f_j^v)=z_k=\hat\Phi_{kj}^h(h_j,v_j)=
\var_{kj}(h_j)+\hat\phi^h_{kj}(h_j,v_j).
$$
We get
\beq
\hat\phi^h_{kj}(h_j,v_j) 
= \phi^h_{kj}(h_j,v_j+f_j^v).\label{h-ext}
\eeq

\subsection{Coboundary operators in symmetric powers and  coordinates}\label{section-cocycles}

In this subsection, we establish the connections
between coordinate changes and
  formal obstructions
  to  the full linearization 
and vertical linearization
 via cohomological groups. In local dynamics, the resonant terms play an important role in the construction of normal forms at least at the formal level, while non-resonant terms play another important role in coordinate changes. In all
problems,
 obstructions are described via  the
  first cohomological groups, while the
coordinate changes are described via solutions to the cohomological equations of first order
approximation.

 Let   $E'$ be a vector bundle of rank $\tau$ over $C$. Let
 $\cL U = \{U_i\}$
 be a covering of $C$ as above. Let $e_j:=\{e_{j,1},\ldots, e_{j,\tau}\}$ be a basis over $U_j$ and let $\xi_j:=(\xi_{j}^{1},\ldots, \xi_{j}^{\tau})^{t}$ be coordinates in $e_j$.
  Let $s_{kj}(z_j)$ 
   be the transition matrices of $E'$ over $U_{k}\cap U_{j}$.
   Using notation in \re{exig}, we have
\ga
\label{use-it}
\xi_k^\all=s^\all_{kj,\beta}(z_j)\xi_j^\beta, \quad e_{k;\all}=s_{jk,\all}^\beta(z_k) e_{j,\beta},\\
z_k=
\var_{kj}(z_j), \quad \xi_k=s_{kj}(z_j)\xi_j, \quad e_k=(s^{-1}_{kj}(z_j))^te_j,
\end{gather}
where $
\var_{kj}$ are the transition functions of $C$. For $N_C^*$,
by \re{dualg+} we have
$$
\zeta_k=(t_{kj}^{-1})^t(z_j)\zeta_j, \quad w_k^*=t_{kj}(z_j)w_j^*,\quad z_k=
\var_{kj}(z_j).
$$
The following fact is well-known. We provide a proof for the reader's convenience.
Let us first introduce
 \eq{def-tilde-f}
  \tilde f_{i_0\cdots i_q}^\la (z_{i_q},\zeta_{i_q}):=\sum_{|Q|=L}  f_{i_0\cdots i_q;Q}^\la (z_{i_q}) \zeta_{i_q}^Q,
 \eeq
for a cochain $\{f_I\}\in C^q(\{U_j\},\cL O(E\otimes S^L(N_C^*)))$
 given by
 \ga\label{def-tilde-f+}
 f_{i_0\cdots i_q}(p)=\sum_{\la=1}^{ \tau} \sum_{|Q|=L} f_{{i_0\cdots i_q}; Q}^\la (z_{i_q}(p)) e_{i_0,\la }(p)\otimes (w_{i_q}^*(p))^Q,
 \end{gather}
 where each $f_{i_0\dots i_q;Q}^\la$ is a holomorphic function on $
\var_{i_q}(U_{i_0\cdots i_q})$,
 and  $U_{i_0\cdots i_q}$ denotes as usual $U_{i_0}\cap\cdots\cap U_{i_q}$. Here we have chosen a representation of cochains in bases that arise from the linearized equations for the problems described above.

Let $f_{i_0\cdots \hat i_\ell\cdots i_{q+1}}$ denote $f_{i_0\cdots i_{\ell-1}i_{\ell+1}\cdots i_{q+1}}$.
Then   $(\del f)_{i_0\cdots i_{q+1}}=\sum(-1)^\ell f_{i_0\cdots \hat i_\ell\cdots i_{q+1}}$  becomes
\aln
(\del f)_{i_0\cdots i_{q+1}}&=\sum_{\ell=1}^q(-1)^\ell  \sum_{\la=1}^{\tau} \sum_{|Q|=L} f_{{i_0\cdots \hat i_\ell\cdots  i_{q+1}}; Q}^\la (z_{i_{q+1}}(p)) e_{i_0,\la }(p)\otimes (w_{i_{q+1}}^*(p))^Q\\
&\quad + \sum_{\la=1}^{
\tau} \sum_{|Q|=L} f_{{i_1\cdots   i_{q+1}}; Q}^\la (z_{i_{q+1}}(p)) e_{i_1,\la }(p)\otimes (w_{i_{q+1}}^*(p))^Q\\
&\quad -(-1)^{q} \sum_{\la=1}^{
\tau}\sum_{|Q|=L} f_{{i_0\cdots   i_{q}}; Q}^\la (z_{i_{q}}(p)) e_{i_0,\la }(p)\otimes (w_{i_{q}}^*(p))^Q
\\
&=:\sum_{\la=1}^{
\tau}
\sum_{|Q|=L} g^\la_{i_0\cdots i_{q+1}}(z_{q+1})e_{i_0,\la }(p)\otimes (w_{i_{q+1}}^*(p))^Q.
\end{align*}
By \re{use-it}, we have $e_{i_1,\la}=s_{i_0i_1,\la}^\mu e_{i_0,\mu}$.
In notation \re{def-tilde-f}, we can express
\al{}\nonumber
\tilde g^\la_{i_0\cdots i_{q+1}}(z_{i_{q+1}},\zeta_{i_{q+1}})&=
\sum_{\ell=1}^q (-1)^\ell \tilde f^\la_{i_0\cdots \hat i_\ell
\cdots i_{q+1}}(z_{i_{q+1}},\zeta_{i_{q+1}})\\
&\quad +s_{i_0i_1,\mu}^\la(
\var_{i_1i_{q+1}}(z_{q+1})) \tilde f^\mu_{{i_1\cdots   i_{q+1}}}(z_{i_{q+1}},\zeta_{i_{q+1}})
\nonumber
\\
&\quad -(-1)^{q}f_{{i_0\cdots   i_{q}}}^\la (
\var_{i_{q}i_{q+1}}(z_{i_{q+1}}), t_{i_{q}i_{q+1}}(z_{i_{q+1}})\zeta_{i_{q+1}})). \nonumber
\end{align}


The above computation especially gives us the following formulae for $0$ and $1$-cochains.
\le{cc} Let $\{U_j\}$ be an open covering of $C$.
Let $t_{kj}$ be the transition  
matrices
  for $N_C$ with respect to basis $w_j$ and let $s_{kj}$ be the transitions functions of $E$ with respect to base $e_j$.
 Let
 \gan
 f_{ij}(p)=
 \sum_{\la=1}^d \sum_{|Q|=L} f_{ij; Q}^\la (z_j(p)) e_{i,\la }(p)\otimes (w_j^*(p))^Q, \quad \tilde f_{ij}^\la (z_j,\zeta_j):=\sum_{|Q|=L}  f_{ij;Q}^\la (z_j) \zeta_j^Q,\\
 u_{j}(p)=\sum_{\la=1}^d
 \sum_{|Q|=L} u_{j, Q}^\la (z_j(p)) e_{j,\la }(p)\otimes (w_j^*(p))^Q, \quad \tilde u_{j}^\la (z_j,\zeta_j):=\sum_{|Q|=L}  u_{j;Q}^\la (z_j) \zeta_j^Q.
\end{gather*}
The following hold $:$
\bpp
\item
 $ f:=\{f_{ij}\}\in Z^{1}(\cL U,\cL O(E\otimes S^L(N_C^*)))$ if and only if
\ga\label{cochain}\nonumber
 \tilde f_{ij}^\la (
\var_{jk}(z_k),
 t_{jk}(z_k)\zeta_k
 ) -\tilde f_{ik}^\la (z_k,\zeta_k)+s_{ij,\ell }^\la (z_j) \tilde f_{jk}^\la (z_k,\zeta_k)
 =0.
\end{gather}
\item 
$u:=\{u_j\}$ solves the first order cohomological equation $\del u=f$ if and only if
\eq{du=f}\nonumber
 s_{ij,\ell }^\la(z_j)
\tilde u_{j}^\ell (z_j,\zeta_j)-
  \tilde u_{i}^\la (
\var_{ij}(z_j),
  t_{ij}(z_j)\zeta_j)=\tilde f_{ij}^\la (z_j,\zeta_j).
\eeq
\epp
\ele

We notice that according to \re{h-cohomo-op} and \re{v-cohomo-op}, we have
$$
-\cL L(f)=-(\cL L^h(f^h), \cL L^v(f^v))= \del(f):=(\del^h(f^h),\del^v(f^v)).
$$

\subsection{Formal 
obstructions   in cohomology groups}

Recall that
\beq
N_{kj}(h_j,v_j):=(
\var_{kj}(z_j),t_{kj}(h_j)v_j).\label{initial}
\eeq
Let us denote the properties depending on an order
$m\geq
1$~:\\

\noindent
$(L_m(\cL U))$ : the neighborhood of $C$
matches
 the neighborhood of zero section of the normal bundle up to order $m$.
 \\
$(V_m(\cL U))$ :
the vertical   components of the transition functions of  neighborhoods of $C$ in $M$ and in $N_C$
match  up to order $m$.
 \\
%

That embedding of $C$ has property $(L_m)$ (resp. $(V_m)$)
means that the order along
 $v_j=0$ of $(\phi_{kj}^h(h_j,v_j), \phi_{kj}^v(h_j,v_j))$ (resp. $\phi_{kj}^v(h_j,v_j)$)
 as defined in
 \re{transitionN-} is $\geq m+1$.

\begin{defn}
We shall say that $N_C$ is a flat (resp. unitary flat), if we can find constant (resp. with values in group of unitary matrices $U_d$) transition functions in a possibly refined covering.
\end{defn}
We will use the following notation: When $N_C$ is flat, we write its transition  
matrices
  $t_{kj}(z_j)$ as $t_{kj}$, indicating that they are independent of $z_j$.
\begin{defn}
We shall say that a change of coordinates $\{F_j\}$ preserves the germ of
a neighborhood of  the zero section of $N_C$ with transition maps $\{N_{kj}\}$ if $F_kN_{kj}=N_{kj}F_j$, in which case we says that $\{F_j\}$ preserves $\{N_{kj}\}$ for simplicity.
\end{defn}

\le{verifycochain} Let the transition functions $\Phi_{kj}$ of a neighborhood of $C$
be given by \rea{transitionN-}-\rea{vord2}.
\bpp
\item
Assume that $C$ satisfies $L_m$. Then the horizontal and vertical components
satisfy
\aln
&[\phi_{kj}^h]^\ell\in Z^1(\cL U,TC\otimes S^\ell (N_C^*)), \quad \text{if $ m<\ell\leq 2m$};\\
&[\phi_{kj}^v]^\ell\in Z^1(\cL U,N_C\otimes S^\ell (N_C^*)),
\quad \text{if $ \ell=m+1$}.
\end{align*}
If    $N_C$ is flat, then the vertical component of $\Phi_{kj}$ further satisfies
\al\label{Dtphi}\nonumber
&
[\phi_{kj}^v]^\ell\in Z^1(\cL U,N_C\otimes S^\ell (N_C^*)),\quad m+1<\ell\leq 2m.
\end{align}
\item Let  $C$ satisfy $V_m$. Assume that $N_C$ is flat.  Then
\eq{pvZ1}
[\phi_{kj}^v]^\ell\in Z^1(\cL U,N_C\otimes S^\ell (N_C^*)), \quad  \ell=m+1.
\eeq
\epp
\ele
\begin{proof}
 When $\ell=m+1$, $(a)$
 is in Griffiths \cite{Gr66}, Morrow-Rossi \cite{MR78} and
 (b) is proved in Ueda~\ci{Ue82} for flat line bundle $N_C^*$ over a compact curve $C$.

$(a)$ The general case can be verified by using
\rl{cc}
to compare coefficients of $w_j^\all$  on both sides of $\Phi_{ij}(z_j,w_j)=\Phi_{ik}\circ\Phi_{kj}(z_j,w_j)$ for $|\all|\leq 2m$. Indeed, we have $\Phi_{ik}= N_{ik}+(\phi_{ik}^h,\phi_{ik}^v)$ and $(\phi_{ik}^h,\phi_{ik}^v)(z_k,w_k)=O(|w_k|^{m+1})$ with $m\geq1$. Thus 
\aln
N_{ik}\circ\Phi_{kj}(z_j,w_j)&=\left\{N_{ik}\circ N_{kj}+DN_{ik}\circ N_{kj}\cdot (\phi_{kj}^h,\phi_{kj}^v)\right\}(z_j,w_j)+O(|w_j|^{2m+1})\\
& =N_{ik}\circ N_{kj}(z_j,w_j)
 +(s_{ik}(
\var_{kj}(z_j))\phi_{kj}^h,t_{ik}(
\var_{kj}(z_j))\phi_{kj}^v)\\
&\quad  +(0, Dt_{ik}(
\var_{kj}(z_j))
\phi_{kj}^h (z_j) t_{kj}(z_j)w_j)+O(|w_j|^{2m+1}).\nonumber
\end{align*}
Here $s_{kj}$ are the transition  
matrices
  of $TC$ given by \re{xskj}. Therefore,
\aln
\Phi_{ik}\circ\Phi_{kj}(z_j,w_j)&=N_{ik}\circ\Phi_{kj}(z_j,w_j)+
(\phi_{ik}^h,\phi_{ik}^v)\circ \Phi_{kj}(z_j,w_j)\\
&  =\left\{N_{ik}\circ N_{kj}
 +
(\phi_{ik}^h,\phi_{ik}^v)\circ N_{kj}\right\}(z_j,w_j)\\
& \quad+\left(s_{ik}(
\var_{kj}(z_j))\phi_{kj}^h(z_j,w_j),t_{ik}(
\var_{kj}(z_j))\phi_{kj}^v(z_j,w_j)\right)\\
& \quad + (0, Dt_{ik}(
\var_{kj}(z_j))
\phi_{kj}^h(z_j)t_{kj}(z_j)w_j)+O(|w_j|^{2m+1}).
\end{align*}
Comparing both sides of $\Phi_{ij}(z_j,w_j)=\Phi_{ik}\circ\Phi_{kj}(z_j,w_j)$ for the coefficients in $w_j$ of order $\ell =m+1$, 
we obtain the desired conclusion by \rl{cc}.


(b)
We have
$\Phi_{kj}(z_j,w_j)=(
\var_{kj}(z_j)+\phi_{kj}^h(z_j,w_j),t_{kj}w_j+\phi_{kj}^v(z_j,w_j))$ with $\phi_{kj}^v(z_j,w_j)=O(|w_j|^{m+1})$.  Here $t_{kj}$ are constant. We get from the vertical components of $\Phi_{kj}=\Phi_{ki}\Phi_{ij}$ that
\aln
\phi_{kj}^v(z_j,w_j)&=t_{ki}\phi_{ij}^v(z_j,w_j)+\phi_{ki}^v(\Phi_{ij}(z_j,w_j))
\\
&
=t_{ki}\phi_{ij}^v(z_j,w_j)+
\var_{ki}(N_{ij}(z_j,w_j))+O(|w_j|^{m+2}),
\end{align*}
since $(\Phi_{ij}-N_{ij})(z_j,w_j)=O(|w_j|^2)$.
This shows that $\{[\phi_{kj}^{ v}]^\ell\}\in Z^1(\cL U,N_C\otimes N_C^{*\ell})$ for
$\ell=m+1$ by \rl{cc} $(a)$.
This gives us \re{pvZ1}.
\end{proof}
	
\subsection{Automorphisms of neighborhood of the zero section of flat vector bundles}

Let $\phi_{kj}$ defined on $U_k\cap U_j$ be the transition functions of $C$. Let $\Phi_{kj}$, defined on $V_k\cap V_j$,  be the transition functions of $M$, and let $N_{kj}$, defined on $\tilde V_k\cap\tilde V_j$ be the transition functions of $N_C$, with $\tilde V_k=\pi^{-1} U_k$. We identify $(C, U_j)$ as  subsets of  $\tilde V_j$ via the zero-section.
 Recall $\Phi_{kj}, N_{kj}$, and $\phi_{kj}$ are the same on $U_k\cap U_j$. By Cartan-Serre theorem, for any integer $m$, the space of global sections, $H^0(C, T_CM\otimes S^mN_C^*)$, is finite dimensional.

We say that a vector bundle is {\it flat}
if its transition matrices are locally constant.

		\begin{defn}
		\begin{enumerate}
			\item A formal tangent vector field $Y_j$ on  ${\tilde V}
			_j$ vanishing at $U_j$ is identified with $Y_j=\sum_{\ell\geq1}Y_j^{\ell}$ with $Y_j^{\ell}\in \Gaa(U_j, T_CM\otimes S^{\ell}N_C^*)$
 via
$$
\sum_{|Q|=\ell}a_{Q}^\all(h_j)v_j^Q\DD{}{h_j^\all}+b_Q^\la(h_j)v_j^Q\DD{}{v_j^\la}\mapsto \sum_{|Q|=\ell}a_{Q}^\all(z_j)(w_j^*)^Q\DD{}{z_j^\all}+b_Q^\la(z_j)(w_j^*)^Q\DD{}{w_j^\la}.
$$

Here $(h_j,v_j)$ is the coordinate map for $v_j^\la\DD{}{w_j^\la}\in (N_C)_p$ and we identity $h_j$ with $z_j|_{U_j}$ and $\DD{}{v_j}$ with $\DD{}{w_j}|_{U_j}$.

			\item A formal automorphism of   $\tilde V_j$ at $U_j$ that is tangent to the identity is an automorphism of a formal neighborhood of the $0$-section of $\tilde V_j$, fixing $U_j$ pointwise.
		\end{enumerate}
	\end{defn}	
	\begin{lemma}\label{equiv-loc-glob}
		Let $
\{F_j\}_j$ be a collection of
formal automorphisms
of $\tilde V_j$ fixing $U_j$ pointwise.
Let  $
\{Y_j\}_j$ be a collection of formal tangent
		vector fields of 
$\tilde V_j$ vanishing at $U_j$. 
		We have
		\begin{enumerate}
			\item $
\{F_j\}_j$ defines an automorphism
$F$ of a formal neighborhood of the $0$-section in $N_C$  
			if and only $F_k\circ N_{kj}=N_{kj}\circ F_j$ for all $k,j$
			\item Suppose that  $N_C$ is flat.
			$\{Y_j\}_j$ defines a vector field $Y$
on a formal neighborhood of the $0$-section in $N_C$ if and only if
			$\{Y_j^{\ell}\}\in H^0(C, T_CM\otimes S^{\ell}N_C^*)$ for all $\ell$.
			\item Suppose that $N_C$ is not flat.
			$\{Y_j\}_j$ defines a vector field on a formal neighborhood of the $0$-section in $N_C$ if and only if
			$\{Y_j\}\in H^0_{{\text{twisted}}}(C, T_CM\otimes \oplus_{\ell\geq 2}S^{\ell}N_C^*)$ with respect to the linear operator 
$\delta_{nf}(\{(Y_j^h,Y_j^v)\})=\{(\tilde Y^h_{kj},\tilde Y_{kj}^v)\}$ with
			\gan
\tilde Y_{kj}^h	
=Y_k^h(N_{kj}(h_j,v_j))-D\phi_{kj}(h_j)Y_j^h(h_j,v_j),\\
\tilde Y_{kj}^v=Y_k^v(N_{kj}(h_j,v_j))-t_{kj}(h_j)Y_j^v(h_j,v_j)-Dt_{kj}(h_j)v_j.Y_j^h(h_j,v_j).
			\end{gather*}
		\end{enumerate}
	\end{lemma}

\begin{proof}
		Let $(h_j,v_j)$ be the coordinates in $N_C$ over $U_j$.
		Note that $\{Y_j\}$ defines a global tangent vector filed of $N_C$ if and only if $DN_{kj}(Y_j)=Y_k$. A homogeneous vector field of degree $\ell$ on $\tilde V_j$ is an element $Y_j^{\ell}\in C^0(U_j, T_CM\otimes S^{\ell}N_C^*)$ defined by $$Y_j^{\ell}(h_j,v_j)=\sum_{m=1}^nY_{j,m}^{\ell,h}(h_j,v_j)\frac{\partial}{\partial h_{j,m}}+ \sum_{r=1}^dY_{j,r}^{\ell,v}(h_j,v_j)\frac{\partial}{\partial v_{j,r}}=: Y_j^{\ell,h}+Y_j^{\ell, v}.
		$$
		Recall that $N_{kj}(h_j,v_j)=(\phi_{kj}(h_j),t_{kj}(h_j)v_j)$. Thus
\aln
		DN_{kj}\left(Y_j^{\ell,h}+Y_j^{\ell, v}\right)&=D\phi_{kj}(h_j)Y_{j}^{\ell,h}(h_j,v_j)+t_{kj}(h_j)Y_{j}^{\ell,v}(h_j,v_j)\\
		&\quad +\sum_{j=1}^n\sum_{r,s=1}^d\f{\pd t_{kj,rs}(h_j)}{\pd h_{j,m}}Y_{j,m}^{\ell,h}(h_j,v_j)v_{j,s}\frac{\partial}{\partial {v}_{k,r}},
\end{align*}
	where the last term is in $C^0(U_k\cap U_j,N_C\otimes S^{\ell+1}N_C^*)$.
	When $N_C$ is flat, we see that $DN_{kj}Y_j=Y_k$ if and only if $DN_{kj}Y_j^\ell=Y_k^\ell$ for each $\ell$ and that the latter holds  if and only if
	\beq\label{relation}
	Y_{k}^{\ell,h}(\phi_{kj}(h_j),t_{kj}v_j)=D\phi_{kj}(h_j)Y_{j}^{\ell,h}(h_j,v_j),\quad Y_{k}^{\ell,v}(\phi_{kj}(h_j),t_{kj}v_j)=t_{kj}Y_{j}^{\ell,v}(h_j,v_j).
	\eeq
	In other words, $\{Y_j^\ell\}_j$ defines a global section of $T_CM\otimes S^\ell N_C^*$.
\end{proof}
\begin{lemma}\label{t-1}
Let $F_j$ be a formal automorphism of $\tilde V_j$ in $N_C$, 
which is tangent to identity and preserves the splitting of $T_C(N_C)$ along $U_j$. Then, $F_j$ is the time-$1$ map of a unique
formal vector field $Y_j$ in $\tilde V_j$,
 vanishing on $U_j$
  up to order $\geq 2$.
\end{lemma}
\begin{proof}
Let $F_j$ be given by
$$
\tilde h_j=h_j+\sum_{|\all|\geq2}A_{j,\all}(h_j)v_j^\all, \quad \tilde v_j=v_j+\sum_{|\beta|\geq2}B_{j,\beta}(h_j)v_j^\beta.
$$
Drop the index $j$. We want to express it as the time-1 map of a tangent vector field
$$
Y=\sum_{\ell\geq2}\left\{\sum_{m=1}^nY_{m}^{\ell,h}(h,v)\frac{\partial}{\partial h_{m}}+ \sum_{r=1}^dY_{r}^{\ell,v}(h,v)\frac{\partial}{\partial v_{r}}\right\},
$$
where $Y_{m}^{\ell,h}(h,v),Y_{r}^{\ell,v}(h,v)$ are homogeneous polynomials in $v$ of degree $\ell$. The flow of $Y$ with time $\theta$ is given by
$$
h_m^\theta =h_m+\sum_{|\all|\geq2} A_{m,\all}^\theta(h)v^\all, \quad  v_r^\theta=v_r+\sum_{|\alpha|\geq2}B^\theta_{r,\alpha}(h)v^\alpha,
$$
where $A^\theta,B^\theta$ satisfy $A^0=B^0=0$ and
$$
\sum_{|\alpha|\geq 2}v_j^{\alpha}\frac{dA_{m,\all}^\theta(h_j)}{d\theta}=\sum_{\ell\geq2}Y_{m}^{\ell,h}(h^\theta,
v^\theta),\quad
\sum_{|\alpha|\geq 2}v_j^{\alpha}\frac{dB_{r,\all}^\theta(h)}{d\theta}= Y_{r}^{\ell,v}(h^\theta,v^\theta).
$$
Inductively, we can verify that $A_{m,\all}^1-Y^h_{m,\all}, B_{m,\all}^1-Y^v_{r,\all}$ are uniquely determined by $Y^{\ell,h}_{m',\beta},
Y^{\ell,v}_{r',\beta}$ with $\ell<|\all|$.
\end{proof}
Note that the formal time-1 mapping of $DN_{kj}(Y_j)$ on $
\tilde V_k\cap\tilde V_j$
 can also be defined and it equals $N_{kj}F_j N_{kj}^{-1}$ where $F_j$ is the time-1 map of $Y_j$. Thus the uniqueness assertion in the lemma implies the following.
\begin{prop}\label{autom-lin}
Any automorphism $F
$ of a formal neighborhood of  
$C$  in $N_C$, 
which is tangent to identity and preserves the splitting of $T_C(N_C),$
 is the time-1 map of a unique vector field defined on a formal neighborhood of  $C$ in
  $N_C$ and vanishing on $C$.
Assume further that $N_C$ is flat. Then any tangent vector field $Y$ of $N_C$
that vanishes on $C$ 
to order two
admits a
decomposition
$$
Y=\sum_{\ell\geq 2} Y^{\ell},\quad Y^{\ell} \in H^0(C, T_CM\otimes S^{\ell}N_C^*).
$$
\end{prop}
We write  $\delta_m= (\delta_m^h,\delta_m^v)$ corresponding to the splitting $T_CM=TC\oplus N_C$. Let us set $\cL G_m:=\text{Range}(\delta_m)$.
We have a decomposition
\eq{non-canonical}
Z^1(\cL U, T_CM\otimes S^mN_C^*)=\cL G_m\oplus\cL N_m
\eeq
 where $\cL N_m\simeq H^1(\cL U, TM_C\otimes S^mN_C^*)$.
Let $C^0(\cL U, TM_C\otimes S^mN_C^*)= \cL R_m\oplus \ker\delta_m$ with $\delta_m(\cL R_m)=\cL G_m$.
We emphasize that the decomposition \re{non-canonical} is not unique. For our convergence result, a natural decomposition will be given via a possibly non-unique minimizing solution. Consequently, $\oplus$ is interpreted as merely a decomposition suitable for convergence proof.

\begin{lemma} Suppose that $N_C$ is flat.
	Any formal transformation
	$F_j$ of $\tilde V_j$ which is tangent to identity and preserves the splitting of $T_C(N_C)$ can be  uniquely factorized as
	$$
	F_j= G_j^{-1}\circ H_j
	$$
	where $H_j-I\in \sum_{m\geq2}\cL R^m$, $
	G_j$ is an  automorphism of
	$\tilde V_j$,
	and terms of order $m$ in $G_j, H_j$ are uniquely determined by the terms of order at most $m$ in $F_j$.
Furthermore,   $G_i N_{ik}=N_{ik} G_k$ for all $i,k$.
\end{lemma}
\begin{proof}
	We know that $F_j=\exp \sum_mC_j^m$  is the time-1 map of  $\sum_{m\geq2}
	C_j^m$.
	
	 We want to decompose
	$$
	\exp \sum_mC_j^m=(\exp\sum_mA_j^m)(I+\sum_m H_j^m).
	$$
	By Campbell-Hausdorff formula, we are led to the equation
	$$
	H_j^m=C_j^m-A_j^m+E_j^m
	$$
	where $E_j^m$ depends only on $C_j^\ell, A_j^\ell$ for $\ell<m$. We determine $A_j^m,B_j^m$ by decomposing $C_j^m$ and $E_j^m$ as follow :  Let $\pi$ be the (non-canonical)
	projection from $C^0(\cL U, TM_C\otimes S^mN_C^*)$ onto $\ker\delta_m$. Let $\{A_j^m\}_j:= \pi(\{C_j^m+E_j^m\})$. Then $\{H_j^m\}\in \cL R_m$.
\end{proof}	
Next, we study the dependence of cohomology classes of $
[\phi^h_{kj}]^\ell,[\phi^v_{kj}]^\ell$ in coordinates.
We first consider the 
full set of linear cohomological equations.

\subsection{Formal coordinates
in the absence of formal obstructions}
For
a power series $u(z_j,w_j)$, let $u^{\leq m}(z_j,w_j)$ be the Taylor polynomial of $u$ about $w_j=0$ with degree $m$. Thus
 we can define
$$
u=u^{\leq m}+u^{>m}, \quad u^{>m}(z_j,w_j)=O(|w_j|^{m+1}),\quad
[u]^m=u^{\leq m}-u^{<m},\quad [u]_{\ell}^m=u^{\leq m}-u^{<\ell}.
$$
  In order to describe the coboundary operator in next lemma,
  we define   the linear  operator $\widetilde D$ by
$$
((\widetilde Du) f)(h_j,v_j):=\f{\pd u}{\pd h_j}(h_j,0)f^h(h_j,v_j)+\f{\pd u}{\pd v_j}(h_j,0)f^v(h_j,v_j),
$$
for a function $u(h_j,v_j)$.
The standard differential $D$ is given by
$$
((Du)f)
(h_j,v_j)=\f{\pd u}{\pd h_j}(h_j,v_j)f^h
(h_j,v_j)+\f{\pd u}{\pd v_j}(h_j,v_j)f^v(h_j,v_j).
$$
Thus
\eq{dtd}
 (Du
 -\widetilde D u)f(h_j,v_j)=(Du(h_j,v_j)-Du(h_j,0))f(h_j,v_j).
\eeq
For a multiindex $\all=(\all_h,\all_v)$, define
$$
  ( \tilde D^\all u)(h_j)=\left\{\f{\pd^{|\all|}u}{\pd h_j^{\all_h} \pd v_j^{\all_v}}\right\}(h_j,0).
$$
\le{centralizer}
 Let $\Phi_{kj}=N_{kj}+\phi_{kj}$
 satisfy condition $L_m$ with $m \geq 1$. Suppose that $F_j(h_j,v_j)=(h_j,v_j)+ f_j(h_j,v_j)$ with $f_j(h_j,v_j)=O(|v_j|^2)$ are formal mappings such that
$\{F_k^{-1}\Phi_{kj}F_j\}\in L_{m}$. Then, on $U_j\cap U_k$,  $l=2,\ldots,m$,
\al\label{partial-autom}
&(\delta\{[f_j]^{\leq l}\})_{kj}(h_j,v_j)=-\Bigl[N_{kj}((I+
[f_j]^{\leq l-2})(h_j,v_j) 
)-N_{kj}(h_j,v_j)\Bigr.
\\
&\qquad-DN_{kj}(h_j,v_j)[f_j]^{\leq l-2}(h_j,v_j)\Bigr]^{\leq l}
- \Bigl(0,(Dt_{kj}(h_j)[f_j^h]^{\leq l-1}(h_j,v_j))v_j\Bigr).\nonumber
\end{align}
\bpp 
\item If    $f_j(h_j,v_j)=O(|v_j|^{m+1})$ for all $j$, then $N_{kj}+\tilde\phi_{kj}=F_k^{-1}\Phi_{kj}F_j+O(|v_j|^{2m+1})$ hold if and only if on $U_j\cap U_k$
\eq{cohom-newtonl}
(\delta\{[f_i]^{\leq   2m}\})_{kj}=   [\tilde\phi_{kj}-\phi_{kj}]^{\leq 2m} -
 \left(0,(Dt_{kj}(h_j)[f_j^h]^{\leq 2m-1})v_j\right).
\eeq
\item
If $\{F_j\}$  defines a germ of biholomorphism of order $m$ at the zero section of the normal bundle, i.e.
$$
F_k^{-1}N_{kj}F_j(h_j,v_j)=N_{kj}(h_j,v_j)+O(|v_j|^{m+1})$$ and if $f_j^h(h_j,v_j)=O(|v_j|^{ m})$, then
$
\mathcal V_j^{\leq m}(h_j,v_j):=(h_j,v_j+[f_j^v]^{\leq m})
$
preserves $\{N_{kj}\}$.
\item \label{newton-scheme}Suppose 
$F_k^{-1}\Phi_{kj}F_j\in L_{2m}$. Assume further  that either $N_C$ is flat or
\eq{h02m}
    H^0(C,TC\otimes S^{p}N_C^*)=0,\quad 2\leq p \leq 2m.\eeq
Then there exist $\hat F_j=I+O(|v_j|^{m+1})$
where
 $[\hat F_j^h]^{2m}_{m+1}$ are uniquely determined  by  $[\Phi_{kj}]_{m+1}^{2m}$ such that $\hat F_k^{-1}\Phi_{kj}\hat F_j\in L_{2m}$. There exists a  unique
 decomposition $\{\hat F_j=\cL H_j\circ  \cL  V_j\circ \tilde F_j\}$ in the form
\ga
\label{cLHi}
\cL H_j(h_j,v_j)=(h_j+H_j(h_j,v_j),v_j),\\
\cL V_j(h_j,v_j)=(h_j,v_j+V_j(h_j,v_j)),\\
\label{tFji}
[\tilde F_j]^i=0, \ \forall 2\leq i\leq 2m,\quad  [H_j]^\ell=[V_j]^\ell=0,\  \forall \ell>2m.
\end{gather}
Furthermore, 
$[H_j]^\ell=[V_j]^\ell=0$ for $\ell\leq m$, and
$H_j$ are uniquely determined by
\eq{Bikj}
(\del^h\{H_i\})_{kj}= - [\phi_{kj}^h]^{\leq 2m}.
\eeq
Moreover,  $\tilde\phi_{kj}=\cL H_k^{-1}\Phi_{kj}\cL H_j-N_{kj}$ satisfy $\tilde\phi_{kj}^h(h_j,v_j)=O(|v_j|^{2m+1})$ and $\tilde\phi_{kj}^v(h_j,v_j)=O(|v_j|^{m+1})$,    and $V_i$ satisfy
\eq{Ajkj}
(\del^v\{V_i\})_{kj}=- [\tilde\phi_{kj}^v]^{\leq 2m}.
\eeq
\epp 
\ele
\begin{proof}
Let $\Phi_{kj}=N_{kj}+
\phi_{kj}$ and $\tilde \Phi_{kj}=N_{kj}+\tilde\phi_{kj}$.
 Suppose that both $\phi_{kj}$ and $\tilde \phi_{kj}$ are of order $\geq m+1$ (i.e. $O(|v_j|^{m+1})$) and $F_k\Phi_{kj}=\tilde\Phi_{kj}F_j$.
Recall that $F_k=I+f_k$.
To use the coboundary operator, we write
\begin{eqnarray}
f_k(N_{kj})-
\widetilde DN_{kj}f_j+ \phi_{kj}-\tilde\phi_{kj}&=& \underbrace{\bigl(f_k(N_{kj}-f_k(N_{kj}+\phi_{kj}))\bigr)}_{A}\label{LmA}\\
&&+\underbrace{\left(\tilde\phi_{kj}(I+f_j)-\tilde\phi_{kj}\right)}_B\nonumber\\
&&+\underbrace{\left(N_{kj}(I+f_j)-N_{kj}-\widetilde DN_{kj}f_j\right)}_C.\nonumber
\end{eqnarray}
Since $f_j$ has order $\geq 2$ at $v_j=0$, by the Taylor expansion at $N_{kj}$ and at $I$ respectively, both $A$ and $B$ are of order $\geq m+2$ (w.r.t $v_j$) at the origin. For the same reason, the $C$ is of order $\geq 4$. We recall that, for each $\ell\in \Bbb N^*$, the coboundary operator $\delta
$ sends $C^0(\cL U, T_CM\otimes S^{\ell}(N_C^{*}))$ into $C^1(\cL U, T_CM\otimes S^{\ell}(N_C^{*}))$ as sections. It is defined in coordinates by
$$(\delta f)_{kj}=\widetilde DN_{kj}f_j(h_j,v_j)-f_k(N_{kj}(h_j,v_j))$$
on $U_j\cap U_k$ when $f=\{f_j\}\in C^0(\cL U, T_CM\otimes S^{\ell}(N_C^{*}))$. As $\del$ preserves the degree $\ell$  of $f_j$  in $v_j$, we shall omit its dependence in $\ell$. Truncating the Taylor expansion of \re{LmA} at
$v_j=0$ up to degree $m$ will lead to the first point.

  Since $f_j(h_j,v_j)=O(|v_j|^2)$, then $A,B$  are of order $\geq m+1$.
Using \re{dtd}, we obtain
\aln
C &=N_{kj}(I+f_j(h_j,v_j))-N_{kj}(h_j,v_j)-DN_{kj}(h_j,v_j)f_j(h_j,v_j)\\
&\quad+(DN_{kj}(h_j,v_j)-DN_{kj}(h_j,0))f_j(h_j,v_j).
\end{align*}
We have
$(DN_{kj}(h_j,v_j)-DN_{kj}(h_j,0))f_j(h_j,v_j)=(0,Dt_{kj}(h_j)f_j^h(h_j,v_j)v_j)$.
Thus,
  $$C=(0,(Dt_{kj}(h_j)f_j^h  (h_j,v_j)v_j)+a(1)-a(0)-a'(0)$$
  with $a(\la)=N_{kj}(h_j+\la f_j^h, v_j+\la f_j^v)$.
   Note that
\aln
a(1)-&a(0)-a'(0)
=\int_0^1(1-\la)a''(\la)\, d\la
\\
&=\sum_{|\all|=2}\f{|\all|!}{\all!}\int_0^1(1-\la)D^\all N_{kj}(I+\la f_j)f_j^\all\, d\la\\
& =
\sum_{|\all|=2}\f{|\all|!}{\all!}\int_0^1(1-\la )D^\all N_{kj}(I+\la [f_j]^{\leq m-2})([f_j]^{\leq m-2})^\all\, d\la +O(|v_j|^{m+1})
\\
&=b(1)-b(0)-b'(0)+O(|v_j|^{m+1})
\end{align*}
 for   $b(\la)=N_{kj}(I+\la [f_j]^{\leq m-2})$. This proves   \re{partial-autom}.

  For point (a), we use \re{LmA} again. This time, we have $A(h_j,v_j)=O(|v_j|^{2m+1})$ and $B(h_j,v_j)=O(|v_j|^{2m+1})$, while $C=(0,DN_{kj}(h_j)[f_j^h]^{\leq2m-1}v_j)+O(|v_j|^{2m+1})$. We have derived \re{cohom-newtonl}.

For point (b),  note that $F_k^{-1}N_{kj}F_j=N_{kj}+O(|v_j|^{m+1})$ is equivalent to $F_kN_{kj}=N_{kj}F_j +O(|v_j|^{m+1})$.  From the vertical components, we  obtain
$$
t_{kj}(h_j)v_j+f_k^v(
\var_{kj}(h_j),t_{kj}(h_j)v_j)=t_{kj}(h_j+f_j^h)(v_j+f_j^v(h_j,v_j))
+O(|v_j|^{  m+1}).
$$
Since   $f_j^h=O(|v_j|^{ m})$ and $f_j^v=O(|v_j|^2)$,   the $m$-jet (w.r.t. $v_j$) above reads
\aln
t_{kj}(h_j)v_j+ [f_k^v]^{\leq m}(
\var_{kj}(h_j),t_{kj}(h_j)v_j)&=t_{kj}(h_j)(v_j+ [f_j^v]^{\leq m}(h_j,v_j)).
\end{align*}
That is that $\cL V_k^{\leq m}N_{kj}=N_{kj}\cL V_j^{\leq m}$, as $\cL V_j^{\leq m}(h_j,v_j)=(h_j,v_j+[f_j]^{\leq m}(h_j,v_j))$.

The point (c) follows from \rp{autom-lin} when $N_C$ is flat.
 For the remaining case, it follows from points (a) and (b) as follows.

By \re{cohom-newtonl} and $H^0(C,TC\otimes S^\ell N_C^*)=0$, we obtain $[f_j^h]^{m}_2=0$.    By $(b)$, we know that $[F_j]^{\leq m}$ preserve $N_{kj}$. Then $\hat F_j=F_j([F_j]^{\leq m})^{-1}$ meet the requirement. The uniqueness of $[\hat F_j^h]^\ell$ for $m<\ell\leq 2m$ follows from the assumption on $H^0$ too.

  We are seeking a unique decomposition $F_j=\cL H_j\circ\cL V_j\circ\tilde F_j$.   Let us write $F_k^{
-1}\Phi_{kj} F_j=N_{kj}+\tilde \phi_{kj}$ with $\tilde \phi_{kj}=O(|v_j|^{2m+1})$. From the horizontal component of \re{cohom-newtonl} in which $[\tilde\phi_{kj}^h]^{\leq 2m}=0$ and condition \re{h02m}, we {\it uniquely} determine
$\{[f_j^h]^{\leq 2m}\}$.
Take
$
\cL H_j(h_j,v_j)=(h_j+[f_j]^{\leq 2m}(h_j,v_j),v_j).
$
Then
\eq{}
\cL H_k^{-1}\Phi_{kj}\cL H_j(h_j,v_j)=(
\var_{kj}(h_j), t_{kj}(h_j)v_j+\tilde\phi^v_{kj}(h_j,v_j))+O(|v_j|^{2m+1}).
\label{partialhl}
\eeq
 We still have $(\cL H_k^{-1} F_k )^{-1}(\cL H_k^{-1}\Phi_{kj}\cL H_j)(\cL H_j^{-1}F_j)\in L_{2m}$.  We have
\ga
\cL H_j^{-1}F_j(h_j,v_j)=\cL V_j(h_j,v_j)+O(|v_j|^{2m+1}), \quad\cL V_j(h_j,v_j)=
(h_j,v_j+V_j(h_j,v_j)),
\label{partialh2}\end{gather}
where $\tilde\phi_{kj}^v, V_j$ contain only terms of orders $\ell$ in $v_j$ for $m+1\leq \ell\leq 2m$.

Since $F_j=\cL H_j\cL V_j+O(|v_j|^{2m+1})$, we
 have
 \eq{}\nonumber
 \cL V_k^{-1}(\cL H_k^{-1}\Phi_{kj}\cL H_j)\cL V_j\in L_{2m}.
 \eeq
  From the vertical components of \re{partialhl}-\re{partialh2},  and \re{cohom-newtonl} in which we take
$Dt_{kj}[f_j^h]^{\leq 2m-1}=0$, we see that    \re{cohom-newtonl} becomes \re{Ajkj}, i.e. $(\del^v[V]^\ell)_{kj}= -[\tilde\phi^v_{kj}]^\ell$ for $\ell=m+1,\dots, 2m$.
To show the
 uniqueness of $[F_j]^{\leq 2m}$, we may assume that $\Phi_{kj}=N_{kj}+O(|v_j|^{2m+1})$. Then the uniqueness follows from the above arguments.
\end{proof}
The following is in Ueda~\ci{Ue82}, when both the dimension and codimension of $C$ are one.

\le{linkv}
Let $\Phi_{kj}$ satisfy condition $V_m$ with $m\geq
 1$.   Suppose that $N_C$ is   flat  and $H^0(C,N_C\otimes S^\ell(N_C^*))=0$ for $1<\ell\leq m$.
 Then
$
[\phi^v_{kj}]^{m+1}\in H^1(\cL U,N_C\otimes S^{m+1}(N_C^*))$ is independent of coordinates of the neighborhoods of $C$.
Furthermore, there are   formal biholomorphic mappings $F_j=I+(f_j^h,f_j^v)$ with $f_j(h_j,v_j)=O(|v_j|^2)$ satisfy
\eq{fjmV}
\{F_k^{-1}\Phi_{kj}F_j\}\in V_{m+1}
\eeq
if and only if $[\phi^v_{kj}]^{m+1} =0$ in $ H^1(\cL U,N_C\otimes S^{m+1}(N_C^*))$.
When \rea{fjmV} holds,     $\{\tilde F_k^{-1}\Phi_{kj}\tilde F_j\}$ is still in $ V_{m+1}$, for
$$
\tilde F_j(h_j,v_j)=(h_j, v_j+[f_j^v]^{m+1}(h_j,v_j)).
$$
\ele
\begin{proof}
Let   $\tilde\Phi_{kj}:=F_k^{-1}\Phi_{kj}F_j$. We want to show that
$$
[\tilde\phi^v_{kj}]^{m+1}=[\phi^v_{kj}]^{m+1} \quad\text{in $H^1(\cL U,N_C\otimes
S^{m+1}(N_C^{*}))$},
$$
provided that $\tilde\Phi_{kj}(h_j,v_j)=  N_{kj}(h_j,v_j)+(\tilde\phi_{kj}^h,\tilde\phi_{kj}^v)$, $\Phi_{kj}(h_j,v_j)=  N_{kj}(h_j,v_j)+(\phi_{kj}^h,\phi_{kj}^v)$, and
\eq{tphim+1} \tilde\phi_{kj}^v(h_j,v_j)=O(|v_j|^{m+1}), \quad \phi_{kj}^v(h_j,v_j)=O(|v_j|^{m+1}).
\eeq

First, we have $F_j(h_j,v_j)=(h_j,v_j)+O(|v_j|^2)$. Suppose that $[f_j^v]^{\leq m_*-1}=0$ for $2\leq m_*\leq m$.
Comparing vertical components of $\Phi_{kj}\circ F_j=F_k\circ \tilde\Phi_{kj}$, we obtain
\al\label{Lkjv}\nonumber
&\left[t_{kj}\cdot \left(v_j+f_j^v(h_j,v_j)\right)\right]^{\leq m_*}=(\Phi^v_{kj}\circ F_j)^{\leq m_*}(h_j,v_j)\\
&\quad =(F^v_k\circ\tilde\Phi_{kj})^{\leq m_*}(h_j,v_j)=
(F_k^v)^{\leq m_*}\circ  N_{kj}(h_j,v_j).\nonumber
\end{align}
  Here the last identity is obtained from  $\tilde\Phi_{kj}(h_j,v_j)-N_{kj}(h_j,v_j)=O(|v_j|^2)$,  $[F_j^v]^{\leq m_*}(h_j,v_j)=v_j+[f_j^v]^{m_*}$, and \re{tphim+1}.
 Looking at terms of order $m_*$ in $w_j$, we see that $\{[f^v_j]^\ell\}$ is a global section of $N_C\otimes S^{\ell}(N_C^{*})$ for $\ell= m_*$.   This shows that $[f_j^v]^{\leq m_*}=0$ and we can take $m_*=m$, i.e. $[f_j^v]^{\leq m}=0$.

We also have $[\Phi^v_{kj}F_j]^{m+1}=t_{kj}[f_j^v]^{m+1}+[\phi_{kj}^v]^{m+1}$ and $[F^v_k\tilde\Phi_{kj}]^{m+1}=[f^v_k]^{m+1}\circ N_{kj} +[\tilde\phi_{kj}^v]^{m+1}.$ This shows that 
\eq{tphiv}
 [\tilde\phi_{kj}^v]^{m+1}-[\phi_{kj}^v]^{m+1}=t_{kj}[f_j^v]^{m+1}
-[f^v_k]^{m+1}\circ N_{kj}.
\eeq
The latter is equivalent to 
$[\tilde\phi_{kj}^v]^{m+1}=[\phi_{kj}^v]^{m+1}$ in
$H^1(\cL U,N_C\otimes S^{m+1}(N_C^{*}))$,
which follows from \rl{cc} $(b)$. The last assertion is equivalent to \re{tphiv} with $[\tilde\phi^v_{kj}]^{m+1}=0$.
%
\end{proof}

  \setcounter{thm}{0}\setcounter{equation}{0}
\section{
A majorant method for the  vertical linearization}
\label{sec:verlin}

Let $C$ be an $n$-dimensional complex compact manifold embedded in
an $(n+d)$-dimensional complex manifold. We assume that the normal bundle $N_C$ is flat and {\it unitary}. Let $\{t_{kj}\}$ be its transition (constant) matrices in a suitable covering ${\mathcal U}=\{U_j\}$ of $C$, we have $t_{kj}t_{kj}^*=\text{Id}$.
Let $K 
(N_C\otimes S^m(N_C^*))$ be the ``norm'' of the cohomological operator acting
on 
$C^0(\mathcal U, N_C\otimes S^m(N_C^*))$ as defined  in \rt{donin-sol-cohom}.
Let us
consider the sequence of numbers $\{\eta_m\}_{m\geq 1}$ with $\eta_1=1$ and
\beq\label{def-eta-foliation}
\eta_m=K(N_C\otimes S^m(N_C^*))\max_{m_1+\cdots +m_p+s=m} \eta_{m_1}\cdots \eta_{m_p},\quad m>1,
\eeq
where  $1\leq m_i<m$ for all $i$ and $s\in\Bbb N$.

In this section, we shall prove the following
\begin{thm}\label{vlin}
Let $C$ be a compact complex submanifold in $M$ with $T_CM=TC\oplus N_C$. Assume that the embedding is vertically linearizable  
by a formal holomorphic mapping which is tangent to the identity and preserves the splitting of $T_CM$
 or that $H^1(C,N_C\otimes S^\ell(N_C^*))=0$ for all $\ell\geq 2$. We also assume that $N_C$ is unitary flat and that $H^0(C,N_C\otimes S^\ell(N_C^*))=0$ for all $\ell\geq 2$.
Assume that for the $\eta_m$ defined above, there are positive constants $L_0,L$  such that $\eta_m\leq L_0L^m$ for all $m$. Then the embedding is actually holomorphically vertically linearizable.
\end{thm}
\begin{rem}\label{rem-vl}
		In the previous \rt{vlin}, if a neighborhood of $C$ is formally  vertically  linearizable by a minimizing  vertical mapping which is tangent to the identity and preserves the splitting of $T_CM$,	
 then the assumption "$H^0(C,N_C\otimes S^\ell(N_C^*))=0$, $\ell>1$" is not necessary.
	Here by a formal {\it minimizing vertical} mapping it means a map of the form $(h_j,v_j+f_j^v(h_j,v_j))$ with $\{f_j^v\}\in C^0(C,\bigoplus_{\ell\geq 2}N_C\otimes S^\ell(N_C^*))$ such that each $\{[f_j^v]^\ell\}_j$ is a possibly non-unique Donin (minimizing) solution of a suitable cohomology equation.
\end{rem}
\begin{cor}\label{compactcenter}
Under assumptions of \rta{vlin}, there exists, in a neighborhood of $C$ in $M$, a  smooth holomorphic $d$-dimensional foliation having $C$ as a leaf.
\end{cor}
\begin{proof}
According to \rt{vlin}, 
there is a neighborhood of the $C$ in $M$ with suitable holomorphic coordinates patches $(V_j, (h_j,v_j))$ with $(h_j,v_j)\in \Bbb C^n\times\Bbb C^d$ and $C\cap V_j=\{v_j=0\}$, such that, on $V_j\cap V_k$, we have
\begin{eqnarray*}
v_k=t_{kj}v_j,\quad
h_k=\tilde
\var_{kj}(h_j,v_j).
\end{eqnarray*}
We then define the foliation in chart $V_j$ by
 $dv_j=0$.
\end{proof}

The rest of the section is devoted to the proof of \rt{vlin}. We follow the method of majorant developed  by T. Ueda~\ci{Ue82} for $1$-dimensional unitary
normal
bundle over compact complex curve.
\subsection{Conjugacy equations and cohomological equations}
Let us first recall \re{h-ext} and \re{v-ext}~:
\beq\label{v-equation}
{\mathcal L}_{kj}^v(f^v_j)= \phi^v_{kj}(h_j,v_j+f_j^v)-\left(f_k^v(\hat\Phi^h_{kj}(h_j,v_j), t_{kj}v_j)-f_k^v(
\var_{kj}(h_j), t_{kj}v_j)\right)
\eeq
where
\begin{eqnarray*}
\hat\Phi^h_{kj}(h_j,v_j)&= &
\var_{kj}(h_j)+\phi^h_{kj}(h_j,v_j+f_j^v),\\
{\mathcal L}_{kj}^v(f^v_j)&=&f_k^v(
\var_{kj}(h_j),t_{kj}v_j)-t_{kj}f_j^v.
\end{eqnarray*}
Let us expand $\phi^h_{kj}(h_j,v_j+f_j^v)$ in power of $v_j$
by using
\begin{eqnarray}\nonumber
\phi^h_{kj}(h_j,w_j) &=:&\sum_{Q\in \Bbb N^d_2}\phi^h_{kj,Q}(h_j)w_j^Q\label{phih}\\
\phi^h_{kj}(h_j,v_j+f_j^v(h_j,v_j))&=:&\sum_{Q\in \Bbb N^d_2}h_{kj,Q}'(h_j)v_j^Q=:h'_{kj}(h_j,v_j).\nonumber
\end{eqnarray}
We have
\beq\label{h'}
\sum_{Q\in \Bbb N^d_2}h_{kj,Q}'(h_j)v_j^Q=\sum_{Q\in \Bbb N^d_2}\phi^h_{kj,Q}(h_j)(v_j+f_j^v(h_j,v_j))^Q.
\eeq
Let us also set
\beq\label{h''}\nonumber
\sum_{Q\in \Bbb N^d_2}h_{kj,Q}''(h_j)v_j^Q:=f_k^v(\hat\Phi^h_{kj}(h_j,v_j), t_{kj}v_j)-f_k^v(
\var_{kj}(h_j), t_{kj}v_j).
\eeq
As we shall see below, the functions $[h']^m$ and $[h'']^m$ are defined by induction on $m\geq 2$ as they depend on $[f]^l$, $l=2,\ldots, m-1$.


Therefore, the homogeneous polynomial of degree $m\geq 2$ of the Taylor expansion of solution of the conjugacy equation satisfies
\beq\label{maj-v-lin}
{\mathcal L}_{kj}^v([f^v_j]^m)= [h'_{kj}]^m+[h''_{kj}]^m.
\eeq
According to \rl{linkv}, there is a solution to the above
 equation
 either by the formal assumption or by the assumption that
  the cohomology class of $[h'_{kj}]^m+[h''_{kj}]^m$ is $0$, i.e. it is a coboundary. Indeed, since the normal bundle is flat, this class is independent of the coordinates system and the neighborhood is formally vertically linearizable.

\subsection{
A modified Fischer norm for symmetric powers}\label{fisher-section}

We define a scaler product on the space of polynomials $\Bbb C[x_1,\ldots, x_d]$ as follows. First, we set
\begin{equation}\label{EqScalp}
\scalxx{x^R}{x^Q}_{\mf}:= \begin{cases}\frac{(r_1!)\cdots (r_d!)}{|R|!}
\;&\text{  if  }\; R= Q\\
0\;&\text{  otherwise  }
\end{cases},\quad \left|\sum_{Q}
C_Qx^Q\right|_{\text{mf}}^2:=\sum_{Q}|C_Q|^2\f{Q!}{|Q|!},
\end{equation}
where $R=(r_1,\dots, r_d)$ and $|R|=r_1+\cdots+r_d$, and $C_Q$ are constants.
The subscript mf stands for ``modified Fischer''.
The associated norm will be denoted by $|.|_{k}$. The Fischer (resp. modified Fischer) scalar product has been used 
in \cite{Fi18,Sh89,IL05} (resp. \cite{LS10}).
Let $\om$ be an open set on $\Bbb C^n$.
For a vector of polynomials $g=(g_1,\dots, g_k)\in \cL O^k(\om)\otimes \Bbb C[x_1,\ldots, x_d]$, we set
\begin{equation}\label{EqScalPd}
|g|_{\mf, \om}^2:=\sup_{z\in\om} |g(z,\cdot)|_{\mf}^2:
=\sup_{z\in\om}\sum_{j=1}^k\sum_{Q\in\Bbb N^d}\frac{Q!}{|Q|!}|g_{j,Q}(z)|^2.
\end{equation}
We now apply the Fischer norm (resp. modified Fischer norm) to $f\in C^q(\cL U,
E\otimes S^L N_C^*)$. Returning to notation in \re{def-tilde-f+}, we write
\eq{fnf-}\nonumber
 f_{i_0\dots i_q}(p)=\sum_{\la=1}^{\rank E}\sum_{|Q|=L} f_{{i_0\dots i_q}; Q}^\la (z_{i_q}(p)) e_{i_0,\la }(p)\otimes (w_{i_q}^*(p))^Q,
\eeq
where $e_{i_0}$ is the base of $E$ over $U_{i_0}$ and $w^*_{i_q}$ is the base of $N_C^*$ on $U_{i_q}$.  Define 
\al\label{fnf}
|f|_{\mf,\cL U}^2&:=\max_{(i_0,\dots, i_q)\in \cL I^{q+1}}\
\sup_{z_{i_q}\in
\var_{i_q}(U_{i_0\dots i_q})}
\ \sum_{\la=1}^{\rank E}\sum_Q \f{Q!}{|Q|!}\left|f_{{i_0\cdots i_q}; Q}^\la (z_{i_q})\right|^2.\end{align}
{\it When there is no confusion, we shall in the sequel write ``f'' instead of ``mf''}.
The following two propositions are a ``version with parameters'' of ~\cite[propositions 3.6-3.7]{LS10} (see also~\cite{IL05}). We only give the proof of the last two points of next proposition. 

\begin{prop}\label{lombardi}
Let $\cL O_n(\om)\otimes \Bbb C[x_1,\ldots, x_d]$ be the set of polynomials $f(x,z)$  in $x$ with coefficients holomorphic in $z\in \om\subset\mathbb C^n$.
\bpp 
\item Let $f,g\in \cL O_n(\om)\otimes \Bbb C[x_1,\ldots, x_d]$ be homogeneous polynomials of  degree $k,k'$
respectively. Then
$$
|fg|_{f,\om}\leq |f|_{f,\om}|g|_{f,\om}.
$$
\item Let $f\in \cL O_n(\om)\otimes \Bbb C[x_1,\ldots, x_d]$  and let $\tilde f_P(z,x)=\f{1}{P!}\pd_z^Pf(z,x)$.
Then
$$
|\tilde f_P|_{f,\om'}  \leq \f{ |f|_{f,\om}}{(\dist_*(\om',\pd\om))^{|P|}}, \quad \forall\om'\subset\om, \
 \dist_*(\om',\pd\om):=\dist(\om',\pd\om)/{\sqrt n}.
$$
\item Let $T$ be a $d\times d$ unitary matrix.
Let $f\in \cL O_n^d(\om)\otimes \Bbb C[x_1,\ldots, x_d]$. Then,
$$
|Tf|_{f,\om}= |f|_{f,\om}.
$$
\item Let $T$ be a $d\times d$ unitary matrix.
Let $f\in \cL O_n(\om)\otimes \Bbb C[x_1,\ldots, x_d]$ and
 $f^T(z,x):=f(z,Tx)$. Then,
$$
|f^T|_{f,\om}= |f|_{f,\om}.
$$

\epp 
\end{prop}
\begin{proof}
We only prove the last two points. Fix $z\in\om'$. The polydisc  center at $z$ with radius $\del:=\dist(\om',\pd\om)/\sqrt n$ is contained in $\om$.

By the Cauchy formula, we have
\aln
\tilde f_P(z,x)&=\f{1}{\del^{|P|}}\int_{[0,2\pi]^n} f(z+\del (e^{i\theta_1},\dots, e^{i\theta_n}),x)(e^{i\theta_1},\dots,e^{i\theta_n})^{-P}\, \f{d\theta_1}{2\pi}\cdots \f{d\theta_n}{2\pi}\\
&=\f{1}{\del^{|P|}}\sum_{Q\in \nn^d} x^Q\int_{[0,2\pi]^n} f_Q(z+\del (e^{i\theta_1},\dots, e^{i\theta_n}))(e^{i\theta_1},\dots,e^{i\theta_n})^{-P}\, \f{d\theta_1}{2\pi}\cdots \f{d\theta_n}{2\pi}.
\end{align*}
We emphasize that the sum is finite.
By the Cauchy-Schwarz inequality  applied to the integral, we have
\aln
|\tilde f_P(z,\cdot)|_{\mf}^2&=\f{1}{\del^{2|P|}}\sum_{Q\in \nn^d}|x^Q|_{\mf}^2\left|\int_{[0,2\pi]^n} f_Q(z+\del(e^{i\theta_1},\dots, e^{i\theta_n}))(e^{i\theta_1},\dots,e^{i\theta_n})^{-P}\, \f{d\theta_1}{2\pi}\cdots \f{d\theta_n}{2\pi}\right|^2\\
&\leq \f{1}{\del^{2|P|}}\sum_{Q\in \nn^d}|x^Q|_{\mf}^2\int_{[0,2\pi]^n} |f_Q(z+\del(e^{i\theta_1},\dots, e^{i\theta_n}))|^2\, \f{d\theta_1}{2\pi}\cdots \f{d\theta_n}{2\pi} \\
&=\f{1}{\del^{2|P|}}\int_{[0,2\pi]^n}
\sum_{Q\in \nn
^d}|x^Q|_{\mf}^2 |f_Q(z+\del(e^{i\theta_1},\dots, e^{i\theta_n}))|^2
 \, \f{d\theta_1}{2\pi}\cdots \f{d\theta_n}{2\pi}\\
&\leq\f{1}{\del^{2|P|}}\int_{[0,2\pi]^n}|f|_{\om}^2\, \f{d\theta_1}{2\pi}\cdots \f{d\theta_n}{2\pi} =\f{1}{\del^{2|P|}}|f|_{\om}^2.
\end{align*}
For the last point,  we have, for a homogeneous polynomial $f$ in $x$ of degree $m$ with holomorphic coefficients in $\om$ the identity:
$$
|f_m|^2_{\om}=\frac{1}{\pi^dm!}\sup_{z\in \om}
\int_{\Bbb C^d} |f(z,x)|^2e^{-|x|^2}dV(x).
$$
In particular, the integral is invariant under the transformation $x\to Tx$ when $T$ is unitary (and constant).
\end{proof}
\begin{prop}  \label{PropHolomorphic}
For a  formal power series $f(h,v)=\sum_k f_k(z,v)$  
with $f_k(z,v)$ being a
homogeneous polynomial in $v$ of degree $k$ of which the coefficients are functions holomorphic in $z\in U,$ the following properties are equivalent:
\bpp 
\item $f$ is uniformly convergent for $v$ in a neighborhood of the origin, uniformly in
$U.$
\item  There exist  $M, R >0 $ such that for every $k$, $|f_{k}|_{\mf,U}\leq \frac{M}{R^k}$.
\epp 
\end{prop}

%

For convenience, we will use the following  orthonormal Fischer base of $ S^LN_C^*$:
$$
e^*_{j,Q}=\sqrt{\f{|Q|!}{Q!}}(w_j^*)^Q, \quad |Q|=L, \quad Q\in\nn^d.
$$
 The transition  
matrices
  $t^L_{kj}$ of $S^LN_C^*$ is then determined   in the following way~: Let $F_k=\sum_{|P|=L} F_{k,P}e_{k,P}^*$. We have
$$
(F_{k,P})_{|P|=L}=t_{kj}^L(F_{j,P})_{|P|=L}.
$$
This can be computed from the transition matrices of $N_C^*$ by expressing the basis $w_{k,1}^*,\ldots, w_{k,d}^*$ in terms of $w_{j,1}^*,\ldots, w_{j,d}^*$.
 Since $ t_{kj}^L$ maps orthonormal basis into orthonormal basis,  by \rp{lombardi} we know that $ t_{kj}^L$ are unitary matrices, i.e. in operator norm
 defined in \re{matrix-norm},
\eq{unitary-tL}
|t_{kj}^L|=1, \quad  L=1,2,\dots.
\eeq
We will apply results in the  appendix to the transition  
matrices
  $t_{kj}^L$.

\subsection{
A majorization in the modified Fischer norm for
the vertical linearization
}
Let $\{f^v_j\}$ be the formal solution of \re{v-equation}. 
We use notation \re{fnf}.
Let $
\var_j(U_j)=\Del_n$ and $U_{kj}:=U_k\cap U_j$. Define $\hat U_{kj}=
\var_{j}(U_{kj})$. Then, $
\var_{kj}(\hat U_{kj})=\hat U_{jk}$.
Let us first assume that $H^0(C,N_C\otimes S^\ell(N_C^*))=0$
 for all $\ell\geq 2$. We shall see later on how to  get rid of this assumption
to prove the general result.

Let us assume that there exists a vertical formal transformation $F:=\{F_j\}$ fixing $C$,
being tangent to identity on it, and preserving the splitting of $T_CM$ that linearizes vertically a neighborhood of $C$ in $M$. Let us write
$$
F_j(h_j,v_j):=(h_j,v_j+f_j),\quad f_j=\sum_{k\geq 2}[f_j]^k,\quad  \{[f_j]^k\}\in C^0(C,N_C\otimes S^k(N_C^*)).
$$

 Assume that there is a sequence $\{A_k\}_{k\geq 2}$ of positive numbers such that
\beq\label{majorfkak}
\forall k<m\quad |[f_j]^k|_{\hat U_j}\leq \eta_kA_k.
\eeq
Let us set
$$
A(t)=\sum_{k\geq 2}A_kt^k
$$
with $t\in \Bbb C$.
Let us first estimate both $|[h'_{kj}]^m|_{\hat U_{kj}}$ and $|[h''_{kj}]^m|_{\hat U_{kj}}$ in term of $J^{m-1}A(t):=A_2t^2+\cdots+ A_{m-1}t^{m-1}$. 

 Since $\phi^h_{kj}$ is holomorphic in $h_j\in \hat U_{kj}$  and $v_j$ in a neighborhood of the origin. We can assume that there is a positive $R$ such that
  \beq\label{estimphih}\nonumber
\sup_{h_j \in \hat U_{kj}}|\phi^{h}_{kj,Q}(h_j)|\leq R^{|Q|}
\eeq
for all $Q\in \Bbb N^d_2$,
where $ \phi^{h}_{kj,Q}$ is defined by \re{phih}  and
 $\Bbb N^d_k:=\{Q\in\Bbb N^d\colon|Q|\geq k\}$.

For $Q\in \Bbb N^d_2$, we have
$$
\left[(v_j+f_j^v(h_j,v_j))^Q\right]^m= \sum_{\dindice{(m_{1,1},\ldots,m_{1,q_1},\ldots,m_{d,1},\ldots,m_{d,q_d}) }{\sum_{i=1}^d{m_{i,1}+\cdots+m_{i,q_i}=m}}}\prod_{i=1}^d[f_{j,i}]^{m_{i,1}}\cdots [f_{j,i}]^{m_{i,q_i}}
$$
where we have set $f_j^v=(f_{j,1},\ldots, f_{j,d})$, $[f_{j,i}]^{1}= 
v_{j,i}$ and $[f_{j,i}]^{0}=0$. In the following,  all $m_{i,j}$ are positive integers. Hence, by the first point of \rp{lombardi}, we have
\beq\label{fnormi+f}
\left|\left[(v_j+f_j^v(h_j,v_j))^Q\right]^m\right|_{\hat U_{kj}}\leq \sum_{\dindice{(m_{1,1},\ldots,m_{1,q_1},\ldots,m_{d,1},\ldots,m_{d,q_d}) }{\sum_{i=1}^d{m_{i,1}+\cdots+m_{i,q_i}=m}}}\prod_{i=1}^d|[f_{j,i}]^{m_{i,1}}|_{\hat U_j}\cdots |[f_{j,i}]^{m_{i,q_i}}|_{\hat U_j}.
\eeq
Let $m\geq 2$, for $Q\in\Bbb N^d_2$, $|Q|\leq m$, let us set
$$
E_{Q,m}=\left\{(m_{1,1},\ldots,m_{1,q_1},\ldots,m_{d,1},\ldots,m_{d,q_d})\in\Bbb N^{|Q|}_1\colon \sum_{i=1}^d{m_{i,1}+\cdots+m_{i,q_i}=m}\right\}.
$$
 Let $M_i=(m_{1,1}^{(i)},\ldots,m_{1,q_1^{(i)}}^{(i)},\ldots,m_{d,1}^{(i)},\ldots,m_{d,q_d^{(i)}}^{(i)})\in\Bbb N^{|Q^{(i)}|}_1$
 with $|Q^{(i)}|\leq m_i$
 and $
 m_i=\sum_{j=1}^d m_{j,1}^{(i)}+\cdots+m_{j,q_j^{(i)}}^{(i)}$, $i=1,2$.  Define the concatenation
  $M_1\sqcup M_2$ 
  to be  $(M_1,M_2)$.
We also have $\sum_{j=1}^2\sum_{i=1}^d m_{i,1}^{(j)}+\cdots+m_{i,q_i^{(j)}}^{(j)}=m_1+m_2$.
Hence, we emphasize that the concatenation 
\beq\label{concatenation}
\left(\bigcup_{2\leq |Q_1|\leq m_1}E_{Q_1,m_1}\right)\sqcup \left(\bigcup_{2\leq |Q_2|\leq m_2}E_{Q_2,m_2}\right)\subset \bigcup_{2\leq |Q|\leq m_1+m_2}E_{Q,m_1+m_2}.
\eeq

As a consequence, according to \re{h'} and \re{estimphih}, we have
\al
\left|\left[\sum_{Q\in \Bbb N^d, |Q|=m} h_{kj,Q}'(h_j)v_j^Q\right]^m\right|_{\hat U_{kj}}&\leq \sum_{|Q|=2}^{m}R^{|Q|}\sum_{M\in E_{Q,m}}\prod_{i=1}^d|[f_{j,i}]^{m_{i,1}}|_{\hat U_j}\cdots |[f_{j,i}]^{m_{i,q_i}}|_{\hat U_j}\nonumber\\
&\leq \sum_{|Q|=2}^{m}R^{|Q|}\sum_{M\in E_{Q,m}}\prod_{i=1}^d\eta_{m_{i,1}}A_{m_{i,1}}\cdots \eta_{m_{i,q_i}} A_{m_{i,q_i}}\nonumber\\
&\leq \left[\sum_{|Q|=2}^{m}\eta_{Q,m}R^{|Q|}(t+J^{m-1}(A(t))^{|Q|}\right]^m\nonumber\\
&\leq  E_{m}[g_m(t)]^m,\label{estim-h'}
\end{align}
where we have set
\aln
&\eta_{Q,m}:=
\max
_{M\in E_{Q,m}}\left(\prod_{i=1}^d\eta_{m_{i,1}}\cdots \eta_{m_{i,q_i}}\right),\quad E_m:=\max_{\dindice{Q\in \Bbb N^d}{2\leq |Q|\leq m}} \eta_{Q,m},\\
& g_m(t):= \sum_{|Q|=2}^{m}R^{|Q|}(t+J^{m-1}(A(t))^{|Q|},\quad   g(t):= \sum_{|Q|\geq 2}R^{|Q|}(t+A(t))^{|Q|}.
\end{align*}
Hence, as formal power series, we have
\beq\label{exp-g}
g(t)=
\left(\frac{1}{1-R(t+A(t))}\right)^d-dR(t+A(t))-1.
\eeq
Let ${\mathcal U}^*=\{U_i^*\}$ be an  open covering of $C$ such that $U_i^*$ is relatively compact in $U_i$. We shall write $\hat U_k^*:=
\var_k(U_k^*)$.
Let us consider the index $j$ as fixed and let us estimate the Fischer norm of $h''_{kj}$ on $\hat U_{kj}^*:=
\var_j(U_j\cap U_k^*)$.
We have
\begin{eqnarray*}
\left[\sum_{Q\in \Bbb N^d, |Q|=m} h_{kj,Q}''(h_j)v_j^Q\right]^m &=&
\sum_{\dindice{P\in \Bbb N_1^n}{m_1+m_2=m}}\frac{1}{P!}\left[ \partial^P_h f_k
(\var_{kj}(h_j),t_{kj}v_j)\right]^{m_1}\left[\left(\phi_{kj}^h(h_j,v_j+f_j^v)\right)^P\right]^{m_2}\\
&=&
\sum_{\dindice{P\in \Bbb N_1^n}{m_1+m_2=m}}
\frac{1}{P!}\left[ \partial^P_h f_k(
\var_{kj}(h_j),t_{kj}v_j)\right]^{m_1}\left[\left(h'_{kj}(h_j,v_j)\right)^P\right]^{m_2}.
\end{eqnarray*}
Here, both indices $m_1$ and $m_2$ are $\geq 2$.
Since the Fischer norm is submultiplicative, we have
\beq\label{h'p}\nonumber
\left|\left[\left(h'_{kj}(h_j,v_j)\right)^P\right]^{m_2}\right|_{\hat U_{kj}^*}\leq E_{m_2}\left[\left(\sum_{|Q|=2}^{
 \frac{m}{2}
}R^{|Q|}(t+J^{m-1}(A(t))^{|Q|}\right)^{|P|}\right]^{m_2}.
\eeq
Indeed,
\begin{eqnarray*}
\left[\left(h'_{kj}(h_j,v_j)\right)^P\right]^{m_2}&=&\left[\prod_{i=1}^n(h'_{kj,i})^{p_i}\right]^{m_2}\\
&=&  \sum_{\sum_i(m_{i,1}+\cdots +m_{i,p_i})=m_2}\prod_{i=1}^n[h'_{kj,i}]^{m_{i,1}}\cdots [h'_{kj,i}]^{m_{i,p_i}}.
\end{eqnarray*}
According to \re{concatenation} and by \re{estim-h'}, we have
\begin{eqnarray*}\left|\prod_{i=1}^n[h'_{kj,i}]^{m_{i,1}}\cdots [h'_{kj,i}]^{m_{i,p_i}}\right|_{\hat U_{kj}^*}&\leq &\prod_{i=1}^nE_{m_{i,1}}\left[g_{m_{i,1}}(t)\right]^{m_{i,1}}\cdots E_{m_{i,p_i}}\left[g_{m_{i,p_i}}(t)\right]^{m_{i,p_i}}\\
&\leq & \max_{2\leq |Q|\leq m_{2}}\eta_{Q, m_{2}}\prod_{i=1}^n\left[g_{m_{i,1}}(t)\right]^{m_{i,1}}\cdots \left[g_{m_{i,p_i}}(t)\right]^{m_{i,p_i}}.
\end{eqnarray*}
Hence, we have
$$
\sum_{\sum_i(m_{i,1}+\cdots +m_{i,p_i})=m_2}\left|\prod_{i=1}^n[h'_{kj,i}]^{m_{i,1}}\cdots [h'_{kj,i}]^{m_{i,p_i}}\right|_{\hat U_{kj}^*}\leq  E_{m_{2}}[g(t)^{|P|}]^{m_2}.
$$
We have, by definition $\left[ \partial^P_h f_k(
\var_{kj}(h_j),t_{kj}v_j)\right]^{m_1}=\partial^P_h
 [f_{k}]^{m_1}(
\var_{kj}(h_j),t_{kj}v_j)$.

Recall that the Fischer norm is
unitary invariant and by \rp{lombardi}, we have 
\begin{eqnarray*}
\left|\partial^P_h  [f_{k}]^{m_1}(
\var_{kj}(h_j),t_{kj}v_j)\right|_{\hat U_{kj}^*} ^2&=& \left|\partial^P_h [f_{k}]^{m_1}(
\var_{kj}(h_j),v_j)\right|_{\hat U_{kj}^*}^2\\
&\leq &  \left(\frac{ P!}{\dist_*(\hat U_{k}^*,\pd \hat U_k)^{|P|}}\right)^2|[f_{k}]^{m_1}|_{\hat U_{k}}^2.
\end{eqnarray*}
Let us set $M:=\inf_k\dist(\hat U_{k}^*,\pd \hat U_k)$. 
 As a consequence, we have
\begin{eqnarray}
\left|\left[\sum_{Q\in \Bbb N^d, |Q|=m} h_{kj,Q}''(h_j)v_j^Q\right]^m\right|_{\hat U_{kj}^*}&\leq &\sum_{m_1+m_2=m}\sum_{\dindice{P\in \Bbb N^n}{|P|\geq 1}}
\frac{1}{M^{|P|}}|[f_{k}]^{m_1}|_{\hat U_{k}}E_{m_{2}}[g(t)^{|P|}]^{m_2}\nonumber\\
&\leq &\sum_{m_1+m_2=m}|[f_{k}]^{m_1}|_{\hat U_{k}}\left[E_{m_{2}}\sum_{\dindice{P\in \Bbb N^n}{|P|\geq 1}}
\left(\frac{g(t)}{M}\right)^{|P|}\right]^{m_2}\nonumber\\
&\leq & \left(\max_{m_1+m_2=m}\eta_{m_1}E_{m_2}\right)\left[ A(t) \left(\left(\frac{
M}{M-g(t)}\right)^n-1\right)\right]^m.\label{estim-h''}
\end{eqnarray}
Collecting estimates \re{estim-h'} and \re{estim-h''}, we obtain
\beq\nonumber
\left|{\mathcal L}_{kj}^v([f^v_j]^m)\right|_{\hat U_{kj}^*}\leq \left[E_m g(t)+ \left(\max_{m_1+m_2=m}\eta_{m_1}E_{m_2}\right)A(t) \left(\left(\frac{M}{M-g(t)}\right)^n-1\right)\right]^m.
\eeq

Let us extend this to an estimate on $\hat U_{kj}=
\var_j(U_j\cap U_k)$. Following again Ueda's argument~\ci{Ue82} let us express the fact that $[h]^m:=[h']^m+[h'']^m$ is a $1$-cocycle with values in $N_C\otimes S^m(N_C^*)$.  Let $p\in U_k\cap U_j$. Then   $p\in U_k\cap U_j\cap U_i^*$ for some $i$.  According to
\re{maj-v-lin} and
\rl{cc},  at $p\in U_k\cap U_j\cap U_i^*$ we have
\eq{3sums}
t_{ki}\sum_{|Q|=m} h_{ik,Q}(z_k(p)) (t_{kj}v_j)^Q -t_{ki}\sum_{|Q|=m}h_{ij,Q} (z_j(p))(v_j)^Q+\sum_{|Q|=m}h_{kj,Q}(z_j(p))(v_j)^Q
 =0.
\eeq
 Here by \re{fnf} the Fischer norms of $h_{kj}$ on all subdomains must be computed in the base  $e_{k}^v$ of $N_C$ 
on $U_k$ and the base $w_j^*$ of $N_C^*$ on $U_j$.
We can apply the previous estimates \re{estim-h'} and \re{estim-h''} to the first two sums respectively on $\hat U_{ik}^*$ and $\hat U_{ij}^*$.
To estimate the first sum,  we need to change  coordinates.
From
section~\ref{sec:formal},  $t_{kj}$ (resp. $s_{kj}$) are transition  
matrices
   of $N_C$ (resp. $TC$).
Recall that $\{[h_{kj}]^m\}\in Z^1(\cL U^{r_*}, N_C\otimes S^mN_C^*)$ and
\aln
 h_{ik}(p)&=\sum_{\la=1}^d \sum_{|Q|=m} h_{ik; Q}^\la (z_{k}(p)) e_{i,\la }^v(p)\otimes (w_{k}^*(p))^Q\\
&=\sum_{\la'=1}^d\sum_{\la=1}^d \sum_{|Q|=m} h_{ik; Q}^\la (z_{k}(p))
t_{
ki,\la}^{\la'} (z_k(p))e_{k,\la' }^v(p)\otimes (t_{kj}w_{j}^*(p))^Q=:\tilde h_{kj}(z_k(p), w_j^*).
\end{align*}
Thus,  $ \sum_{|Q|=m} h_{ik,Q}(z_k(p)) (t_{kj}v_j)^Q=\tilde h_{kj}(z_k(p),v_j)$.
By the unitary invariance  by multiplication and composition of the Fischer norm and
by definition \re{fnf}, we have for fixed $z_k(p)\in\hat U_{ik}^*$,
\aln
|\tilde h_{kj}( z_k(p), v_j)|_{\mf}^2
&=
\sum_{\la'=1}^d \left|\sum_{|Q|=m} \left(\sum_{\la=1}^d t_{ki,\la}^{\la'} ( z_k) h_{ik; Q}^\la ( z_{k}) \right)(t_{kj}v_j)^Q\right|_{\mf}^2\\
&=\sum_{\la'=1}^d\left|\sum_{|Q|=m} \left(\sum_{\la=1}^d t_{ki,\la}^{\la'} ( z_k) h_{ik; Q}^\la ( z_{k}) \right)v_j^Q\right|_{\mf}^2\\
&=
\sum_{\la'=1}^d\sum_{|Q|=m}\f{Q!}{|Q|!}\left|\sum_{\la=1}^d t_{ki,\la}^{\la'} ( z_k)h_{ik; Q}^\la (z_{k}) \right|^2\\
&\leq\sum_{\la'} 
\sum_{|Q|=m}\f{Q!}{|Q|!}\sum_{\la=1}^d  \left|h_{ik; Q}^\la (z_{k})  \right|^2\leq d |h_{ik}|_{\hat U_{ik}^*}^2,
\end{align*}
where the second last inequality is obtained by the Cauchy-Schwarz inequality.
In a similar way, we have a similar estimate for the second sum in \re{3sums} on $
\var_j(U_k\cap U_j\cap U_i^*)$. For the third sum in \re{3sums}, we note that the entries of the unitary matrix $t_{ki}$ have modulus at most one.
Thus, there exist constants $M', \tilde M$ such that the third sum in \re{3sums} satisfies
\begin{eqnarray*}
|h_{kj}|_{\hat U_{kj}}&\leq &
 M'\max_i(|h_{ik}|_{\hat U_{ik}^*}+|h_{ij}|_{\hat U_{ij}^*})\\
&\leq &
\tilde M\max\left(E_m,\max_{\dindice{m_1+m_2=m}{m_1,m_2\geq 2}}\eta_{m_1}E_{m_2}\right)\left[g(t)+ A(t) \left(\left(\frac{M}{M-g(t)}\right)^n-1\right)\right]^m.
\end{eqnarray*}

We now adapt the estimate in  \rl{SD-p} (see also \rt{donin-sol-cohom}). Recall that $[h_{kj}]^{\leq m}$ depends only on $[f]^{\leq m-1}$
and the hypothesis \re{majorfkak}. By the formal assumption, we have a solution to \re{maj-v-lin}:
$$
\cL L_{kj}([f_j^v]^m)=[h_{kj}]^m.
$$
 By assumptions, $H^0(C,N_C\otimes S^\ell(N_C^*))=0$, for all $\ell\geq 2$. Hence, the solution of the previous equation is unique.   By \rl{SD-p},
 \re{second-sol} and \re{unitary-tL}, 
the solution satisfies the   estimate:  
\eq{sup-est}\nonumber
|\{[f_j^v]^m\}|_{\cL U}\leq C(1+K_*(N_C\otimes S^mN_C^*)|)\{[h_{kj}]^m\}|_{\cL U}.
\eeq
Here,  $C$ depends neither on $N_C$ nor on $S^mN_C^*$. 
Therefore,
we have
$$
|[f_j^v]^m|_{\hat U_j}\leq K(N_C\otimes S^m(N_C^*)) \max_k\left|{\mathcal L}_{kj}^v([f^v_j]^m)\right|_{\hat U_{kj}^*}.
$$

By definition \re{def-eta-foliation}, we have
$$
K(N_C\otimes S^m(N_C^*))\max\left(E_m,\max_{\dindice{m_1+m_2=m}{m_1,m_2\geq 2}}\eta_{m_1}E_{m_2}\right)\leq \eta_m.
$$
Hence, we have
\beq\label{estim-f-induc}
|[\{f^v\}]^m|_{\hat U}\leq \tilde M\eta_m\left[g(t)+ A(t) \left(\left(\frac{M}{M-g(t)}\right)^n-1\right)\right]^m.
\eeq
Let us consider the functional equation
\beq\label{equ-A}\nonumber
A(t)=\cL F(t, A(t)):=\tilde M\left(g(t)+ A(t) \left(\left(\frac{M}{M-g(t)}\right)^n-1\right)\right),
\eeq
where $g(t)$ is a function of $A$ by \re{exp-g}. This equation has a unique analytic solution vanishing at the origin at order $2$.

We now can prove the theorem. Indeed by assumption, there are positive constants $M,L$ such that $\eta_m\leq ML^m$ for all $m\geq 2$. Since $A(t)$ converges at the origin, then $A_m\leq D^m$ for some positive $D$. According to \re{estim-f-induc}, we have also proved
$$
|[\{f^v\}]^m|_{\hat U}\leq \eta_mA_m,
$$
so that, finally, $|[\{f^v\}]^m|_{\hat U}\leq M(DL)^m$ for all $m\geq 2$. Hence, $f^v=\sum_{m\geq 2}[\{f^v\}]^m$ converges at the origin and this proves the theorem.

Let us see how we can prove \rrem{rem-vl}. The issue is that, when considering a solution $[f^{v}_j]^m$ of the cohomological equation ${\mathcal L}_{kj}([f^{v}_j]^m)=R^m$, the estimate given by \rl{SD-p} and \rp{defSD} might be obtained by another solution. Hence, the formal solution might not be the good one for the estimate.   Furthermore, we cannot replace a solution at degree $m$ as we wish to ensure that higher order terms in the vertical component can be eliminated formally.
We now explain the general result as formulated in the theorem. We will assume that there are formal mappings
$$
\tilde F_j(h_j,v_j)=(h_j,v_j)+\left(0,\sum_{\ell>2}\tilde f^v_{j,\ell}(h_j,v_j)\right)
$$
satisfying the following
\begin{enumerate}
	\item 
	$
	 \{\tilde F_k^{-1}\Phi_{kj}\tilde F_j-N_{kj}\}^v=0$ for all $ k,j.
	$
	In other words,  $\{\tilde F_j\}$ formally linearizes  $\Phi_{kj}$ vertically.  In particular,
	$$
	\{(\tilde F^m_k)^{-1}\Phi_{kj}\tilde  F^m_j-N_{kj}\}^v=[\phi^v_{kj}]^{m}+R_{kj}^{m}(\{[\phi_{kj}]^\ell, [\tilde f_k^v]^{\ell}\}_{2\leq\ell< m})+O(|v_j|^{m+1})$$
	 for
	$$
	\tilde F_j^m=(h_j,v_j)+\left(0,\sum_{2\leq\ell\leq m}\tilde f^v_{j,\ell}(h_j,v_j)\right).
	$$
	(The last assertion can be check easily since $(\tilde F_j^m)^{-1}\tilde F_j(h_j,v_j)=(h_j,v_j)+O(|v_j|^{m+1})$).
		\item Each $\{\tilde f_{j,m}^v\}_{j}$ is a ``minimizer'' in the sense that it satisfies the equation
		$$
		\{\del^v\tilde f_{m}^v\}_{kj}=[\phi^v_{kj}]^{m}+[R^{m}(\{[\phi_{kj}]^\ell, [\tilde f_k^v]^{\ell}\}_{2\leq\ell< m})]^m	$$
		and the estimate
		$$
	|\tilde f^v_m|\leq K(N_C\otimes S^m(N_C^*))|[\phi^v]^{m}+[R_{kj}^{m}(\{[\phi_{kj}]^\ell, [\tilde f_k^v]^{\ell}\}_{2\leq\ell< m})]^m|.		$$
		\end{enumerate}
		
As a consequence, the scheme of convergence applies to that formal
solution $\{\tilde F_j\}$ and we are done.

\setcounter{thm}{0}\setcounter{equation}{0}
\section{
A majorant method for the full linearization with a unitary normal bundle}
\label{sec:majorlin}

  In this  section, we shall devise  a proof of  \rt{lin-unit}, that is of the linearization of the neighborhood problem in the case $N_C$ is unitary (and  flat)  following a majorant method scheme.  

Let us recall the {\it horizontal cohomological operator}
$$
{\mathcal L}_{kj}^h(f^h_j):=f_k^h(
\var_{kj}(h_j),t_{kj}v_j)  -s_{kj}(h_j)f_j^h(h_j,v_j),
$$
where $s_{kj}(h_j)=D
\var_{kj}(h_j)$. We then have the horizontal equation \re{h-conjug}
\begin{eqnarray}
{\mathcal L}_{kj}^h(f^h_j)&=& \phi^h_{kj}(h_j+f_j^h,v_j+f_j^v)\label{h-lin-flat}\\
&&+\, \
\var_{kj}(h_j+f_j^h(h_j,v_j))-
\var_{kj}(h_j)-D
\var_{kj}(h_j)f_j^h(h_j,v_j).\nonumber
\end{eqnarray}

Let us recall the {\it vertical cohomological operator}
\beq\nonumber
{\mathcal L}_{kj}^v(f^v_j):=f_k^v(
\var_{kj}(h_j),t_{kj}v_j)-t_{kj}f_j^v, 
\eeq
and vertical equation \re{v-conjug} (recall that $N_C$ is flat)
\beq
{\mathcal L}_{kj}^v(f^v_j)=
\var^v_{kj}(h_j+f_j^h,v_j+f_j^v).\label{v-lin-flat}
\eeq
By assumption, there exists a formal solution $f_j=(f_j^h,f_j^v)=\sum_{k\geq 2}[f_j]^k$ with $\{[f_j]^k\}\in C^0(C, T_CM\otimes S^k(N_C^*))$. In case we assume $H^1(C, T_CM\otimes S^k(N_C^*))=0$, for all $k\geq 2$, this follows from \rl{verifycochain}.
We now use
 the ``norm'' of the cohomological operator acting on $C^0(\mathcal U, T_CM\otimes S^m(N_C^*))$ as defined by \rt{donin-sol-cohom}.
We have, for $m\geq 2$
$$
\tilde K_m:=\max\left( K(N_C\otimes S^m(N_C^*)), K(T_C\otimes S^m(N_C^*))\right).
$$
As in the foliation problem, we consider the sequence of numbers $\{\eta_m\}_{m\geq 1}$ with $\eta_1=1$ and, if $m\geq 2$
\beq\label{def-eta}
\eta_m:=\tilde K_m\max_{m_1+\cdots +m_p+s=m} \eta_{m_1}\cdots \eta_{m_p},
\eeq
where, in the maximum, $1\leq m_i<m$ for all $i$ and $s\in\Bbb N$.
In what follows, $f_j^\bullet$ (resp. $\phi_{kj}^\bullet$) stands for either $f_j^h$ or $f_j^v$ (resp. $\phi_{kj}^h$ or $\phi_{kj}^v$).
As in the previous section, let us expand $\phi^{\bullet}_{kj}(h_j+f_j^h,v_j+f_j^v)$
appeared in \re{h-lin-flat}
and \re{v-lin-flat}  in power series  of $v_j$ and let us define
\begin{eqnarray}\nonumber
\phi^{\bullet}_{kj}(z_j,w_j) &=:&\sum_{Q\in \Bbb N^d_2}\phi^{\bullet}_{kj,Q}(z_j)w_j^Q\label{phihl}\\
\phi^{\bullet}_{kj}(h_j+f_j^h(h_j,v_j),v_j+f_j^v(h_j,v_j))&=:&\sum_{Q\in \Bbb N^d_2}h_{kj,Q}^{\bullet}(h_j)v_j^Q=:h^{\bullet}_{kj}(h_j,v_j).
\nonumber
\end{eqnarray}
Then we obtain
\beq\label{hl}\nonumber
\sum_{Q\in \Bbb N^d_2}h_{kj,Q}^{\bullet}(h_j)v_j^Q=\sum_{Q\in \Bbb N^d_2}\phi^{\bullet}_{kj,Q}(h_j+f_j^h(h_j,v_j))(v_j+f_j^v(h_j,v_j))^Q.
\eeq
We further expand the first expression on the right-hand side as
\beq\label{h-compo-lin}\nonumber
\tilde h^{\bullet}_{kj,Q}:=\phi^{\bullet}_{kj,Q}(h_j+f_j^h(h_j,v_j))=\sum_{P\in \Bbb N^n}\frac{1}{P!}\partial^P_h\phi^{\bullet}_{kj,Q}(h_j)(f_j^h(h_j,v_j))^P.
\eeq
Hence, for any $m\geq 2$,
$$
[h^{\bullet}_{kj}]^m=\sum_{m_1+m_2=m}\sum_{Q\in \Bbb N^d_2}\sum_{P\in \Bbb N^n}\frac{1}{P!}\partial^P_h\phi^{\bullet}_{kj,Q}(h_j)\left[(f_j^h(h_j,v_j))^P\right]^{m_1}\left[(v_j+f_j^v(h_j,v_j))^Q\right]^{m_2}.
$$

Let $\{f^{\bullet}_j\}$ be the formal solution of  
\re{h-lin-flat} and \re{v-lin-flat}.  Let us first assume that $H^0(C,T_CM
\otimes S^\ell(N_C^*))=0$
 for all $\ell\geq 2$. We shall see later on how to get rid of the assumptions.
Assume that there is a sequence $\{A_k\}_{k\geq 2}$ of positive numbers such that
\beq\label{majorfkakl}\nonumber
\forall k<m\quad |[f_j]^k|_{\hat U_j}\leq \eta_kA_k.
\eeq
Let us set
$$
A(t)=\sum_{k\geq 2}A_kt^k
$$
with $t\in \Bbb C$.

Since $\phi^{\bullet}_{kj}$ is holomorphic in $h_j\in \hat U_{kj}$ and $v_j$ in a neighborhood of the origin, we can assume that there is a positive $R$ such that
 \beq\label{estimphih}
\sup_{h_j \in \hat U_{kj}}|\phi^{\bullet}_{kj,Q}(h_j)|\leq R^{|Q|}.
\eeq

According to \re{fnormi+f} and the proof of \re{estim-h'}, we obtain 
\begin{eqnarray}
\left|\left[(v_j+f_j^v(h_j,v_j))^Q\right]^{m_2}\right|_{\hat U_{kj}}&\leq &\sum_{\dindice{(m_{1,1},\ldots,m_{1,q_1},\ldots,m_{d,1},\ldots,m_{d,q_d}) }{\sum_{i=1}^d{m_{i,1}+\cdots+m_{i,q_i}=m_2}}}\prod_{i=1}^d|[f_{j,i}]^{m_{i,1}}|_{\hat U_j}\cdots |[f_{j,i}]^{m_{i,q_i}}|_{\hat U_j}\nonumber\\
&\leq &  \sum_{M\in E_{Q,m_2}}\prod_{i=1}^d\eta_{m_{i,1}}A_{m_{i,1}}\cdots \eta_{m_{i,q_i}} A_{m_{i,q_i}}\nonumber\\
&\leq & \eta_{Q,m_2}\left[\left(t+J^{m_2-1}A(t)\right)^{|Q|}\right]^{m_2}. \label{Qm2}\nonumber
\end{eqnarray}
 On the other hand, let ${\mathcal U}^*=\{U_i^*\}$ be an  open covering of $C$ such that $U_i^*$ is relatively compact in $U_i$. We shall write $\hat U_k^*:=
\var_k(U_k^*)$. Let us set
$$
M:= \min_k\dist(\hat U_{k}^*,\pd \hat U_{k}).
$$
Let us consider the index $j$ as fixed and let us estimate the Fischer norm of $[\tilde h^{\bullet}_{kj}]^{m_1}$ on $\hat U_{kj}^*:=
\var_j(U_j\cap U_k^*)$. We get
\begin{eqnarray*}
\left|[\tilde h^{\bullet}_{kj}]^{m_1}\right|_{\hat U_{kj}^*}&=& \sum_{P\in \Bbb N^n}\frac{1}{P!}\left|\partial^P_h\phi^{\bullet}_{kj,Q}(h_j)\left[(f_j^h(h_j,v_j))^P\right]^{m_1}\right|_{\hat U_{kj}^*}\\
&\leq & \sum_{P\in \Bbb N^n}\left(\frac{1}{\dist(
\hat U_{k}^*,\pd \hat U_{k})}\right)^{|P|}\left|\phi^{\bullet}_{kj,Q}\right|_{\hat U_{kj}}\left|\left[(f_j^h(h_j,v_j))^P\right]^{m_1}\right|_{\hat U_{kj}^*}\\
&\leq &\sum_{P\in \Bbb N^n}\left(\frac{1}{M}\right)^{|P|}R^{|Q|}\left|\left[(f_j^h(h_j,v_j))^P\right]^{m_1}\right|_{\hat U_{kj}^*}.
\end{eqnarray*}
Since $f_j$ is of order $\geq 2$ at $v_j=0$, we have $|P|\leq \f{m_1}{2}$ in the 
above sum. According to estimate \re{fnormi+f} and following the proof of \re{estim-h'}, we obtain
\begin{eqnarray}
\left|[\tilde h^{\bullet}_{kj,Q}]^{m_1}\right|_{\hat U_{kj}^*}&\leq &\sum_{P\in \Bbb N^n,|P|=0}^{\frac{m_1}{2}}\left(\frac{1}{\dist(\hat U_{k}^*,\pd \hat U_{k})}\right)^{|P|}R^{|Q|}\eta_{P,m_1}\left[A(t)^{|P|}\right]^{m_1}.\label{Pm1}
\end{eqnarray}
Combining inequalities \re{Pm1} and \re{Qm2}, we obtain
\begin{eqnarray*}
\left|[h^{\bullet}_{kj}]^m\right|_{\hat U_{kj}^*}&\leq & \sum_{m_1+m_2=m}\sum_{Q\in \Bbb N^d_2}\sum_{P\in \Bbb N^n}\frac{1}{P!}\left|\partial^P_h\phi^{\bullet}_{kj,Q}(h_j)\left[(f_j^h(h_j,v_j))^P\right]^{m_1}\left[(v_j+f_j^v(h_j,v_j))^Q\right]^{m_2}\right|_{\hat U_{kj}^*}\\
&\leq & \sum_{m_1+m_2=m}\sum_{\dindice{Q\in \Bbb N^d}{|Q|=2}}^{m_2}\sum_{\dindice{P\in \Bbb N^n}{|P|=0}}^{\f{m_1}{2}}\left(\frac{1}{M}\right)^{|P|}R^{|Q|}\eta_{P,m_1}\left[A(t)^{|P|}\right]^{m_1}\eta_{Q,m_2}\left[\left(t+J^{m_2-1}A(t)\right)^{|Q|}\right]^{m_2}\\
&\leq &\sum_{m_1+m_2=m}\sum_{\dindice{Q\in \Bbb N^d}{|Q|=2}}^{m_2}\sum_{\dindice{P\in \Bbb N^n}{|P|=0}}^{\f{m_1}{2}}\left[\left(\frac{A(t)}{M}\right)^{|P|}\right]^{m_1}\eta_{P,m_1}\eta_{Q,m_2}\left[\left(Rt+RJ^{m_2-1}A(t)\right)^{|Q|}\right]^{m_2}\\
&\leq &\tilde E_m \left[ \left(\frac{1}{1-\frac{A(t)}{M}}\right)^n\left(\left(\frac{1}{1-(Rt+RA(t))}\right)^d- 1 -d(Rt+RA(t))\right)\right]^m.
\end{eqnarray*}
 Here, we have set
$$
\tilde E_m =\max_{m_1+m_2=m}\max_{\dindice{P\in \Bbb N^n, Q\in \Bbb N^d}{|P|\leq \frac{m_1}{2}, 2\leq |Q|\leq m_2, }} \eta_{P,m_1}\eta_{Q,m_2}.
$$
It remains to estimate the rest of terms in \re{h-lin-flat}. We define
\aln
B_m:&=\left[
\var_{kj}(h_j+f_j^h(h_j,v_j))-
\var_{kj}(h_j)-D
\var_{kj}(h_j)f_j^h(h_j,v_j)\right]^m\\
&=\sum_{l=2}^{\f m2}\sum_{|P|=l}\frac{1}{P!}\partial_h^P
\var_{kj}(h_j)\left[(f_j^h)^P\right]^m.
\end{align*}
Hence, as above, we have
\begin{eqnarray*}
|B_m|_{\hat U_{kj}^*} &\leq & |
\var_{kj}|_{\hat U_{kj}}\sum_{l=2}^{\f m2}\sum_{|P|=l}\left(\frac{1}{M}\right)^{|P|}\left[(A(t))^{|P|}\right]^m\\
&\leq & |
\var_{kj}|_{\hat U_{kj}}\left[\left(\frac{1}{1-\frac{A(t)}{M}}\right)^n-1-n\frac{A(t)}{M}\right]^m.
\end{eqnarray*}

By the same reasoning as in the foliation section, the previous estimates on $\hat U_{kj}^*$ extend to estimates on $\hat U_{kj}$, by multiplication by a constant $\tilde M$.

Let us   
define constant $C_0:= \max_{kj}|
\var_{kj}|_{\hat U_{kj}}$
Since we have
$$
|[f_j^{\bullet}]^m|_{\hat U_j}\leq \tilde K_m\max_k\left|{\mathcal L}_{kj}([f^{\bullet}_j]^m)\right|_{\hat U_{kj}},
$$
then
\begin{align*}
|[f_j^{\bullet}]^m|_{\hat U_j}&\leq \tilde M\tilde K_m\left(
 C_0
\left[\left(\frac{1}{1-\frac{A(t)}{M}}\right)^n-1-n\frac{A(t)}{M}\right]^m\right.\\
&
\quad
+\left. \tilde E_m \left[ \left(\frac{1}{1-\frac{A(t)}{M}}\right)^n\left(\left(\frac{1}{1-(Rt+RA(t))}\right)^d- 1 -d(Rt+RA(t))\right)\right]^m\right).
\end{align*}
 We emphasize that due to the vanishing assumption of the spaces $H^0(\mathcal U, T_CM\otimes S^m(N_C^*))$, $m\geq 2$, the solution of cohomological equation ${\mathcal L}_{kj}([f^{\bullet}_j]^m)=R^m$ is unique and is equal to the minimizing solution obtained in \rl{SD-p} and \rp{defSD}.
Consider the following analytic functional equation~:
\begin{eqnarray*}
A(t)&=&\tilde M\left(  
C_0
\left[\left(\frac{1}{1-\frac{A(t)}{M}}\right)^n-1-n\frac{A(t)}{M}\right]\right.\\
&& +\left. \left(\frac{1}{1-\frac{A(t)}{M}}\right)^n\left(\left(\frac{1}{1-(Rt+RA(t))}\right)^d- 1 -d(Rt+RA(t))\right)\right).
\end{eqnarray*}
It has a unique analytic solution $A$ of  order $\geq 2$ at the origin.
Since we have
$$
\tilde K_m\max(1,\tilde E_m)\leq \eta_m,\quad|[f_j^{\bullet}]^m|_{\hat U_j}\leq A_m\eta_m, \quad m\geq 2
$$
then $\sum_{m\geq 2}[f_j^{\bullet}]^m$ converges in a neighborhood of the origin.\\

Let us see how the general case reduces to the previous one. The issue is that, when considering a solution $[f^{\bullet}_j]^m$ of the cohomological equation ${\mathcal L}_{kj}([f^{\bullet}_j]^m)=R^m$, the estimate given by \rl{SD-p} and \rp{defSD} might be obtained by another solution. Hence, the formal solution might not be the good one
for the
 estimates. So we will need to correct it. As we already emphasized, equations \re{h-lin-flat} and \re{v-lin-flat} read
	$$
	\mathcal{L}_{kj}(\{[f]^{\ell}_i\})=\mathcal{R}_{kj,\ell}([f]^{\ell'}, \ell'<\ell; [\Phi]^{l}, l\leq \ell)
	$$
where $	\mathcal{R}_{kj, \ell}$ is an analytic function of its arguments. Let us start at $\ell=2$.
	\begin{enumerate}
		\item  $\mathcal{R}_{kj,2}$ is just a function of the $[\Phi_{kj}]^2$'s and and we have $\mathcal{L}_{kj}([f]^{2})=\mathcal{R}_{kj, 2}$. Let $\{[\tilde f_{j,2}]^2\}$ be the minimizer solution of this equation obtained by \rl{SD-p} and \rp{defSD} and let $[k_j]^2:=[f_j]^2-[\tilde f_{j,2}]^2$. We have $\{[k_j]^2\} \in H^0(\mathcal U, T_CM\otimes S^2(N_C^*))$.
		\item According to lemma \rl{autom-lin}, $F_{j,2}:=F_j\exp(-[k_j]^2 )$ linearizes $\Phi_{kj}$ since
		$$
		F_{j,2}^{-1}\Phi_{kj}F_{j,2}= \exp(-[k_j]^2 )^{-1}N_{kj}\exp(-[k_j]^2 )=N_{kj}.
		$$
		$F_{j,2}$ is tangent to identity and its 2nd order term is the minimizer $[\tilde f_j]^2$.
		\item Assume that $F_{j, \ell}$ linearizes $\Phi_{kj}$, is tangent to identity at the origin and has the minimizers solution up to degree $\ell$ as Taylor expansion at 0. This means that  $F_{j, \ell}=Id+\sum_{l=2}^{\ell}[\tilde f_{j,l}]^{l}+ \sum_{l\geq \ell +1}[f_{j,\ell}]^{l}$.
		Let us write the conjugacy equation. By induction we have, for all $2\leq l\leq \ell$,
		$$
		\mathcal{L}_{kj}(\{[\tilde f_{i,l}]^{l}\})=\mathcal{R}_{kj,l}(\{[\tilde f_{i,l'}]^{l'}\}_i, l'<l; [\Phi]^{m}, m\leq l).
		$$
		Furthermore, it satisfies at degree $\ell+1$
		$$
				\mathcal{L}_{kj}(\{[f_{i,\ell+1}]^{\ell+1}\})=\mathcal{R}_{kj,\ell+1}(\{[\tilde f_{i,\ell'}]^{\ell'}\}_i, \ell'\leq \ell; [\Phi]^{m}, m\leq \ell+1).
		$$
		Let $[\tilde f_{i,\ell +1}]^{\ell+1}$ be the minimizer solution of the above cohomological equation. Let $[k_{i,\ell+1}]^{\ell+1}= [f_{i,\ell+1}]^{\ell+1}-[\tilde f_{i,\ell+1}]^{\ell+1}$. As above, it defines an element of  $H^0(\mathcal U, T_CM\otimes S^{\ell+1}(N_C^*))$.
		Let us set
		$F_{j,\ell+1}=F_{j, \ell}\exp( [k_{j,\ell+1}]^{\ell+1})^{-1}$. Then it linearizes $\Phi_{kj}$ and has the minimizers solution up to degree $\ell+1$ as Taylor expansion at 0: $F_{j, \ell+1}=Id+\sum_{l=2}^{\ell+1}[\tilde f_{j,l}]^{l}+ \sum_{l\geq \ell +2}[f_{j,\ell+1}]^{l}$.
		\item Since $F_{j,\ell+1}F_{j,\ell}^{-1}=I+O(\ell+1)$, the sequence $\{F_{j,\ell}\}_{\ell}$ converges in the space of formal power series to $\tilde F_j$. Furthermore, $\{\tilde F_j\}$ linearizes $\{\Phi_{kj}\}$ as each $\{F_{j,\ell}\}_{j}$ does.
		The Taylor expansion of $\tilde F_j$  at the origin is
		$$
		 \tilde F_{j}=Id+\sum_{l\geq 2}[\tilde f_{j,l}]^{l}.
		$$
		\item We can estimate the $[\tilde f_{j,l}]^{l}$ as we did above in the case of vanishing cohomology since the Taylor coefficient are minimizer solutions of the same equations.
	\end{enumerate}
Hence, we are done.

 In summary, we have proved the following theorem.
\begin{thm}\label{flin}
Let $C$ be an embedded compact manifold in $M$.
Assume that the embedding is   linearizable  
by a formal holomorphic mapping which is tangent to the identity
and preserves the splitting  of $T_CM$, and
$N_C$ is unitary. 
Suppose that $\{\eta_m\}_{m\geq 1}$ defined by \rea{def-eta} satisfy $\eta_m\leq L_0L^m$, for some positive numbers $L_0,L$
and for all $m$. Then the embedding is actually holomorphically  linearizable.
\end{thm}

We remark that in general there is a rigid theory on deformations in an  analytic family of complex complex manifolds due to Kodaira~\ci{Ko62}. Strengthening \nrc{compactcenter}, we finish the section with the following corollary. This may be regarded as a rigidity    for a simplify connected manifold.
\begin{cor}\label{cptleaves} Keep the assumptions in \rta{flin}. Assume further that $C$ is simply connected.
Then a neighborhood of $C$ in $M$ is biholomorphic to $C\times B^d$ where $B^d$ is the unit ball in $\cc^d$.
\end{cor}
\begin{proof}
We already know that $M$ admits a horizontal foliation by
 \nrc{compactcenter}.
 To show that each leaf is biholomorphic to $C$,
 we may assume that $M=N_C$ and we will use the projection $\pi\colon N_C\to C$.
 We fix $x_0\in C$.
We take a point
$p\in\pi^{-1}(x_0)$ close to $C$. Let $L$ be the (connected) leaf of the foliation containing $p$. Then $L$ intersects each fiber of $N_C$ at a unique point. To verify this, we connect a point in
$x\in C$ to $x_0$ by a continuous path $\gaa$ in $C$ with $\gaa(0)=x_0$ and $\gaa(1)=x$.
By continuation along leaves, we can find a lifted continuous  path $\tilde\gaa$ and
the germ $L^*_\gaa(t)$ at $\tilde \gaa(t)$  of a leaf $L_\gaa(t)$  such that
$\pi(\tilde\gaa(t))=\gaa(t)$. Note that $L^*_\gaa(t'), L^*_\gaa(t)$ are contained in the same leaf
on which $\pi$ is injective, when $t'$ is sufficiently close to $t$. The lifting $\tilde\gaa(1)$  is independent of $\gaa$. Indeed if $\gaa^\theta (a\leq \theta\leq b)$ is a continuous family of paths connecting $x_0$ to $x$. Let $L_{\gaa^\theta}$ be the leaf associated to $\gaa^\theta$. Then  $\widetilde{\gaa^\theta}(t) \in L_{\gaa^{\theta_0}}(t)$ when $\theta$ is sufficiently close to $\theta_0$, as $L_{\gaa^\theta}(0)=L_{\gaa^{a}}(0)$ as a leaf near $p$.

Obviously,  $x\mapsto\tilde\gaa(1)$ gives a biholomorphism from $C$ onto the leaf through
$p$.
And $(x,v)\to \tilde\gaa(1)$
defines a biholomorphisms
from $C\times B$ into $N_C$, where $B$ is a small neighborhood of $0\in\pi^{-1}(x_0)$.
\end{proof}
\setcounter{thm}{0}\setcounter{equation}{0}
\section{The full linearization}
\label{sec:fulllin}

The main purpose of this section is to solve the linearization problem in the general setting (i.e. $N_C$ not necessarily being flat) under a general hypotheses on the existence of bounds to the cohomology equations.
At the end of the section we will illustrate the results with Arnold's examples~\ci{Ar76},
following  computations by Arnol'd~\ci{Ar88}.

We shall devise a Newton scheme to solve 
the linearization of the neighborhood problem.  Let us recall the condition.

\sk
${(L_m)}$ : The neighborhood of $C$
agrees with
 the neighborhood of 
 the zero section of the normal bundle up to order $m$.

\sk

That embedding of $C$ has property $(L_m)$ means that the order  of $(\phi_{kj}^h(h_j,v_j), \phi_{kj}^v(h_j,v_j))$
along
 $v_j=0$  as defined in \re{transitionN} is $\geq m+1$.

Assuming that $(L_m)$ holds. We shall assume either that $H^0(C,TC 
\otimes S^{p}N_C^*)=0$, $2\leq p \leq 2m$ or that 
$N_C$ is flat. According to \rl{centralizer} (c) and (d), the following linearization step in the Newton method
is fulfilled~:

\sk
$(N_m):$ If $\{\Phi_{kj}\}\in L_m$,  then $\{F_k^{-1}\Phi_{kj}F_j\}\in L_{2m}$ for some $\{F_j=I+f_j\}$ with $f_j(h_j,v_j)=O(|v_j|^{m+1})$.


\

%


\subsection{%
Domains for iteration and the Donin condition}

Following \rl{lower-est}
and \rp{regcover}, we shall consider a  family of nested coverings ${\cL U}^{r}=\{U_i^r\}_{i\in I}$ of $C$  with $r_*\leq r\leq r^*$. Let us fix a trivialization of $N_C^*$ (resp. $TC$) over $U_i^{r^*}$ by fixing a holomorphic basis $e_i=(e_{i,1},\dots, e_{i,n+d})$ of $\T_CM$ on $U_i^{r^*}$.

We first define various domains.
Let $
\hat U_j^r:=\var_j(U_j^r)=\Del^r_n$ and $U_{kj}^r:=U_k^r\cap U_j^r$. We have $U_{kj}^r=U_{jk}^r$. Define $\hat U^r_{kj}=
\var_{j}(U^r_{kj})$. Then
\eq{}\nonumber
\var_{kj}(U_{kj}^r)=\hat U_{jk}^r.
\eeq
\noindent
{\bf Donin Condition}. Let $\cL U^r$ be a family of nested covering of $C$  
for $r_*<r<r^*$.
Let $E'=TC$ or $N_C$.
Suppose that there are constants
$D(E'\otimes S^m N_C^*)$ for $m=2,3,\dots$ such that  for all $r', r''$ with $ 
r_*<r''<r'<r
<r^*$ and $r'-r''\leq r^*-r$,
and  all $f\in Z^1(\cL U^{r'}, E'\otimes S^mN_C^*)$ with $f=0$ in $H^1(\cL U^{r'},E'\otimes S^mN_C^*)$, there is a solution $u\in C^0(\cL U^{r''},E'\otimes S^mN_C^*)$  to $\del u=f$ such that
\eq{doninconst}
\max_j\sup |u_j|_{L^\infty(\hat U_j^{r''})}\leq \f{D(E'\otimes S^mN_C^*)}{(r'-r'')^\tau}\max_{k,j} |f_{kj}|_{L^\infty(\hat U_{kj}^{r'})},
\eeq
where $D(E'\otimes S^mN_C^*)$ is independent of $r',r''$ and $f$ and $\tau=\tau(N_C^*)$ is independent of $m$.

\

In what follow, we shall express sections of bundles in coordinates.    It is more convenient to express domains by using the trivialization of the vector bundle $N_C$.
 Recall that  the  $N_C$ has trivializations $N_j$ and transition functions $N_{kj}$.
Let $B_d^r$ be the ball  of radius $r$ in $\cc^d$ centered at the origin.  Thus,  we define
\ga
\label{dom-hV}
\hat V_j^r=N_j( V^r_j)=\hat U_j^r\times B_d^r, \quad  V_{i_0\cdots i_q}^r:=V_{i_0}^r\cap \cdots\cap V_{i_q}^r,\\
\nonumber \hat V_{i_0\cdots i_q}^r:=N_{i_q}(V_{i_0\cdots i_q}^r)\subset  \var_{i_q}(U_{i_0\cdots i_q}^r)\times \Bbb \cc^d, \\  \hat V_{jk}^r=N_{kj}(\hat V_{kj}^r),
\quad
N_{kj}=N_{jk}^{-1} \quad \text{on $\hat  V^r_{kj}$}, \\
N_{ki}N_{ij}=N_{kj}\quad\text{on $\hat  V_{kij}^r$}. \end{gather}
Denote the corresponding domains by $\tilde V_j^r, \tilde V_{kj}^r$ when $N_j$ are replaced by $\Phi_j$. Then we still have the above relations when $N_j,N_{kj}$ are replaced by $\Phi_j,\Phi_{kj}$. We know that  $\Phi_{kj}$ are perturbations of the transition functions $N_{kj}$ of the normal bundle of $C$ in $M$, which are defined on different domains but in the same space.
 We will however  work on domains $\hat V_{kj}^r$ for $\Phi_{kj}$, instead of $\tilde V_{kj}^r$.

With notation of section \ref{section-cocycles}, for $L\geq 1$ and for $r_*\leq r\leq r^*$, we consider a cochain $\{f_{ I }\}\in C^{q+1}(\cL U^{r},\cL O(T_CM\otimes S^L(N_C^*)))$, given by
\beq\label{cochain}\nonumber
 f_{ I }:=f_{ i_0\cdots  i_q}(p)=\sum_{\la=1}^{n+d} \sum_{|Q|=L} f_{{ i_0\cdots  i_q}; Q}^\la (z_{ i_q}(p)) e_{ i_0,\la }(p)\otimes (w_{ i_q}^*(p))^Q
\eeq
where $I =( i_0,\ldots, i_q)\in \cL I^{q+1}$. Recall that   $\hat V^r_{I} = N_{i_q}(V^r_{i_0}\cap\cdots\cap V^r_{i_q})$.
Define 
\eq{fIr}\nonumber
|f_I|_{r}=\sup_{(h_{i_q},v_{i_q})\in \hat V^r_I}|\sum_Qf_{I,Q}(h_{i_q})v_{i_q}^Q|.
\eeq
We also set $|\{f_I\}|_r=\max_I|f_I|_{r}$.

Note that $\hat V_j^r=\hat U_j\times B_d^r$ are product domains. Also,
\ga
\label{initialDomain}\nonumber
\hat U^r_{kj}\times  B_d^{c_*r}\subset \hat V_{kj}^r\subset \hat U^r_{kj}\times B_d^{c^*r}, \quad c_*\leq1\leq c^*.
\end{gather}
Define $B^r_{kj}(h_j)$ to be $\{v_j\in B_d^r\colon t_{kj}(h_j)v_j\in B^r_d\}$.
 The skewed domain  $\hat V_{kj}^r$ can be described as follows:
\ga\nonumber
(h_j,v_j)\in \hat V_{kj}^r \quad  \text{if and only if}\quad   h_j\in \hat U_{kj}^r,\  v_j\in B^r_{kj}(h_j).
\end{gather}

Next, we note that the  $d$-torus  action $(h_j,v_j)\to (h_j,(\zeta_1v_1,\dots, \zeta_{ d}v_d))$ with $\zeta\in(S^1)^{d}$  does not preserve $\hat V_{kj}^r$   when $t_{kj}(h_j)$ is not diagonal. Nevertheless, the $\hat V_{kj}^r$ has a {\it disc structure\,}:
\ga\nonumber
(h_j,\zeta v_j)\in\hat V_{kj}^r, \qquad\forall (h_j,v_j)\in\hat V_{kj}^r, \quad\forall \zeta\in\Del.
\end{gather}
Indeed, suppose that $(h_j,v_j)\in\hat V^r_{kj}$. Then $h_j\in\hat U_{kj}^r$ and $(h_j,v_j)=N_j(p)$ with $p\in V_k^r\cap V_j^r$ and $N_k(p)=(h_k,v_k)\in \hat V_k^r$.  By definition, $\hat V_j^r=\hat U_j\times B_d^r$. Take $\tilde p=N_j^{-1}(h_j,\zeta v_j)$. We have $\tilde p\in V_j^r$ and $N_k(\tilde p)=(h_k,t_{kj}(\zeta v_j))=(h_k,\zeta t_{kj}( v_j))\in \hat U_{kj}^r\times B_d^r$.

Throughout this section, we use
$$
|u_j|_\rho=\sup_{(h_j,v_j)\in\hat V_j^\rho}|u_j(h_j,v_j)|, \quad |u_{kj}|_\rho=\sup_{(h_j,v_j)\in\hat V_{kj}^\rho}|u_{kj}(h_j,v_j)|
$$
where $u_j,u_{kj}$ are functions on $\hat V_j^r$ and $\hat V^r_{kj}$, respectively. We also define
$|\{u_I\}|_\rho=\max_I|u_I|_\rho$.

We now  prove the following.
\le{sch-jet} Let $u_{kj}$ be a holomorphic function on $\hat V_{kj}^r$
 with $r_*<r<\tilde r<r^*$.
  Suppose that
\eq{NE2} 
\hat V^{r_*}_{kj}\neq\emptyset.
\eeq
For $0<\theta<1$
with $\theta r>r^*$, we have
\ga
\label{sch1}\nonumber
|u_{kj}|_{\theta r}\leq \theta^{m} 
|u_{kj}|_r, \quad\text{if $u_{kj}(h_j,v_j)=O(|v_j|^{m})$};\quad
|[u_{kj}]^\ell |_r 
\leq 
 |u_{kj}|_r; 
\\
\sum_{\ell=i}^\infty 
|[u_{kj}]^\ell |_{\theta r} 
\leq \f{\theta^i}{1-\theta} 
|u_{kj}|_r. 
\label{h-part}\nonumber
 \end{gather}
 \ele
\pf{Let $u=u_{kj}$.
 The first inequality  follows from the Schwarz lemma  applied to the holomorphic function $\zeta\to u(h_j,\zeta v_j)$ on the unit disk for fixed $(h_j,v_j)\in\hat V^r_{kj}$.  Note that
$[u]^i(h_j,\zeta v_j)=\zeta^i[u]^i(h_j, v_j)$. Thus the second inequality follows directly by averaging,
$$
[u]^\ell (h_j, v_j)=\f{1}{2\pi i}\int_{\zeta\in\pd \Del}u(h_j,\zeta v_j) \, \f{d\zeta}{\zeta^{\ell+1}}, \quad (h_j,v_j)\in \hat V_{kj}^r.$$
The last inequality follows from the first two inequalities.
}

   %
For the rest of this section, we rename $r$ in the Donin Condition by $\tilde r$ which is fixed now. We will let $r$ vary in $(r_*,\tilde r)$.
\le{nestF} Let $r_*<\theta r<r<
\tilde r<r^*<1$. Fix $k,j\in\cL I$.
Suppose that $(1-\theta^4)r<r^*-\tilde r$ and \rea{NE2} 
holds.
\bpp
\item    We have
\ga\label{nestV}
\dist(\hat V_{j}^{\theta r},\pd \hat V_{j}^r)\geq  r(1-\theta)/{C_0}, \quad \dist(\hat V_{kj}^{\theta r},\pd \hat V_{kj}^r)\geq r(1-\theta)/{C_0},
\end{gather}
  for some constant $C_0$.

\item  Assume further that $\theta^4r>r_*$. There exists a constant
 $C_0^*$ such that if    $F_j=I+f_j$ satisfy
\eq{smallf}
|f_j|_{\theta^2r} 
\leq (1-\theta)r/{C_0^*},
\eeq
 then we have
\ga\label{FjFj}
F_j(\hat V_{j}^{ \theta^2  r})\subset \hat V_{j}^{ \theta   r},\quad F_j(\hat V_{kj}^{ \theta^2 r})\subset \hat V_{kj}^{ \theta   r},\\
\quad F_j^{-1}(\hat V_{j}^{ \theta^4  r})\subset \hat V_{j}^{\theta^3 r},\quad F_jF_j^{-1}=I \ \text{on $\hat V_{j}^{\theta^4 r}$}. \label{Fj-1V}
\end{gather}
\epp
\ele
\begin{proof}
$(a)$  The   $\hat V_j^r$ is the product domain $\hat U^r_j\times B_d^r$. Thus the first inequality in \re{nestV} holds   trivially
 since $\hat U_j^r$ is  a polydisc.
  Note that $\hat V_{kj}^r$ are open sets.  Then
$
\del:=\dist((h,v),(\tilde h,\tilde v))=\dist(\hat V_{kj}^{\theta r},\pd \hat V_{kj}^r)
$
is attained by 
\eq{missed7.8}
(h,v)\in \pd\hat V_{kj}^{\theta r}, \quad (\tilde h,\tilde v)\in\pd\hat V_{kj}^r.
\eeq
If $\tilde h\in\pd\hat U_{kj}^r$, we immediately get $\del\geq\dist(\hat  U_{kj}^{\theta r},\pd\hat U_{kj}^r)\geq (1-\theta)r/C$
 by \rl{lower-est}. Assume that $\tilde h\in \hat U_{kj}^r$. Then by the continuity of the function $t_{kj}$,  $\tilde v$ must be in $\pd B_{kj}^r(
 \tilde h)$.   Otherwise,  both $\tilde h\in \hat U_{kj}^r$ and  $\tilde v\in B_{kj}^r(
 \tilde h)$ are interior points of the two sets, then any small perturbation of $(\tilde h, \tilde v)$   still satisfies 
the second condition in  \re{missed7.8}.
 The last assertion implies that  $(\tilde h, \tilde v)$ cannot be a boundary point and we get a contradiction.   Therefore,  we have
$$
\tilde v\in\pd B_{d}^r \quad \text{or}\quad t_{kj}(\tilde h)\tilde v\in\pd B_{d}^r.
$$
The first case yields $
|\tilde v-v|\geq\dist(B_d^{\theta r}, \pd B_d^r)=(1-\theta)r$. We now consider the second case.
By assumption $t_{kj}$ is holomorphic in $\ov\om$ for a neighborhood $\om$ of $\hat U_{kj}$. Thus there is $\del_*>0$ depending only on $\hat U_{kj}$ such that if $h\in\hat U_{kj}$ and $|\tilde h-h|<\del_*$, then the line segment $\gaa$ connecting $h,\tilde h$ is contained in $\om$.
Suppose that $|\tilde h-h|<(1-\theta)r/{C_1}$ for $C_1$ to be determined so that $(1-\theta)r/{C_1}<\del_*$.  Applying
 the mean-value-theorem to $t_{kj}(\gaa)$   and using $t_{kj}(h)v\in B_d^{\theta r}$, we get
\aln
C_4|\tilde v-v|&\geq |t_{kj}(\tilde h)(\tilde v-v)|\geq  \left||t_{kj}(\tilde h)\tilde v-t_{kj}(h)v)|-|(t_{kj}(\tilde h)-t_{kj}(h))v|\right|\\
&\geq (1-\theta) r-C_5|\tilde h-h||v|\geq(1-\theta)r/2,
\end{align*}
when $C_1$ is sufficiently large. Thus we get
 $\dist(\hat  U_{kj}^{\theta r},\pd\hat U_{kj}^r)\geq (1-\theta)r/C$ as in the first case.
If $|\tilde h-h|\geq(1-\theta)r/{C_1}$, the required estimate is immediate.

$(b)$ Note that $\theta>r_*$. 
By choosing a larger 
$C_0^*$,  \re{FjFj} follows from \re{nestV}  immediately.
We want to find $F^{-1}$. By \re{smallf} and the Cauchy estimate, we know that
\eq{hjfj}
|\pd_{h_j}f_j(h_j,v_j)|+|\pd_{v_j}f_j(h_j,v_j)|\leq C_6/{C_0^*}, \quad\forall (h_j,v_j)\in\hat V_j^{ \theta^3 r}.
\eeq
Note that  $\ov{V_j^r}=\ov{\hat U_j^r}\times \ov{B_d^r}$ is convex. By \re{hjfj} and the fundamental theorem of calculus, we have
 \eq{intrlam}\nonumber
 |f_j(p_1)-f_j(p_0)|\leq  C_7 |p_1-p_0|/{C_0^*}, \quad \forall p_0,p_1\in \hat V^{\theta^3r}_{j}.
 \eeq
Suppose that $C_0^*>2C_7$. Then $F_j\colon\hat V^{\theta^3r}_{j} \to \hat V^{\theta^2 r}_{j}$ is injective, and $T(h_j,v_j)=(\tilde h_j,\tilde v_j)-f_j(h_j,v_j)$ defines a contraction mapping on $\hat V_{j}^{ \theta^3 r}$, if $(\tilde h_j,\tilde v_j)\in \hat V_{j}^{\theta^4 r}$
and $C_0^*$ is sufficiently large. This gives us
 \re{Fj-1V}.
 \end{proof}

In this section, we change notation and let
$$
f_j^\bullet=(f_j^h,f_j^v), \quad
\phi_{kj}^\bullet=(\phi_{kj}^h,\phi_{kj}^v).
$$
\le{nest}  Let $r_*<\theta r<r
<\tilde r<r^*<1$.
Suppose that $\hat V_{kj}$ satisfies \rea{NE2}.
 There exists a constant   $C_1^*$ such that if
\ga
\label{smallphi}
|\phi_{kj}^\bullet|_r 
\leq (1-\theta)r/{C^*_1}
\end{gather}
then we have
\eq{}\nonumber
\Phi_{kj}(\hat V_{kj}^{ \theta  r})\subset \hat V_{jk}^r.
\eeq
 \ele
\begin{proof} 
 Note that $\theta>r_*$.
Since $\Phi_{kj}-N_{kj}=\phi_{kj}^\bullet$ and $N_{kj}(\hat V_{kj}^{\theta r})=\hat V_{jk}^{\theta r}$, the assertion follows from \re{nestV} 
 and \re{smallphi} for sufficiently large $C_1^*$.
%
 \end{proof}

 \pr{nestNewPhi} Let $r_*<
 \theta^7 r<r<
 \tilde r<r^*<1$.
 Assume that $\hat V_{kj}$ satisfies \rea{NE2}. Suppose that
 $\Phi_{kj}=N_{kj}+\phi_{kj}^{\bullet}$  satisfy \rea{smallphi}. 
 Let  $F_j=I+f_j$ satisfy $f_j(h_j,v_j)=O(|v_j|^2)$.

Suppose
   $\tilde\Phi_{kj}=F_k^{-1}\Phi_{kj}F_j=N_{kj}+ {\tilde \phi}_{kj}^\bullet$. 
 There exists a constant $C_2^*$ such that if
 \eq{smallf+}
|\{f_j\}|_{\theta^2r} 
\leq (1-\theta)r/{C_2^*},
\eeq
 and $\tilde\phi_{kj}^\bullet(h_j,v_j)=O(|v_j|^{\widetilde m})$, then
 \ga|\{ {\tilde \phi}_{kj}^\bullet\}|_{\theta^7r} 
 \leq C_2 
  \theta^{\widetilde m}(| \{f_j\}|_{\theta^2 r} 
 +|\{\phi^\bullet_{kj}\}|_r, 
\label{tilphi+}\\
\label{tilphi}
 |\{ {\tilde \phi}_{kj}^\bullet\}|_{\theta^7r}
 \leq  C_2 
  \theta^{\widetilde m}(1-\theta)r.
 \end{gather}
 \epr
\begin{proof}
Let us write $\tilde\Phi_{kj}=N_{kj}+\tilde\phi^\bullet_{kj}$ and $F_k^{-1}=I+g_k$. Thus
\aln
\tilde\phi^h_{kj}&=g_k^h\circ \Phi_{kj} \circ  F_j+\phi^h_{kj}\circ F_j+(
\var_{kj}(I+f_j^h)-
\var_{kj}),\\
\tilde\phi^v_{kj}&=g_k^v\circ  \Phi_{kj} \circ  F_j+ \phi^v_{kj}\circ F_j\\
&\quad +(t_{kj}(h_j+f^h_j)-t_{kj}(h_j))\times (v_j+f^v_j)+t_{kj}(h_j)\times f^v_j(h_j,v_j).
\end{align*}
According to \re{Fj-1V}, we have $F_k(I+g_k)=I$ on $\hat V_k^{\theta^4 r}$. Thus
$
g_k=-f_k\circ F_k^{-1}
$
implies that
\eq{}\nonumber
|g_k|_{\theta^4 r} 
\leq |f_k|_{\theta^3r}. 
\eeq
For $(h_j,v_j)\in\hat V_{kj}^{\theta^6r}$, using $\dist(\hat U_{kj}^{\theta^6 r},\pd\hat U_{kj}^{\theta^5 r})\geq (1-\theta)\theta^5r/{C_0}$, we can obtain $|t_{kj}(h_j+f^h_j(h_j,v_j))-t_{kj}(h_j)|\leq C_3|f^h(h_j,v_j)|$ and
$|\
\var_{kj}(h_j+f_j^h(h_j,v_j))-
\var_{kj}(h_j,v_j)|\leq C_3|f_j(h_j,v_j)|$.
Nesting domains and using \re{smallphi},  \re{smallf+} and hence \re{smallf}, we obtain
by \rl{nestF} in which $r$ is replaced by $\theta^5r$~:
 \gan|\{ {\tilde \phi}_{kj}^\bullet\}|_{\theta^6r} 
 \leq C_4 (|\{f_j\}|_{\theta r} 
 +|\{\phi^\bullet_{kj}\}|_{r}, 
 \\
 | \{{\tilde \phi}_{kj}^\bullet\}|_{\theta^6r} 
 \leq C_4 (1-\theta)r.
 \end{gather*}
Applying Schwarz inequality, we get \re{tilphi+}-\re{tilphi}.
\end{proof}
 When we apply the above to iteration, 
 the new  $\Phi_{kj}$  in the sequence of iteration is defined by
 $$
 (F_k^{(m)})^{-1}(\cdots ((F_k^{(1)})^{-1}\Phi_{kj}F_j^{(1)})\cdots) F_j^{(m)}
 $$
  on $\hat V_{kj}^{r_{m+1}}$ with $F_j^{(m)}(\hat V_{kj}^{r_{m+1}})\subset\hat V_{kj}^{r_m}$.

\

Let us find $[f_j]_{m+1}^{2m}(h_j,v_j)$, a  polynomial of order $\geq m+1$ and of degree $\leq 2m
$ in $v_j$ (holomorphic in $h_j$), such that $\{F_k^{-1}\Phi_{kj}F_j\}\in L_{2m}$ holds  for some $\{F_j=I+[f_j]_{m+1}^{2m}\}$. 

Let us consider the neighborhood written in the new coordinates  $\{F_j\}$. We obtain for $(h_k,v_k)=\hat\Phi_{kj}(h_j,v_j)$:
\begin{eqnarray}
h_k &= & \hat\Phi^h_{kj}(h_j,v_j):=
\var_{kj}(h_j)+\hat\phi^h_{kj}(h_j,v_j),\nonumber\\
 v_k &= &\hat\Phi^v_{kj}(h_j,v_j):= t_{kj} (h_j)v_j+\hat\phi^v_{kj}(h_j,v_j).
\label{transitionN}
\end{eqnarray}
We assume that $\hat\phi^\bullet_{kj}:=(\hat\phi_{kj}^h,\hat\phi_{kj}^v)$ has order $\geq 2m+1$ at $v_j=0$.

Let us write down the horizontal and vertical equations for the linearization problem: $F_k\hat\Phi_{kj}=\Phi_{kj}F_j$.  We obtain the horizontal equation
\begin{align}\nonumber
\var_{kj}(h_j)&+\hat\phi^h_{kj}(h_j,v_j)+f_k^h(
\var_{kj}+\hat\phi^h_{kj}, t_{kj}  (h_j)v_j+\hat\phi^v_{kj})\label{h-conj-equ}
\\ &=
\var_{kj}(h_j+f_j^h(h_j,v_j))+\phi^h_{kj}(h_j+f_j^h,v_j+f_j^v). \nonumber
\end{align}
The vertical equation reads
\begin{align}\nonumber
t_{kj}(h_j)v_j&+\hat\phi^v_{kj}(h_j,v_j)+f_k^v(
\var_{kj}+\hat\phi^h_{kj}, t_{kj}  (h_j)v_j+\hat\phi^v_{kj})
\\ &=
t_{kj} (h_j+f_j^h)(v_j+f_j^v)+\phi^v_{kj}(h_j+f_j^h,v_j+f_j^v).\nonumber
\end{align}
We will interpret the above identity as  power series in $v_j$ with coefficients being holomorphic in $
\var_j(U_k\cap U_j)$. In what follows, degrees or orders of sections are considered w.r.t.  $v_j$ at $v_j=0$.

\subsection{
A Newton method for the full linearization}
For this problem, the two previous equations can be written as
\eq{LfDf}
{\mathcal L}_{kj}(f_j)=\left(0,Dt_{kj}(h_j)f_j^hv_j\right)+\mathcal F_{kj}(f_j),
\eeq
where ${\mathcal L}_{kj}(f_j)$ stands for $({\mathcal L}_{kj}^h(f^h_j),{\mathcal L}_{kj}^v(f^v_j))$ as defined by \re{h-cohomo-op},\re{v-cohomo-op}:
\al
{\mathcal L}_{kj}^h(f^h_j)&:=f_k^h(
\var_{kj}(h_j),t_{kj}(h_j)v_j)  -s_{kj}(h_j)f_j^h(h_j,v_j),\label{h-cohomo-op-L}\\
{\mathcal L}_{kj}^v(f^v_j)&:=f_k^v(
\var_{kj}(h_j),t_{kj}(h_j)v_j)  -t_{kj}(h_j)f_j^v(h_j,v_j).\label{h-cohomo-op-L+}
\end{align}
Recall that $s_{kj}(h_j)=D
\var_{kj}(h_j)$ is the Jacobian matrix of $
\var_{kj}$.
 Furthermore, we have the horizontal error term
\begin{eqnarray}
\mathcal F_{kj}^h(f_j) &:=&\phi^h_{kj}(h_j+f_j^h,v_j+f_j^v)  -\hat\phi^h_{kj}\label{h-conjug-partial}\\
&&+\,\left(f_k^h(
\var_{kj}, t_{kj}  (h_j)v_j)-f_k^h(
\var_{kj}+\hat\phi^h_{kj}, t_{kj} (h_j)v_j+\hat\phi^v_{kj})\right)\nonumber\\
&&+\
\var_{kj}(h_j+f_j^h(h_j,v_j))-
\var_{kj}(h_j)-D
\var_{kj}(h_j)f_j^h(h_j,v_j),\nonumber
\end{eqnarray}
as well as the vertical error term
\begin{eqnarray}
\mathcal F_{kj}^v(f_j) &:=& \phi^v_{kj}(h_j+f_j^h,v_j+f_j^v)-\hat\phi^v_{kj} +Dt_{kj}(h_j)f_j^hf^v_j
\label{v-conjug-partial}\\
&&+\,\left(f_k^v(
\var_{kj}, t_{kj} (h_j)v_j)-f_k^v(
\var_{kj}+\hat\phi^h_{kj}, t_{kj} (h_j)v_j+\hat\phi^v_{kj})\right)\nonumber\\
&&+\left(t_{kj}(h_j+f_j^h(h_j,v_j))- t_{kj}(h_j)-Dt_{kj}(h_j)f_j^h\right)(v_j+f_j^v).\nonumber
\end{eqnarray}

We collect $2m$ jets from \re{LfDf}, \re{h-conjug-partial}, \re{v-conjug-partial}. Since $f_j=O(m+1)$ and $\hat\phi_{kj}^{\bullet}=O(2m+1)$,  this gives us 
\al\label{Hsol}
[(\del^h f^h)_{kj}]^{\leq 2m}&=-[\phi^h_{kj}]^{\leq 2m},\\
[(\del^v f^v)_{kj}]^{\leq 2m}&=- Dt_{kj}(h_j)[f_j^h]^{\leq 2m-1}v_j- [\phi_{kj}^v]^{\leq 2m}.
\label{Vsol}
\end{align}
Under formal assumptions,  according to \rl{centralizer}~{$(c)$},
equations \re{Hsol}-\re{Vsol} have a solution $([f_j^h]_{m+1}^{2m},[f_j^v]_{m+1}^{2m})$.

We first consider the case that $H^0(C, \oplus_{k=2}^{2m}TC 
\otimes S^k(N_C^*))=0$. Then,  for any $r_*<r''<r'<
\tilde r< r^*$ with
$$
r''=\theta r'=\theta^2r,\quad
r'-r''< r^*-
\tilde r,
$$
 the solution to \re{Hsol} is unique and by \rt{donin-sol-cohom} that unique solution satisfies the estimate
\eq{Hsole}
|\{[f^h_k]^{l}\}|_{ r'} 
\leq \f{ 
D(TC\otimes S^l(N_C^*))}{
(r-r')^\tau}
|\{[\phi^h_{kj}]^{l}\}|_{r}, 
\quad l=m+1,\ldots, 2m.
\eeq
In particular, $\{[f^h_k] 
_{m+1}^{ 2m}\}$ has been determined. The solvability of \re{Vsol} and \rt{donin-sol-cohom}
imply that we can find a solution $\{[f^v_k]^{2m}_{m+1}\}$ such that  for $l=m+1,\ldots,2m$,
\eq{Vsole}
|\{[f^v_k]^{l}\}|_{  r''} 
\leq \f{D(N_C\otimes S^l(N_C^*))}{(r'-r'')^\tau}\left\{c\f{D(TC\otimes S
^{l-1}(N_C^*))}{(r-r')^\tau}|\{[\phi^h_{kj}]^{l-1}\}|_{r} 
+|\{[\phi^v_{kj}]^{l}\}|_{r} 
\right\}.
\eeq
Here $c$  depends only on the $Dt_{kj}$ over the initial covering.

If  $H^0(C, \oplus_{k=m+1}^{2m}TC 
\otimes S^k(N_C^*))\neq0$, we are in the flat case, that is $Dt_{kj}=0$. Thus, we can find a solution $\{[f^v_k]^{2m}_{m+1}\}$ such that for $l=m+1,\ldots,2m$,
\eq{}
|\{[f^v_k]^{l}\}|_{r''} 
\leq \f{D(N_C\otimes S^l(N_C^*))}{(r'-r'')^\tau} |\{[\phi^v_{kj}]^{l}\}|_{r}. 
\label{estim-v-linb}
\eeq


Let us set

\al\label{om2m}
&D_*({2m}):= 1+
\displaystyle{\max_{2\leq l\leq 2m}}\left\{(1+c  
K(TC\otimes S^{l-1}(N_C^*))) D(N_C\otimes S^l(N_C^*)\right\}.
\end{align}


Hence, in any case, estimates \re{Hsole}-\re{estim-v-linb} lead to
\eq{cohome}\nonumber
|\{[f^{\bullet}_k]^{l}\}|_{\theta^2   r} 
\leq \f{C_1D_*(2m)}{
(
r-\theta^2r)^{2\tau}}|\{[\phi^{\bullet}_{kj}]^{l}\}|_r 
\eeq
for all
$\theta$ and $r$ satisfying $r_*\leq \theta^2r< r
<\tilde r<r^*$ and all $m+1\leq l\leq 2m$. Assume further that $\theta^6r>r_*$
and $(1-\theta^7)r<r^*-\tilde r$. We obtain, by \rp{nestNewPhi} with $\tilde m=2m+1$
\ga\label{tilphi1}\nonumber
 | {\hat \phi}_{kj}^\bullet|_{\theta^7r} 
 \leq \f{C_1D_*({2m}) \theta^{2 m +1
 }}{
 (r-\theta^2r)^{2\tau}} | {  \phi}_{kj}^\bullet|_r
 \leq\theta^{2m+1} (1-\theta)r/C_0,
\end{gather}
provided
\ga
\label{cond1}
| \{\phi^{\bullet}_{kj}\}|_r 
\leq(1-\theta)r/C_0,\\
\label{cond2}
\f{D_*({2m})}{
(r-\theta^2 r)^{2\tau}}|\{\phi^{\bullet}_{kj}\}|_r 
\leq(1-\theta)r/C_0.
\end{gather}
Note that condition \re{cond1} follows from \re{cond2} as
  $D_*(\ell)\geq1$.

Rename $\Phi_{kj}, \phi_{kj}^\bullet, F_j,
f^\bullet_j,\hat\Phi_{kj}, \hat\phi^\bullet_{kj}$ respectively as $\Phi_{kj}^{(0)}, \phi_{kj}^{(0)}, F_j^{(0)},f_j^{(0)}, \Phi_{kj}^{(1)}, \phi_{kj}^{(1)}$. Thus $\Phi_{kj}^{(1)}=(F_k^{(0)})^{-1}\Phi_{kj}^{(0)}F_j^{(0)}$. Repeating this formally, we obtain
\gan
\Phi_{kj}^{(\ell+1)}=(F_k^{(\ell)})^{-1}\Phi_{kj}^{(\ell)}F_j^{(\ell)},
\quad F_j^{(\ell)}=I+f_j^{(\ell)}, \quad \Phi_{kj}^{(\ell+1)}=N_{kj}+\phi_{kj}^{(\ell+1)}.
\end{gather*}
Set $r_{\ell+1}=\theta_\ell^7r_\ell$ and $m_\ell=2^\ell$.
We also have
\ga\label{boundedF}
 F_j^{(\ell)}( \hat V_{j}^{r_{\ell+1}})\subset \hat V_{j}^{r_\ell},\\
\label{tilphi+1}
 |   \phi_{kj}^{(\ell+1)}|_{r_{\ell+1}} 
 \leq\theta_\ell^{2m_{\ell}+1} (1-\theta_\ell)r_\ell/C_0
 \end{gather}
provided
\ga
\label{cond2r}
r_*\leq  \theta_\ell^7r_k<1, \quad 0<\theta_k<1;\\
\label{cond2-el}
 \f{C_1D_*({2m_\ell)}}{
 (r_\ell-\theta_\ell^2r_\ell)^{2\tau}}|\{\phi^{(\ell)}_{kj}\}|
 _{r_\ell} 
 \leq (1-\theta_\ell)r_\ell/C_0.
\end{gather}

To set parameters, we follow Russmann~\ci{Ru02}; see ~\ci{Br71, St00} for different choices of parameters. As in \ci{Ru02}, we now use
an addition assumption that
\eq{om1ell}
D^*_\ell 
 \geq 
 \ell,\quad \ell\geq1.
\eeq
Indeed, when 
 $\tilde D_*(k)=\max(D_*(k),k)$  
replaces with 
$D_*(k)$,
the sequence
$D_*(k)$ 
still
increases and $\sum 2^{-k}\log D_*({2^k} )
$   converges.   For a constant $C_* \geq1$ to be determined later, define
\gan
m_\ell=2^{\ell_0+\ell}, \quad
r_{\ell+1}=\theta_\ell^7r_\ell, \quad r_0=1,\\
1-\theta_\ell=\del_\ell,\quad \del_\ell =
C_*\f{\log D_* 
({m_{\ell+2}})
}{m_{\ell+2}}.
\end{gather*}
Note that in~\cite[Lemma 6.2]{Ru02} and \ci{Br71, St00},  $\om(m_{\ell+1})$ is used to define $\del_\ell$.
Shifting the index by $1$, we use $D_*(m_{\ell+2})$ to simplify the argument.
We can find $\ell_0=\ell_0(C_*)$ such that
 $0<\theta_\ell<1$ for all $\ell$ and furthermore
%
\aln
\prod_{\ell=0}^\infty \theta_\ell^7
=\prod_{\ell=0}^\infty (1-\del_\ell)^7\geq\exp\left\{
-
\sum_{\ell=0}^\infty\f{7C_*}{2} \f{
\log D_*(m_{\ell+2})}{m_{\ell+2}}\right\}.
\end{align*}
Since $\sum 2^{-k}
\log  
D_*(
{2^k})<\infty$,
the latter is larger than $r_*$, provided $\ell_0>\ell_0(C_*)$.
Inductively, we want to show that if $\re{cond2-el}_{\ell}$ holds, then $\re{cond2-el}_{\ell+1}$ also holds.  Indeed, with $\re{cond2-el}_\ell$, we can use $\re{tilphi+1}_{\ell+1}$ to obtain
\aln
& \f{C_1D_*({m_{\ell+2})}}
{ 
(r_{\ell+1}-\theta_{\ell+1}^2r_{\ell+1})^{2\tau}}|\{\phi^{\ell+1}_{kj}\}|
_{r_{\ell+1}} 
\times \f{C_0}{(1-\theta_{\ell+1})r_{\ell+1}} \\
 &\hspace{12ex}\leq\f{D_*(m_{\ell+2})\theta_\ell^{2m_{\ell}-6}}{(r_{\ell+1}-\theta_{\ell+1}^2r_{\ell+1})^{2\tau}}
 \times \f{1-\theta_\ell }{ 1-\theta_{\ell+1} }\qquad \text{(by \re{tilphi+1})}\\
 &\hspace{12ex} \leq \f{C_2D_*({m_{\ell+2}}) \theta_\ell^{2m_{\ell}-6}}{(1-\theta_{\ell+1})^{2\tau+1}}  = \f{C_2 D_*({m_{\ell+2}})
 (1-\del_\ell)^{2m_{\ell}-6}}{\del_{\ell+1}^{2\tau+1}}.
\end{align*}
We need to check that the last expression is less than one by using logarithm. Note that
$$
\log (1-\del)<-\del,\quad \forall \del\in(0,1).
$$
Therefore,
\aln
&\log\f{C_2 D_*({m_{\ell+2}})(1-\del_\ell)^{2m_{\ell}-6}}{\del_{\ell+1}^{2\tau+1}}<\log C_2-(2m_\ell-6)\del_\ell
+\log D_*({m_{\ell+2}})-(2\tau+1)\log\del_{\ell+3}\\
&\quad =\log C_2-(2m_\ell-6)C_*  \f{\log D_*({m_{\ell+2}})}{m_{\ell+2}}
+\log D_*({m_{\ell+2}})-(2\tau+1)\log\left(C_* \f{\log D_*({m_{\ell+3}})}{m_{\ell+3}}\right)\\
&\quad =\left\{\log C_2
- \f{(2m_\ell-6)C_*}{3} \f{\log D_*({m_{\ell+2}})}{m_{\ell+2}}\right\}
+\left\{ \log D_*({m_{\ell+2}})-\f{(2m_\ell-6)C_*}{3}  \f{\log D_*({m_{\ell+2}})}{m_{\ell+2}}\right\}
\\
&\quad\quad +\left\{-\f{(2m_\ell-6) C_*}{3} \f{\log D_*({m_{\ell+2}})}{m_{\ell+2}}- (2\tau+1)\log\left(C_* \f{\log D_*({m_{\ell+3}}}{m_{\ell+3}}\right)\right\}.
\end{align*}
When  $\ell_0$ is sufficiently large,  then $m_{\ell+2}>24$. This implies that if $C_*>12$,  the sum in each of first two  braces is negative.
  Since $\log$ increases, we have by \re{om1ell}
\gan
-\log D_*({m_{\ell+3}})\leq \log\f{1}{m_{\ell+3}},
\\
-\log\left(C_*\f{\log D_*({m_{\ell+3}})}{m_{\ell+3}}\right)\leq -\log\left(-C_*\f{\log\f{1}{m_{\ell+3}}}{m_{\ell+3}}\right).
\end{gather*}
With   $m_\ell>6$, the difference in the last brace is bounded above by
\aln
\f{(2m_\ell-6) C_*}{3} \f{\log\f{1}{m_{\ell+2}}}{m_{\ell+2}}- (2\tau+1)\log\left(C_* \f{\log m_{\ell+3}}{m_{\ell+3}}\right)\leq \left(-\f{1}{12}C_*+2\tau+1\right)\log m_{\ell+2},
\end{align*}
which is negative when $C_*>24\tau +12$. We have determined $C_*$.  This allows us to determine
 $\ell_0(C_*)$ so that $0<\theta_\ell<1$ and $\prod_{\ell=0}^\infty\theta_\ell^7>r_*$.
 Therefore, $\re{cond2-el}_\ell$ holds if it holds for initial value $\ell=0$.
Using a dilation $v_j\to\e v_j$ for $\e>0$, we may replace $\Phi_{kj}(h_j,v_j)$ by
$(
\var_{kj}(h_j)+\phi_{kj}^h(h_j,\e v_j),t_{kj}(h_j)v_j+\e^{-1}\phi_{kj}^v(h_j,\e v_j))$. This yields $\re{cond2-el}_0$ when $\e$ is sufficiently small, as $\phi_{kj}^\bullet(h_j,v_j)=O(|v_j|^2)$.

To finish the proof, we set $\Psi_j^{(\ell)}:=F_j^{(0)}\circ  \cdots \circ F_j^{(\ell)}$. We have
\eq{}\nonumber
\Psi_j^{(\ell)}(\hat V_{j}^{r_{\ell+1}})\subset\hat V_{j}^{r_{\ell}},\quad
\Psi_j^{(\ell+1)}(h_j,v_j)-\Psi_j^{(\ell)}(h_j,v_j)=O(|v_j|^\ell).
\eeq
Consequently,  the sequence $\Psi_j^{(\ell)}$ is  bounded in $\hat V_j^{r_\infty}$.
Fix $0<\theta<1$. By the Schwarz lemma, we  get
$$
\sup_{\hat U_j^{r_\infty}\times B_d^{\theta r_\infty}}|\Psi_j^{(\ell+1)}-\Psi_j^{(\ell)}|\leq C\theta^\ell.
$$
Therefore,  of $\Psi_j^{(\ell)}$ converges uniformly on $\hat U_j
^{r_\infty}\times B_d^{\theta r_\infty}$   to  a holomorphic mapping $\Psi_j^\infty$. Then  $F:=N_j^{-1}\Psi_j^\infty\Phi_j$ is well defined. Indeed, $N_k^{-1}\Psi_k^\infty\Phi_k=N_j^{-1}\Psi_j^\infty\Phi_j$ is equivalent to $\Psi_k^\infty(\Phi_k\Phi_j^{-1})=(N_kN_j^{-1})\Psi_j^\infty$.   Since $\Psi_j^\infty$ are tangent to the identity, they are germs of biholomorphisms.
 Therefore,
 $F$ linearizes a small neighborhood of $C$ in $M$.

Therefore, we have proved the following full linearization result.
\begin{thm}
Let  a neighborhood of the compact manifold $C$ in $M$   
be   equivalent to a neighborhood of the zero section  of normal bundle $N_C$ of $C$ in $M$ 
by a formal holomorphic mapping which is tangent to the identity and preserves the splitting of $T_CM$.
 Assume that
  $H^0(C,TC\otimes S^{\ell}(N_C^{*}))=0$ for all $\ell>1$
  or   that
    the normal bundle $N_C$ is flat.
   If
  $\{D_*({2^k})\}$ defined by \rea{doninconst} and \rea{om2m} 
satisfies
\beq\label{bruno-cond}
\sum_{k\geq 1}\frac{\log D_*({2^{k+1}})}{2^{k}}<+\infty,
\eeq
  there is a neighborhood of the compact manifold $C$ in $M$ that is biholomorphic to a neighborhood of the zero section of normal bundle of $C$ in $M$.
\end{thm}
When $C$ is affine and $N_C$ is flat, the formal equivalence assumption can be relaxed by assuming that the neighborhoods
 are equivalent under a formal biholomorphisms fixing $C$ pointwise.
This follows from \rl{dfj-} $(c)$.
%
%
%
%
%
%
%

  We now present   
   two
   examples to illustrate the results in this paper.

\subsection{
An example of Arnol'd}
\label{arnoldcomp}

 This is originally studied by Arnold~\cite{Ar76}, ~\cite[\S 27]{Ar88}  for linearization of a neighborhood.
See also Ilyashenko-Pyartli~\ci{IP79} for linearization for flat tori
in higher dimensions.

\begin{exmp}\cite[\S 27]{Ar88}. Let $C$ be defined by identifying points in $\cc$ via
$$
h=0\mod (2\pi,2\om),\quad h\in\cc,
$$
where $\om=a+ib$ with $b>0$ and $a\geq0$.
Consider domains in $C$ defined by parallelograms
\gan
U_1=P(-r\pi-r\om, (1+r)\pi-r\om,  (1+r)\pi+(1+r)\om, -r\pi+(1+r)\om ) \\ U_4=U_1+\pi, \quad
U_3=U_4+\om, \quad U_2=U_3-\pi.
\end{gather*}
Suppose that $r>0$ is sufficiently small. Then $U_i\cap U_j$ has two connected components $U_{ij,0}$ and $ U_{ij,1}$ with
$$
U_{14,1}=U_{14,0}-\pi, \quad U_{34,1}=U_{34,0}-\om, \quad U_{23,1}=U_{23,0}-\pi,\quad U_{12,1}=U_{12,0}-\om.
$$
Let $\hat U_j=U_j$ and $\hat V_j=\hat U_j\times\Del_\del$. Define $M=\cup \hat V_j/\sim$, $ V_j=\{[x]\colon x\in \hat V_j\}$, $\Phi_j\colon V_j\to \hat V_j$ and the transition functions $\Phi_{kj}$   on $V_{kj}=V_k\cap V_j$ of $M$   as follows. Let
\eq{arnold}\nonumber
f(h,v)=(h+2\om+vb(h,v),\la v(1+va(h,v))), \quad |\IM h|<\del
\eeq
where $a,b$ are $2\pi$ periodic holomorphic function in $h$.
Define
\ga\label{Phi120}
 \Phi_{12,0}=\I, \quad \Phi_{43,0}=\I,\quad \Phi_{12,1}=f|_{\hat V_{12}}, \quad \Phi_{43,1}=f|_{\hat V_{43}},\\
 \label{Phi120+}
\Phi_{14}=\I, \quad \Phi_{23}=\I,\\
\label{Phi120++}
\Phi_{13,0}=\I, \quad \Phi_{13,1}=f|_{\hat V_{13,1}},\quad
\Phi_{42,0}=\I, \quad \Phi_{42,1}=f|_{\hat V_{42,1}}.
\end{gather}
The linearization of a neighborhood of $C$ in $M$ is equivalent to $G_k^{-1}\Phi_{kj}G_j=\hat\Phi_{kj}$ where $\hat\Phi_{kj}$ are constructed as above by replacing $f$ with $\hat f$ defined by
$$\hat f(h,v)=(h+2\om,\la v).
$$
 Thus $TM
    $ has transition functions:
$$
\hat\Phi_{14}=\I, \quad \hat\Phi_{23}=\I, \quad \hat\Phi_{12,0}=\I, \quad \hat\Phi_{43,0}=\I, \quad \hat\Phi_{12,1}=\hat f|_{\hat V_{12}}, \quad \hat\Phi_{43,1}=\hat f|_{\hat V_{43}}.
$$
  Then we have $g:=G_1=G_4$ on $\hat V_1\cap\hat  V_4$, $g:=G_2=G_3$ on $\hat V_2\cap \hat V_3$,    $g:=G_1=G_2$ on $\hat V_{12,0}$ and $g:=G_3=G_4$ on $\hat V_{34,0}$. In other words, $g$ is $2\pi$ periodic and defined on $-\del\IM\om<\IM h<2(1+\del)\IM\om$. The cohomology equation is
   reduced to  $G_1^{-1}\Phi_{12}G_2=\hat\Phi_{12}$ and $G_4^{-1}\Phi_{43}G_3=\hat\Phi_{43}$. Equivalently, we need to solve
\eq{g-1fg}
g^{-1}fg=\hat f.
\eeq
Assume that $f$ has been normalized so that
$$
va(h,v)=v^na_n(h)+O(n+1), \quad vb(h,v)=v^nb_n(h)+O(n+1),\quad n=1,2,\dots.
$$

For the purpose of illustration, we will only restrict to a {\it special} unitary line bundle case where $|\la|=1$. 
    Then by the non-resonance condition that $\la$ is not a root of unity,  we may assume that as
  in~\cite[p.~211]{Ar88}
  $$g(h,v)=(h+v^nB_n(h),v(1+v^n A_n(h))+O(n+1).$$
This leads to decoupled equations of the form
\ga\label{Aan}\nonumber
\la^nA_n(h+2\om)-A_n(h)=-a_n(h),\\
\la^nB_n(h+2\om)-B_n(h)=-b_n(h).
\label{Bbn}
\end{gather}
Note that $a_n,b_n$ are holomorphic in $|\IM h|<\del$ and we are seeking a solution on a large strip
$$
-\del'<\IM h<\IM\om+\del'.
$$
In Fourier coefficients $a_{n,\ell}$ and a  non-resonant condition, the Fourier coefficients of $A_n$ are given by
$$
A_{n,j}=\f{a_{n,j}}{\la^ne^{2\om j\sqrt{-1}}-1}.
$$
Assume that $a_n$ are holomorphic  and $2\pi$ periodic in $h$ for $S_\del\colon |\IM h|<\del$. Suppose that
\eq{Arnold-sd}\nonumber
|\la^ne^{2j\om\sqrt{-1}}-1|\geq c|\la^n-1|. \eeq
 Then
\gan
|A_{n,j}|\leq
\f{C}{|\la^n-1|} |a_n|_{L^2(S_\del)}e^{-|j|\del},\\
|A_{n,j}e^{jh}|\leq  
\f{C}{|\la^n-1|} |a_n|_{L^2(S_\del)}e^{-|j|(\del-\del')},\quad -\del'<\IM h<\IM\om+\del'.
\end{gather*}
Furthermore, we can verify that
\eq{L2est}\nonumber
|A_n|_{L^2(S_{\del'})}\leq \f{C}{(\del-\del')|\la^n-1|}
|a_n|_{L^2(S_{\del})}.
\eeq
Note that $t_{kj}$ are locally constant with values $1,\la,\la^{-1}$.

Therefore, we have verified
 $$
 D((T_C\oplus N_C)\otimes S^nN_C^*)\leq \f{C}{|\la^n-1|}.
$$
 By \rl{SD-p}, we get an estimate with equivalent bounds (up to a scalar) but in the original domain, i.e. without shrinking domains.

Strictly speaking, the above covering $\{U_j^r\}$ has non smooth boundary. The intersection is non-transversal either. However, this covering can be easily modified to get a generic covering defined early, replacing $\hat U_j$ by smooth strictly convex domains $\hat U_j$ and then replacing $\hat U_j$ by $\hat U_j+c_j$ for suitable small constants.
\end{exmp}
\subsection{Counter-examples}
We now  show that a certain small-condition is necessary to ensure the vertical and full linearizations. We will achieve this by establishing a connection between the classical linearization problem for germs of one-dimensional holomorphic mappings and the vertical linearization of foliated neighborhood of an elliptic curve.

We keep the notation in subsecton~\ref{arnoldcomp}. Let us start with a power series
\eq{bv}
 a(h,v)=\sum_{n\geq2}a_nv^n :=a(v).
\eeq
Set $b(h,v)=0$. Then we have a neighborhood of $C$ associated to
\eq{fhv}
f(h,v)=(h+2\om,\la v+ a(v)).
\eeq
Since the vertical part of the transition functions depends only on $v$, then $M$ is already admits a horizontal foliation with center $C$ being compact.
\pr{link}Let $M_{\la, \om, a}$ be neighborhood of $C$ defined by transition functions $\Phi_{kj}$ given by \rea{Phi120}-\rea{Phi120++} where $f$ is given by \rea{bv}-\rea{fhv}. Suppose that $\la,\om$ satisfy the nonresonant condition
\eq{nonres-c}
\la^ne^{2j\om\sqrt{-1}}-1\neq0, \quad \quad n=2,3,\dots, j\in\zz.
\eeq
Then $M
_{\la,\om,a}$ is vertically $($resp. formally$)$ linearizable
by a mapping tangent to the identity
 if and only if the germ of holomorphic mapping $\var(v)=\la v+ a(v)$ is holomorphically $($resp. formally$)$ linearizable.
\epr
\begin{proof}Suppose that $M$ is vertically linearizable by a holomorphic mapping that is tangent to the identity. By \rp{2-prob}, it is vertically linearization by a mapping
 $G_j$ such that
	$$
	G_j(h_j,v_j)=(h_j,v_j+O(|v_j|^2)).
	$$
By the non-resonant condition \re{nonres-c}, we can verify that
	\re{g-1fg} is equivalent to that the  $g$ in \re{g-1fg}
	has the form $g(h,v)=(h,\psi(v))$ and $\var$ is linearized by $\psi$.
\end{proof}
The existence of non-holomorphically
linearizable $\var$ is well-known. By theorems of  Bruno~\ci{Br71} and Yoccoz~\cite{Yoccoz-ast},  \rp{link} shows that   $M_{\la,\om,\hat a}$ with $\hat a(v)=v^2$ is vertically linearizable and hence linearizable if and only if $\la$ is a Bruno number, that is $$\sum_{k\geq 1}\frac{\log\max_{2\leq j\leq 2^k}|\lambda^j-1|^{-1}}{2^k}<+\infty.$$
\subsection{
A foliation example}
 Here we specialize Ueda's theory for elliptic curves.
 Let us first discuss the Fischer norms and Bergman norm when the $N_C$ is unitary.
Let us recall two formulae from Zhu~\ci[p.~22]{Zh05}:
\gan
\int_{\pd B_r^d}|z^Q|^2\, d\sigma_d=\f{(d-1)!Q!}{(|Q|+d-1)!}r^{2d-1+2|Q|},\\
\quad \int_{B_r^d}|z^Q|^2\, dV_d=\f{d!Q!}{(|Q|+d)!}r^{2|Q|+d}.
\end{gather*}
Therefore, there is a precise asymptotic behavior of Fischer norm and the Bergman norm:
\ga
c_d\|g\|^2_{L^2(B_r^d)}\leq |g|^2_{f,r}\leq C_d\|g\|^2_{L^2(B_r^d)}, \quad 1/4<r<4.
\end{gather}
We also have Bergman's  inequality for $L^2$ holomorphic functions~\ci[p.~189]{GR04}:
\ga
|f|_{\infty, \hat V_j^{(1-\theta)r}}\leq \f{C_d}{(\theta r)^d}\sup_{h_j}|f(h_j,\cdot)|_{L^2(B_r^d)}, \\
\sup_{h_j}|f(h_j,\cdot)|_{L^2(B_r^d)}\leq C_d|f|_{\infty,B_r^d}, \quad 1/4<r<4.
\end{gather}
In general,  we get
\ga
|\phi_{kj}^\bullet
|_{L^\infty( \hat V_{kj,(1-\theta)r})}\leq \f{C_d}{(\theta r)^d}\sup_{h_j}|\phi_{kj}^\bullet(h_j,\cdot)
|_{L^2(B_{kj,r}^d(h_j))}, \\
 \sup_{h_j}|\phi_{kj}^\bullet
|_{L^2(B_{kj,r}^d(h_j))}\leq C_d|\phi_{kj}^\bullet|_{L^\infty(\hat V_{kj}^r)}, \quad 1/4<r<4.
\end{gather}
Note that when $t_{kj}$ are unitary, the skewed domain $\hat V_{kj}^r$
defined in \re{dom-hV} are actually product domains
$$
\hat V_{kj}^r=\hat U_{kj}^r\times B_d^r.
$$
Therefore, the Fischer norm and Bergman norm bound each other with constants
depending only on $\theta$ and $d$.
We can fix $\theta$ too by applying \rl{SD-p} as we did in sections~\ref{sec:verlin} and~\ref{sec:majorlin}.
Therefore, any estimate of cohomology equations in  Fischer norms has a counter part in super norm on the unit ball in $\cc^d$ and vice versa.


 Note that the small divisors condition
\eq{Ssd}
|\la^n-1|\geq Cn^{-\ta}, \quad n=1,2,\dots
\eeq
for some constants $C,\tau$ is equivalent to Ueda's condition in terms of  $\dist(N_C^n,1)$ for the foliation problem when $C$ is an elliptic curve of type zero. In this case the linearized equation corresponds to equation~\re{Bbn} we can take the small divisor
$1/K_*(N_C\otimes S^n N_C^*)$ to be $|\la^n-1|$.

Finally, we should mention that the assumption $\eta_m\leq L_0L^m$ is satisfied under Siegel's small divisor condition $|\la^n-1|\geq Cn^{-\tau}$  by a  method of Siegel; see Ueda~\ci{Ue82} for the vertical linearization
 problem. It is also satisfied under    the Bruno condition~\ci{Br71}
 which is a condition weaker than \re{Ssd}. For the details, we refer to \ci{Br71,Po86}.
%
%


\appendix
\setcounter{thm}{0}\setcounter{equation}{0}
\section{$L^2$  bounds of cohomology solutions and small divisors
}\label{sec1emb}

\subsection{A question of Donin}

 Let $E$ be a holomorphic vector bundle on a compact complex manifold $C$.
The main purpose of this section is to obtain $L^2$ and sup-norm bounds for the cohomology equation
\eq{deluf}
\del u=f
\eeq
where $f\in Z^1(\cL U,\cL O(E))$ and $\cL U$ is a suitable covering of $C$. Our goal is to show that if $f=0$ in $H^1(C,\cL O(E))$, then there is a solution $u$ such that
\eq{uCf}
\|u\|_{\cL U}\leq K(E)\|f\|_{\cL U}.
\eeq
  Here $\|\cdot\|_{\cL U}$ is the $L^2$-norm for cochains of the covering $\cL U$. The main assertion is that the solution $u$ admits estimate on the {\em original} covering $\cL U$ without any refinement, which is important to the application in this paper. For this purpose,  we will choose the covering $\cL U$ which consists of biholomorphic images of the unit polydisc, which are in the general position.  The question on the existence of such an estimate and solutions was raised by Donin who asked the general question if $\cL O(E)$ is replaced by a coherent analytic sheaf $\cL F$ on $C$ and $f$ is any $p$-cocycle,
  with $p>0$,  of a covering $\cL U$~\ci{Do71}. 
  The result in this appendix provides an affirmative answer to Donin's question  
   for $p=1$ and
   the sheaf of holomorphic sections of a holomorphic vector bundle. Furthermore, we will introduce the small divisor for \re{deluf} in \re{uCf}.  Although some of results in this appendix can be further developed for a general setting, we limit to the case of $H^1(C,\cL O(E'\otimes E''))$; this suffices  applications in this paper. One may take $E''$ to be the trivial bundle to deal with a general vector bundle $E$.  In the applications we have in mind, $C$ is embedded into a complex manifold $M$ and we will take $E''$ to be symmetric powers $\operatorname{Sym}^\ell N_C^*$
  of $N_C^*$, the dual of the normal bundle of $C$ in $M$.
  In this paper, $S^\ell E$ denotes the symmetric power $\operatorname{Sym}^\ell E$ of a vector bundle $E$ over $C$.
    We are mainly concerned with how various bounds depend on $\ell$ as $\ell\to \infty$ when we employ the important Fisher metric on $S^\ell N_C^*$  for unitary the normal bundle $N_C$. This will be crucial in our applications.

  To prove \re{uCf}, we will first use the original estimate of Donin~\ci{Do71}, without solving the cohomology equation.  This  serves as a {\it smoothing decomposition} in the sense of Grauert~\cite{GR04} by expressing
  \eq{fgU}
  f=g+\del u
  \eeq
where $g$ is defined on a larger covering while $u$ is defined on a shrinking covering.
   We will then combine with the proof of finiteness theorem of cohomology groups from Grauert-Remmert~\cite{GR04} to refine the decomposition \re{fgU} by expressing $g$ in a base of cocycles. Finally, we will obtain \re{uCf} by avoiding shrinking of covering. This last step is motivated by a method of Kodaira-Spencer and Ueda~\ci{Ue82}. We take a different approach by an essential use of the uniqueness theorem. This allows us to introduce the {\it small divisors} in \re{uCf} to the cohomology equation \re{deluf}.

\subsection{Bounds of solutions of cohomology equations}
\label{secA.2}
We now start to introduce nested coverings of $C$. This will be an essential ingredient of the small divisors for the cohomology equation.
We cover $C$ by finitely many open sets $U_i, i\in \cL I$ such that there are open sets $V_i$ in $M$ with $V_i\cap C=U_i$. We also assume that there are biholomorphic mappings $\Phi_i$ from $V_i$ onto the polydisc
$\Del_{n+d}^{r^*}$ of radius $r^*$, where $n$ is the dimension of $C$ and $n+d$ is the dimension of $M$.
 Assume further that  
$
\Phi_i(U_i^{r^*})=\Del_n^{r^*}\times\{0\}$ for $
\var_i\times\{0\}=\Phi_i|_{U_i}$.
 Set $\cL U^r=\{U_i^r\colon i\in\cL I\}$  
 with $U_i^r=
\var_i^{-1}(\Del_n^r)$. We assume that $r^*<1$ and  $\cL U^{r_*
 }$ with 
 $r_*<r^*$,  
 remains a covering of $C$. When $U_I^r:=U^r_{i_0}\cap\cdots U_{i_q}^r$ is non-empty,  it is still Stein~\cite[p.~127]{GR04}.

\begin{defn}\label{nested-cov}
Let $\{U_j^r\}$ be an open covering of $C$ for each $r\in[r_*,r^*]$. We say that the family of coverings $\{U_j^r\}$ is {\it nested}, if each connected component of $U_{k}^{\rho}\cap  U_{j}^{r_*}$ intersects $U_{k}^{r_*}\cap  U_{j}^{r_*}$ when $r_*\leq \rho\leq r^*$.
In particular, $U_{k}^{r_*}\cap  U_{j}^{r_*}$ is non-empty if and only if $U_{k}^{\rho}\cap  U_{j}^{r_*}$ is non-empty.
\end{defn}

Let $N(U_i^{ r^*})$ be the union of all $\ov{U_k^{r^*}}$ that intersect $\ov{U_{i}^{r^*}}$;
as in~\ci{Do71}
we will call the union the {\it star} of $U_i^{r^*}$.
Refining  $\cL U^{r^*}$ if necessary, we may assume that there is a 
 biholomorphism  $
\var_i$ from a neighborhood of the star onto an open set   in $\cc^n$.
If $E',E''$ are holomorphic vector bundles over $C$, we will fix a trivialization of $E'$ over $U_i$ by fixing a holomorphic basis $e'_k=\{e'_{k,1},\dots, e'_{k,m}\}$ in $\ov{U_k^{r^*}}$. We also fix a holomorphic base
$e''_j=\{e''_{j,1},\dots, e''_{j,d}\}$ of $E''$ in $\ov{U_j^{r^*}}$. On $U_{ I}^{r^*}=U_{i_0}^{r^*}\cap\cdots\cap U^{r^*}_{i_q}$, it will be convenient to use the base
\eq{e0q}\nonumber
e_ {i_0\dots i_q}:=e_{i_0}'\otimes e_{i_q}'':=\{e'_{i_0,k}\otimes e''_{i_q,j}\colon 1\leq k\leq m, 1\leq j\leq d\}.
\eeq

Throughout the paper $\|\cdot\|_D$ and $ |\cdot|_D$ denote respectively the $L^2$ and sup norms of a function in $D$, when $D$ is a domain in $\cc^n$. If $f=(f_1,\dots, 
 f_d)$ is a vector of functions, we define 
 the $L^2$ norm, metric, and sup norms as follows:
\al\label {twonorms}\nonumber
 \|f\|_{D}^2&:= \|f\|_{L^2(D)}^2:=\|f_1\|^2_{D}+\cdots+\|f_d\|^2_{D},\\
|f|_{D}^2&:=\sup_{z\in D}|f_1(z)|^2+\cdots+|f_d(z)|^2, \nonumber \\
|f|_{\infty, D} &:=\sup_{z\in D}\max\{|f_1(z)|, \dots, |f_d(z)|\}.\nonumber
\end{align}
For a $d\times d$ matrix $t$ 
 of functions on $D$, denote by $|t|_D, \|t\|_D$, 
 $|t|_{\infty,D}$ respectively the operator norms
defined by
\eq{matrix-norm}\nonumber
|t|_D=\sup_{|f|_D=1}|tf|_D, \quad 
\|t\|_D=\sup_{\|f\|_D=1}\|tf\|_D, \quad 
|t|_{\infty,D}=\sup_{|f|_{\infty,D}=1}|tf|_{\infty,D}.
\eeq
Therefore, $\|t\|_D\leq|t|_D$ as
$\|tf\|_D\leq (\sup_{z\in D }|t(z)|)\|f\|_D=|t|_D\|f\|_D$.

Then we   define the $L^2$ norm for $f\in C ^q(\cL U^{r},\cL O(E'\otimes E''))$  by
\begin{eqnarray}\nonumber
a_ Ie_{ I}&:=&\sum_{\mu=1}^{md}a_ I^{\mu} e_{ I,\mu},\\
\|f\|_{\cL U^{r}}&:=&\max_{I=(i_0,\dots, i_q)\in \cL I^ {q+1}}\left\{\|a_{ I}\circ
\var_{i_q}^{-1}\|_{
\var_{i_q}(U_I)}\colon f_ i=a_{ I} e_{ I}\ \text{in $U_{ I}$}\right\}.
\label{defnorm}\nonumber
\end{eqnarray}
Sometimes we denote   $\|f\|_{\cL U^{r_*}}$ by $\|f\|$ for abbreviation. We define similarly 
 the metric  norm $|f|_{\cL U^{r_*}}$,
or $|f|$,
 and the sup-norm $|f|_{\infty,\cL U^{r_*}}$ or $\sup|f|$. 
  It is obvious that
 \ga\label{trivial-est}\nonumber
 ||f||\leq C
 |f|, \quad \sup|f|\leq\|f\|\leq C\sqrt{\rank(E'\otimes E'')}\sup |f|, \\
 |t|_{\infty}\leq|t|\leq C\rank(E'\otimes E'') 
 |t|_{\infty},\nonumber
  \end{gather}
where $C$ does not depend on $E',E''$.

The first result of this appendix is to find a way to obtain solutions with bounds to \re{deluf} on the original covering, if a  solution with a bound exists on a shrinking covering.
This relies on the nested coverings defined above.
We first study the $L^2$ norms case.
\begin{lemma} \label{SD-p} Let $\cL U^r=\{U_i^r\colon i\in \cL I\}$ with $r_*\leq r\leq r^*$ be  a family of nested finite coverings of $C$.
  Suppose that $f\in C^1(\cL U^{r^*},E'\otimes E'')$ and $f=0$ in $ H^1(\cL U^{r^*},E'\otimes E'')$. Assume that there is a solution $v\in C^{0}(\cL U^{r_*})$  such that
 \eq{first-sol}
 \del v=f, \quad
 \|v\|_{\cL U^{r_*}}\leq K\|f\|_{\cL U^{r^*}}.
 \eeq
  Then there exists a solution $u\in  C^{0}(\cL U^{r^*})$ such that $\del u=f$ on $\cL U^{r^*}$ and
\eq{second-sol}
 \|u\|_{\cL U^{r^*}}\leq C( 
  |\{t'_{kj}\}|_{\cL U^{r^*}}+ K|\{t'_{kj}\}|_{\cL U^{r^*}}|\{t''_{kj}\}|_{\cL U^{r^*}})\|f\|_{\cL U^{r^*}},
\eeq
where $t_{kj}',t_{kj}''$ are the transition  
matrices
  of $E',E''$, respectively, and $C$ depends only on the number  $|\cL I|$
of open sets in $\cL U^{r^*}$ and transition functions of $C$. In particular, $C$ does not depend on $E',E''$.
\end{lemma}
\begin{proof}By assumptions, we have
\ga
f_{jk}=(\del v)_{jk}, \quad U_j^{r_*}\cap U_k^{r_*},\\
\|v\|_{\cL U^{r_*}}\leq K\|f\|_{\cL U^{r^*}}.\label{vKf}
\end{gather}
 Take any  $v^*\in C^0(\cL U^{r^*},E'\otimes E'')$ such that $\del v^* =f$. Then $(\del v^*-\del v)_{jk}=0$ in
 $U^{r_*}_j\cap U^{r_*}_k$, because $(\del v^*)_{jk}=f_{jk}$ on the larger set $U_j^{r^*}\cap U_k^{r^*}$.  Since $\{U_j^{r_*}\}$ is a covering of $C$ then $w:= v_j-v^*_j$ is a global   section of $E'\otimes E''$. This shows that $v_j$, via $v^*_j$, extends to a holomorphic section in $U_j^{r^*}$. In fact, $v_j$ is the restriction of $u_j=v^*_j+w$ defined on $U_j^{r_*}$.

  We now derive the bound for $u_j$. Suppose that $U_j^{r^*}\cap U_k^{r_*}$ is  non-empty. By the assumptions, each component of $U_j^{r^*}\cap U_k^{r_*}$ intersects $U_j^{r_*}\cap U_k^{r_*}$. We have
$
 u_j= u_k+f_{jk}$ on $ U_j^{r_*}\cap U_k^{r_*}
$
  and hence the uniqueness theorem implies that it holds on $U_j^{r^*}\cap U_k^{r_*}$ too.
And on    $U_{j}^{r^*}\cap U_k^{r_*}$, we  have $u_k=v_k$ and $ u_j=v_k-f_{kj}$.   We express the identity in coordinates
$$
u_j=\tilde u_j e_j, \quad v_k=\tilde v_ke_k=\hat v_{kj}e_{j}, \quad
f_{kj}=\tilde f_{kj}e_{kj}=\hat f_{kj}e_{jj}.
$$
Let $t_{kj}', t_{kj}''$ respectively be the transition  
matrices
  of $e_{j}', e_{j}''$ for $E', E''$. Then
$\tilde t_{kj}=t_{kj}'\otimes t_{kj}''$ are the transition  
matrices
  of $e_{kj}$ for $E'\otimes E''$. Then we have
\ga\nonumber
\hat v_{kj}=
t_{jk}'\otimes t_{jk}''\tilde v_k, \quad \hat f_{kj}=t_{jk}'\otimes I_{d}\tilde f_{kj}.
\end{gather}
Thus, $\tilde u_j=\hat v_{kj}-\hat f_{kj}=
t_{jk}'\otimes t_{jk}''\tilde v_k -t_{jk}'\otimes I_{d}\tilde f_{kj}$.
We have
\aln
\|\tilde u_j\|_{L^2(U_j^{r^*}\cap U_k^{r_*})}&=\|\tilde u_j\circ
\var_j^{-1}\|_{L^2(
\var_j(U_j^{r^*}\cap U_k^{r_*}))}\\
&\leq \|(
 t_{jk}'\otimes t_{jk}''\tilde v_k)\circ
\var_j^{-1}\|_{L^2(
\var_j( U_j^{r^*}\cap U_k^{r_*}))}+
\|(t_{jk}'\otimes I_{d}\tilde f_{kj})\circ
\var_j^{-1}\|_{L^2(
\var_j(U_j^{r^*}\cap U_k^{r^*}))}.
\end{align*}
 Here $t_{jk}\circ
\var_j^{-1}=t_{jk}\circ
\var_k^{-1}
\circ\var_{kj}$.
By the properties of operator norm and 
$\|t_{kj}'\otimes t_{kj}''\|_D\leq |t_{kj}'\otimes t_{kj}''|_D\leq |t_{kj}'|_D| t_{kj}''|_D$ for $D=
\var_j( U_j^{r^*}\cap U_k^{r_*})$, we have
\aln
&\|
(t_{jk}'\otimes t_{jk}''\tilde v_k) \circ
\var_j^{-1}\|^2_{D}\leq
C_*
|t'_{jk}|_D^2 \times |t''_{jk}|_D^2\times \|\tilde v_k\|^2_{
\var_k( U_j^{r^*}\cap U_k^{r_*})},
\end{align*}
%
  where the constant $C_*$ comes from the Jacobian of
   $z_k=
\var_{kj}(z_j)$. By \re{vKf}, we have
  $$
 \|\tilde v_k\circ
\var_k^{-1}\|_{L^2}^2\leq K^2\|f\|_{L^2}^2.
    $$
  We also have
   \aln
 \|(
 t_{jk}'\otimes I_{d} \tilde f_{kj})\circ
\var_j^{-1}\|
 _{
\var_j( U_j^{r^*}\cap U_k^{r_*})}\leq
  |t_{jk}'\circ
\var_j^{-1}|_{
\var_j( U_j^{r^*}\cap U_k^{r_*})}\times
 \|f\|_{
\var_j( U_j^{r^*}\cap U_k^{r_*})}.
     \end{align*}
Since $U_j^{r^*}$ is covered by $\{U_{j}^{r^*}\cap U_k^{r_*}\}$, we get the desired bound from
$$\|\tilde u_j\|_{L^2(U_j^{r^*})}\leq\sum_k\|\tilde u_j\|_{L^2(U_j^{r^*}\cap U_k^{r_*})}.\qedhere
$$
\end{proof}

The argument for the norm $|\cdot|$ is  verbatim and we can take the above constant $C_*$ to be one.

\begin{cor}\label{SD-p-sup} With 
 notations and assumptions  in \rla{SD-p}, the solution $u$ also satisfies
\eq{sup-est}\nonumber
 |u|_{\infty,\cL U^{r^*}}\leq C( 
  |\{t'_{kj}\}|_{\cL U^{r^*}}+ K|\{t'_{kj}\}|_{\cL U^{r^*}}|\{t''_{kj}\}|_{\cL U^{r^*}})\sqrt{\rank(E'\otimes E'')}|f|_{\infty,\cL U^{r^*}},
  \eeq
  where $C$ does not depend on $E',E''$.
\end{cor}

The above lemma leads us to the following proposition and definition.
\begin{prop}\label{defSD}
Let $\cL U^r=\{ U_i^r\colon i\in\cL I\}$ with $r_*\leq r\leq r^*$ be a family of nested coverings of a compact complex manifold $C$. Let $E'$ $($resp. $E'')$ be a  holomorphic vector bundle over $C$ with bases $\{e_j'\}$ $($resp. $\{e_j''\})$ and
transition  
matrices
  $t_{kj}'$ $($resp. $\{t_{kj}''\})$.  Suppose that there is a finite number $K$ such that for any    $f\in C^1(\cL U^{r^*},E'\otimes E'')$ with $f=0$ in $ H^1(\cL U^{r^*},E'\otimes E'')$, there is a solution $v\in C^{0}(\cL U^{r_*},E'\otimes E'')$ satisfying \rea{first-sol}. Then there is a possible different solution $v\in C^0(\cL U^{r_*}, E'\otimes E'')$  satisfying \rea{first-sol} in which
 $K$ is replaced by
 \al\label{SD}
K_*(E'\otimes E'')=\sup_{u_1}\inf_{u_0}&\bigl\{\|u_0\|_{\cL U^{r_*}}\colon  \text{$\del u_0=\del u_1$ on $\cL U^{r_*}$}, \\
&\quad \| \del u_1 \|_{\cL U^{r^*}}=1,   u_i\in C^0(\cL U^{r_i},E'\otimes E'')\bigr\}.\nonumber
\end{align}
\end{prop}
\begin{proof} By the assumption, $K_*=K_*(E'\otimes E'')$ is well-defined and $K_*\leq K$.
Fix $u_1\in C^0(\cL U^{r_i},E'\otimes E'')$. Suppose that $\del u_1=f$ and $\|f\|_{\cL U^{r^*}}=1$.   By the definition \re{SD}, there exists $u_0^j$ such that $\del u_0^m=f$ on $\cL U^{r_*}$ and $\|u_0^m\|_{\cL U^{r_*}}\leq K_*+1/m$.
 By the Cauchy formula on polydiscs,  $(u_0^m)_j\circ
\var_j^{-1}$ is locally bounded in $
\var_j(U_j)$ in sup-norm. We may assume that as $m\to\infty$,  $(u_0^m)_j$ converges uniformly to $u_0^\infty$ on each compact subset of $U_j$ for all $j$.    This shows that
$\|(u_0^\infty)_j\circ
\var_j^{-1}\|_{L^2(E)}\leq K_*$ for any compact subset $E$ of $
\var_j(U_j)$. Since $E$  is arbitrary,  we obtain $\|u_0^\infty\|_{U^{r_*}}\leq K_*$. By the uniform convergence, we also have $\del u_0^\infty=f$ on $\cL U^{r_*}$.
\end{proof}
\begin{defn}\label{defK}Let $E', E'', e_j',e_j'', t_{kj}',t_{kj}''$ be as in \rpa{defSD}. Let $t_{kj}''(S^mE'')$ be the transition  
matrices
  of the symmetric power $S^mE''$ induced by $t_{kj}''$.  For  $m=2,3,\dots$, we shall call
\begin{align}\nonumber
 K(E'\otimes S^mE'')&=
|\{t'_{kj}(E')\}|_{\cL U^{r^*}}\\
&\quad + K_*(E'\otimes S^m E'')|\{t'_{kj}(E')\}|_{\cL U^{r^*}}|\{t''_{kj}( S^mE'')\}|_{\cL U^{r^*}}
\nonumber
\end{align}
 the {\it generalized small divisors} of $E'\otimes E''$ with respect to $e_j'', t_{kj}''$.
\end{defn}
\subsection{Donin's smoothing decomposition}
Grauert's smoothing decomposition for 
cochains of analytic sheaves is an important tool. Here we will follow an approach of Donin~\ci{Do71}, by specializing for vector bundles.

We first need to introduce coverings by analytic polydiscs.
  \le{lower-est}Let $C$ be a compact complex  manifold. Let $\{U_i^{r_*}\colon i\in\cL I\}$ be a finite open covering of $C$, and let  $
\var_j$ map $ U_j^r$ biholomorphically  onto $\Del^n_{r}$ for $r_*<r<r^*<1$.
 Assume further that  $
\var_i$ is a biholomorphism defined in a neighborhood of the star $N(U_i^{r^*})$ onto a domain in $\cc^n$. Suppose that $r_*<r_i'<r_i<r^*$, and
 \eq{NE3}\nonumber
 U_I^{r'}:=U_{i_0}^{r_0'}\cap\cdots\cap U_{i_q}^{r_q'}\neq\emptyset.
 \eeq
 Then for  constant $c_n
 \in (0,1)$ depending only on $n$,
\ga
\dist\left(\pd (
\var_{i_q}(U_I^r)),\pd (
\var_{i_q}(U_I^{r'}))\right)\geq c_n\kappa\min_j(r_j-r_j'),\\
 \kappa:= \inf\left\{
 1,\frac{|
\var_{i_q}\circ
\var_{i_\ell}^{-1}(z')-
\var_{i_q}\circ
\var_{i_\ell}^{-1}(z)|}{|z'-z|}\colon z,z'\in\Del_{
 r^*}^n,\forall U^{r^*}_{i_0\dots i_q}\neq\emptyset \right\}.
\label{kap1}
\end{gather}
\ele
\begin{proof}
Note that for sets in $\cc^n$, if $A\subset A'$, $B\subset B'$,  and $A$, $B$ are non-empty,  then
$$\dist( A,B)\geq\dist(A',B').
$$
Recall that $
\var_{i_q}$ is a diffeomorphism from a neighborhood $V$ of
the star $N(U_{i_q})$   onto  a subset $\hat V$ of $\cc^n$. We have
$\pd
\var_{i_q}(U_I^{r})\subset\cup_j\pd
\var_{i_q}(U_{i_j}^{r})$. Thus
\aln
&\dist(\pd
\var_{i_q}(U_I^{r}),
\var_{i_q}(U_I^{r'}))\geq \min_j\dist(\pd\
\var_{i_q}(U_{i_j}^{r}),
\var_{i_q}(U_I^{r'}))\geq
\min_j\dist(\pd
\var_{i_q}(U_{i_j}^{r}),
\var_{i_q}(U_{i_j}^{r'})).
\end{align*}
 We have $\dist(\pd(
\var_{i_q}(U_{i_j}^{r}),
\var_{i_q}(U_{i_j}^{r'}))=\dist(\pd(
\var_{i_q}\circ
\var_{i_j}^{-1}(\Del^{n}_{r})),
\var_{i_q}\circ
\var_{i_j}^{-1}(\Del^{n}_{r'}))$.  Recall that $
\var_{i_q}$ is  defined on $ N(U_{i_q})\supset U^{r_*}_{i_j}$. Then the distance is attained for some $z'\in\pd \Del_{r'}^n$ and $z\in\pd\Del_r^n$.  By the definition of $\kappa$, we get the desired estimate.
\end{proof}

We will recall  the following  smoothing decomposition of Donin~\ci{Do71}. Here we restrict to the case of $H^1$ and
 the holomorphic vector bundle   to indicate the specific bounds in the estimates.
\begin{thm}[Donin~\ci{Do71}]
\label{donin}
Let $C$ be a compact complex manifold and let $\cL U^{r}$ $(r_*<r<r^*<1)$ be a family of open coverings of $C$
 as in \rla{lower-est}.
 Let $E'\otimes E''$ be a holomorphic vector bundle of 
 rank $m$ over $C$ and fix a holomorphic base $e'_j$ $($resp. $e_j'')$ for $E'$ $($resp. $E'')$
over $U_j$. Let $r_*<r''<r'<r<r^*$,  and
\eq{4rs}
r'-r''\leq r^*-r.\nonumber
\eeq
Assume that
 \eq{NE4}
 U_{kj}^{
 r_*} \neq\emptyset,
  \quad \text{whenever $U_{kj}
  ^{r^*}\neq\emptyset$}.
 \eeq
 Let $\{f_{jk}\}\in Z^1(\cL U^{r'},\cL O(E'\otimes E''))$.
Then there exist $g\in Z^1(\cL U^{r},\cL O(E'\otimes E''))$ and $u\in C^0(\cL U^{r''},\cL O(E'\otimes E''))$ such that
\ga
\label{smoothingDec}
f=g+\delta u,\quad 
\text{in $C^1(\cL U^{r''},\cL O(E'\otimes E''))$},\\
 \label{dConstant}
\|u\|_{\cL U^{r''}}+\|g\|_{\cL U^{r}}\leq \f{  C_{n}
 |\{t'_{kj}\}||\{t_{kj}''\}|  
}{(r'-r'')\kappa }\|f\|_{\cL U^{r'}},
\end{gather}
where $\kappa$ is defined  \rea{kap1}. The constant $C_n$ is independent of $E',E''$.
Furthermore, $f\mapsto g=Lf$ and $f\mapsto u=Sf$ are $\cc$-linear.
\end{thm}
\begin{proof}With $f^{r'}_{ij}=f_{ij}$ we are given a cocycle $\{f^{r'}_{ij}\}$ of holomorphic sections of $E'\otimes E''$ over the covering $\cL U^{r'}$. Recall that $r_*<r''<r'<r<r^*$ and $\cL U^{r''}$ is an open covering of $C$.

As in~\ci{Do71}, we will apply $L^2$-theory for $(0,1)$-forms  on a bounded pseudoconvex domain in $\cc^n$. In our case the domain is actually a polydisc.
Fix a holomorphic base $e_{k}'=(e'_{k,1},\dots,e'_{k,m})$ for the vector bundle $E'$ in $U_k^{r^*}$ with transition functions $t_{kj}^{\prime }(z_j)$. Analogously, let $t_{kj}^{\prime\prime }(z_j)$ be the transition  
matrices
  for basis $e_k''$ of $E''$ for $\cL U^{r^*}$.
For brevity, we write $t_{kj}$ for $t_{kj}(z_j)$.

 We can write
\eq{frpij}
f^{r'}_{ij}=  \tilde f^{r'}_{ij}e_{ij}
=t_{ki}^{\prime }\otimes t_{kj}^{\prime\prime }\tilde f_{ij}^{r'}e_{kk}:=
\hat f_{ ij;k
}^{r';r^*}  e_{
 kk}, \quad \text{on $U_i^{r'}\cap U_j^{r'}\cap U_k^{r^*}$}.
\eeq
The $U_k^{r^*}$ is covered by $\cL U_k^{r';r^*}:=\{U^{r'}_{i}\cap U^{r^*}_k\}_{ i}$, while $\{\hat f_{ij;k}^{r';r^*}\}\in Z^1(\cL U_k^{r';r^*},
\cL O^{md})$.
Now $\{\hat f_{ij;k}^{r';r^*}\circ
\var_k^{-1}\}\in Z^1(
\var_k(\cL U_k^{r';r^*}),\cL O^{md})$, where $
\var_k(\cL U_k^{r';r^*})$
is a covering of the polydisc $\Del^n_{r^*}$.
 By \rl{lower-est}, we have 
\eq{Ldist}
c_{i;k}:=\dist(\pd (
\var_k(U_{i}^{r'}\cap U_{k}^{r^*})),  
\var_k(U_{i}^{r''}\cap U_{k}^{r}))\geq c_n \kappa (r'-r'').
\eeq
 Let $d_{ i;k}(z)$ be the distance to $
 \var_k(U_{i}^{r''}\cap U_{k}^{r})$ from $z\in\cc^n$. Let $\chi$ be a non-negative smooth function in $\rr$ so that $\chi(t)=1$ for $t<3/4$ and $\chi(t)=0$ for $t>7/8$. By smoothing
 the Lipschitz function  $\chi(\f{1}{c_{i;k}}d_{i;k}(z))$, we obtain a non-negative smooth function $z\to
 \tilde\phi_{i;k}^{r'';r'}(z)$ that equals $1$ when $d_{i;k}(z)\leq \f{1}{2}c_{i;k}$ and
 by \re{Ldist} it  has compact support in $
\var_k(U_{i}^{r'}\cap U_{k}^{r^*})$. Note that we can achieve 
\eq{ntvr}
|\nabla \tilde\phi_{i;k}^{r'';r'}|< C_nc_{i;k}^{-1}\leq c_n C_n
\kappa^{-1} /{(r'-r'')}.
\eeq
 Then $\tilde\phi_{i;k}^{r'';r'}\circ
\var_k$ is a non negative function with compact support in $U_i^{r'}\cap U_k^{r^*}$ such that for $\tilde\phi_k^{r'';r'}:=\sum\tilde\phi_{i;k}^{r'';r'}$, we have  $\tilde\phi_k^{r'';r'}\circ
\var_k>1/2$ in $U_k^{r}=\bigcup_i(U_i^{r''}\cap U_k^r)$
 since  $\chi(\f{1}{c_{i;k}}d_{i;k})=1$ on $
\var_k(U_{i}^{r''}\cap U_{k}^{r})$. Then
  by the mean-value theorem and the first inequality of \re{ntvr}, we get
 \eq{quater+}
 \tilde\phi_k^{r'';r'}(
\var_k(x))>1/4, \quad \text{if $\dist(
\var_k(x),
\var_k(U_k^r))< \min_{i}c_{i,k}/{C_*}$},
 \eeq
 for some suitable $C_*$.
 Recall that $c_n\leq1$ and $\kappa_n\leq1$. Since $\dist(
\var_k(U_k^{r}),
\var_k(\pd U_k^{r^*}))
 =r^*-r'\geq c_n\kappa(r'-r'')$,  there is a   smooth function $\hat\phi_k^{r;r^*}\colon 
\var_k(U_k^{r^*})\to[0,1]$ with compact support such that $\hat\phi_k^{r;r^*}=1$ in $
\var_k(U_k^{r})$, and
 \eq{half-}
 \hat\phi_k^{r;r^*}(x)<3/4, \quad \text{if $\dist(
\var_k(x),
\var_k(U_k^{r}))> \min_{i}c_{i,k}/{C_*}$}.
 \eeq
Note that the latter can be achieved with
$$
|
\nabla \hat\phi_k^{r;r^*}|<\tilde C_1/{\min_{i}c_{i,k}}\leq C_2
\kappa^{-1}/{(r'-r'')}.
$$
In $U_k^{r^*}$, define a non-negative smooth function
$$
 \phi_{i;k}^{r'';r'} =\left\{\f{\tilde\phi_{i;k}^{r'';r'}}
 {1-\hat\phi_k^{r;r^*}+\tilde\phi_{k}^{r'';r'}}\right\}\circ
\var_k,
$$
where the smoothness follows from the denominator being bigger than $1/4$ by \re{quater+} and \re{half-}. Thus, $ \phi_{i;k}^{r'';r'}$ has compact support in $U_i^{r'}\cap U_k^{r^*}$  and $\sum_i \phi_{i;k}^{r'';r'}=1$ in $U_k^{r}=\bigcup_i(U_i^{r''}\cap U_k^r)$, as $\hat\phi_k^{r;r^*}=1$ on $U_k^r$.  We can verify that 
  \eq{deriv-partition}
|\nabla(\phi_{i;k}^{r'';r'}\circ
\var_k^{-1})|< C'\kappa^{-1} /{(r'-r'')}.
\eeq
Consider the expression
\eq{hfLjk}
w_{j;k}=
\sum_\ell\phi_{\ell;k}^{r'';r'}\hat f_{ \ell j;k}^{r';r^*}.
\eeq
Recall that $\phi_{\ell;k}^{r'';r'}$ has compact support in $U_\ell^{r'}\cap U_k^{r^*}$.
 Thus it is smooth on $\om:=U_j^{r'}\cap U_k^{r^*}\cap U_{\ell}^{r'}$ and vanishes on an open set $D$ containing $ U_{j}^{r'}\cap U_k^{r^*}\setminus \om$. On the other hand,
 $
\hat f^{r';r^*}_{\ell j;k}$ is holomorphic in $\om$. Hence the product $\phi_{\ell;k}^{r'';r'}\hat f_{ \ell j;k}^{r';r^*}$ is
smooth   in $
U_j^{r'}\cap U_k^{
r^*}$.  Then $v_{j;k}=\ov\partial w_{j;k}$ is a smooth $(0,1)$ form in $
U_j^{r'}\cap U_k^{
r^*}$.

Let $\cL A$ denote the sheaf of smooth functions on $C$.  We now pull back the forms from the polydisc $\Del^n$ via $
\var_k$.
For each fixed $k$, we have $\{w_{j;k}\}_{j}\in C^0(\cL U_k^{r';r^*},\cL A^m)$.
 Let us  denote
$t_{kj}^{\prime}\otimes I$ by $t_{kj}'$.  By  $f_{ij}=f_{ik}-f_{jk}$ and \re{frpij}, we have
\eq{}\nonumber
t_{ki}^{\prime }\otimes t_{kj}^{\prime\prime}\tilde f_{ij}^{r'}=t_{ki}^{\prime} \tilde f_{ik}^{r'}-
t_{kj}^{\prime} \tilde f_{jk}^{r'}.
\eeq
Since $\sum_i\phi_{i;k}^{r'';r'}=1=\hat\phi_k^{r;r^*} 
\circ
\var_k$ on $U_k^r$, then by $\del f=0$ and
 \re{frpij}, we get   on $ U_i^{r'}\cap U_k^{r}\cap U_j^{r'}$
\aln
w_{i;k}-w_{j;k}
&=
\sum_\ell\phi_{\ell;k}^{r'';r'}(\hat f_{ \ell i;k}^{r';r^*}-\hat f_{ \ell j;k}^{r';r^*})=
\sum_\ell\phi_{\ell;k}^{r'';r'}(t_{k\ell}^{\prime }\otimes t_{ki}^{\prime\prime }\tilde f_{\ell i}^{r'}
-t_{k\ell}^{\prime }\otimes t_{kj}^{\prime\prime }\tilde f_{\ell j}^{r'})
\\
&=\sum_\ell\phi_{\ell;k}^{r'';r'}(  t_{kj}^{\prime }\tilde f_{jk}^{r'}
-  t_{ki}^{\prime }\tilde f_{ik}^{r'})
=t_{kj}^{\prime }\tilde f_{jk}^{r'}
-  t_{ki}^{\prime }\tilde f_{ik}^{r'}.
\end{align*}
 The latter is holomorphic.  Thus $(\del v)_{ij;k}=\ov\partial(\del w)_{ij;k} 
  =0$ on $U_i^{r'}\cap U_k^{
   r^*}\cap U_j^{r'}$. This shows that
 \eq{vkcpt}\nonumber
 v_k:=v_{j;k}\eeq
  is actually a $\ov\pd$-closed $(0,1)$ form in $U_{k}^{ r^*}$. 
  Thus $(
\var_k^{-1})^*v_k$ is a $\ov\pd$-closed $(0,1)$-form on the polydisk $\Del_{ r^*}^n$.
  By the $L^2$ theory~\cite[Thm.~4.4.3]{Ho90} applied to each component of $v_k=\sum_{\ell=1}^m
  \tilde v_{k}^{\ell}e_{kk,\ell}$, 
   we have a bounded linear operator $S\colon v_k\to u_k$ such that
  $\ov\pd (
   (
\var_k^{-1})^* u_k)=(
\var_k^{-1})^*v_k$.
  Returning to the complex manifold via $
\var_k$, we have
\aln
\|u_k\|_{U_k^{ r^*}}&=
\|u_k\circ
\var_k^{-1}\|_{L^2(\Del^n_{r^*})}\leq C\| v_k\circ
\var_k^{-1}\|_{L^2(\Del^n_{r^*})}\\
&\leq
\f{\tilde C\kappa^{-1}
|\{t'_{kj}\}||\{t_{kj}''\}|}{r'-r''}\|f\|_{L^2(\cL U^{r^*})}.
\end{align*}
Here we have used \re{hfLjk}, estimate \re{deriv-partition} and the definition of norm \re{defnorm}.
Note that the $\tilde C$ is independent of the rank since we applied the $L^2$ componentwise.
Set 
$\hat g_{j;k}^{r';r}=w_{j;k}-u_{k}$ on $U_j^{r'}\cap U_k^r$.
We obtain
\ga\label{htik}
\hat g_{i;k}^{r',r} -\hat g_{j;k}^{r',r} =\hat f_{ij;k}^{r';r^*}, \quad  U_{i}^{r'}\cap U_k^{r}\cap  U_{j}^{r'},\\
\max_j\|\hat g_{j;k}^{r';r}\|_{U_j^{r'}\cap U_k^r}\leq \f{C\kappa^{-1}
|\{t'_{kj}\}||\{t_{kj}''\}|}{r'-r''}\|f\|_{\cL U^{r'}}.
\end{gather}
 We have obtained \re{dConstant}.

 To verify \re{smoothingDec},
  we will use the same base $e_k$ and
take the
 product of \re{htik} with $e_k$ in order to obtain
on $U_i^{r''}\cap U_j^{r''}\cap U_k^{r}\cap U_\ell^{r}$
$$
 g_{i;k}^{r';r}-g_{j;k}^{r';r}=\hat f_{i
 j;k}^{r';r^*}e_k=f_{ij}^{r'}=\hat f_{ij;\ell}^{r';r^*}e_{\ell}= g_{i;\ell}^{r';r} -  g_{j;\ell}^{r';r}
$$
and thus
\eq{urpp}
 g_{j;\ell}^{r';r} - g_{j;k}^{r';r} =  g_{i;\ell}^{r';r} - g_{i;k}^{r';r}, \quad
  \text{on $U_i^{r''}\cap U_j^{r''}\cap U_k^{r}\cap U_\ell^{r}$}.
\eeq
Then we have a (well-defined)  holomorphic section
$$
g_{k\ell}^{r}:=g_{i;\ell}^{r';r}-g_{i;k}^{r';r}, \quad   U_k^{r}\cap U_\ell^{r}.
$$
We verify that $\{g_{k\ell}^{r}\}\in Z^1(\cL U^{r},\cL O^m)$. Set $u_i^{r''}:=g_{i;i}^{r'';r}$. Since $r'\leq r$ we actually have $\{u_i^{r''}\}\in C^0(\cL U^{r'},E'\otimes E'')$. 
 However, only on $U_{i}^{r''}\cap U_j^{r''}$, we can verify
via \re{urpp} that
$$
g_{ij}^{r}-f_{ij}^{r'}=(g_{i;j}^{r'';r}-g_{j;j}^{r'';r})-(g_{i; j}^{r'';r}-g_{i;i}^{r'';r})=u_i^{r''}-u_j^{r''}.\qedhere
$$
\end{proof}

The above result is a type of Grauert's smoothing decomposition, which can also be obtained by open mapping theorem. See for instance \cite[p.~200]{GR04}. 
However, this yields an unknown bound in the estimates. 

%

\subsection{Finiteness theorem with bounds}
The above smoothing decomposition does not provide a solution to the cohomology equations,
i.e. if $f=0$ in $H^1(\cL U^{r'},\cL O(E'\otimes E''))$, then there exists $u\in C^0(\cL U^{r''},\cL O(E'\otimes E''))$ such that $\delta u=f$ on $\cL U^{r''}$, for some $r''\leq r'$.
We will follow \cite{GR04} to derive   the finiteness theorem  with explicit bounds. In particular, this provides solutions of first cohomology equations with bounds on shrinking domains.

We first recall the resolution atlases from \cite[p.~194]{GR04}, specializing them for the vector bundles.
Assume that we have 
 coordinate charts
$$
\var_k\colon U_k^{r^*}\to P_k:=
\var_k(U_k^{r^*})=\Del^{r^*}_n.
$$
Define $U_I^{r^*}=U_{i_0}^{r^*}\cap\dots\cap U^{r^*}_{ i_q}$ for $I\in \cL I^{q+1}$.
Then  $
\var_I=(
\var_{i_0},\dots,
\var_{i_q})$ is defined on $U_I^{r^*}$ with range $\hat U_I^{r^*}$.
Unless otherwise stated, we omit the superscript $r^*$ in $U_I^{r^*}$. We can define a {\it proper}   embedding
$$
\var_{I}\colon U_{I}\to\hat U_I\hookrightarrow P_{I}:=\Del^{r^*}_{ 
n_q}, \quad n_q=n(q+1).
$$
Then the push-forward of the  vector bundle $E'\otimes E''|_{U_{
I}}$ defines a coherent analytic  sheaf $(
\var_I)_*(E'\otimes E'')$ over
 $P_I$ by trivial zero extension; see \cite[p.~5, p.~195]{GR04} and \cite[p. 239]{GR84}.
  A section $f\in\Gaa(U_I, E'\otimes E'')$ yields a section $\hat f_I$ of $(
\var_I)_*(  E'\otimes E'' )$
over $P_I$ by
$$
\hat f_I\circ
\var_I(x)=(f_I
(x),\dots, f_I(x)),
   \quad \hat f_I|_{P_I\setminus\hat U_I}=0.
$$
 Note that $\ov{U^{
r^*}}$ has a Stein neighborhood.
Then following notation in \cite[p.~196]{GR04} we have an epimorphism by Cartan's Theorem~A:
\eq{EsL-}
\epsilon_I\colon\cL O^{\ell}|_{\Del^{
r^*}_{n_q}}\to (
\var_I)_*(E'\otimes E'' )|_{U_I}, \quad \ell\geq\rank (E'\otimes E'') ,\nonumber
\eeq
where $
\e_I$ is defined  by finitely many  global sections defined in a neighborhood of $\ov {P_I}$.
When $E'\otimes E'' $ is a vector bundle, we take $\ell$ to be the minimal value,  the rank of $E'\otimes E'' $, and
specify the above $\e_I$ by taking
\eq{EsL}
\e_I\colon g_I\to \tilde g_I:=(
\var_I)_*\{g_I\circ
\var_I
e_{I}\}.\nonumber
\eeq
 Here we want to obtain a more general description without restricting to a vector bundle.
Define
\eq{CqL}\nonumber
  C^q(\cL U):=\prod_{I\in \cL I^{q+1}}\cL O^\ell(P_{I}).
\eeq
(Set $\cL O^\ell(P_{I})=0$ when $U_I^{
r^*}$ is empty.)
We recall that  $P_I=\Del^{
r^*}_{n_q}$ is independent of the order of multi-indices.
 Thus
\eq{cqk}\nonumber
C^q(\cL U)\cong (\cL O(\Del^{
r^*}_{n_q}))^{L}:=\cL O^{L}(\Del^{
r^*}_{n_q}).
\eeq
 Here $L\leq |\cL I^{q+1}|\ell$.   Let $\cL O_h(\Del^{r}_{n_q})$ be the space of holomorphic functions on $\Del^{r}_{n_q}$ with  finite $L^2$ norm on $\Del^{r}_{n_q}$. Set $P_I^r=\Del^{r}_{n_q}$ for $I\in\cL I^{q+1}$.
We  define a Hilbert space
\eq{CqhL}\nonumber
C^q_h(\cL U^{r}):=\prod_{I\in \cL I^{q+1}}\cL O_h^\ell(P^{r}_{I}):=\cL O_h^{
L}(\Del^{r}_{n_q}),
\eeq
which is a subspace of $C^q(\cL U^{r})$.

 Using the collection $\e=\{\e_I\colon I\in \cL I^{q+1}\}$,
 we define
$$
\cL C_h ^q(\cL U^{r},E'\otimes E'' ):=\e(C_h ^q(\cL U^{r}))\cong C_h ^q(\cL U^{r})/(\ker\e\cap C_h ^q(\cL U^{r})),
$$
which is the vector space of $q$-cochains, equipped with the standard coboundary operator $\del$.
\begin{rem}
Our cochains are not necessary alternating. As in \cite[p.~35]{GR04},
we let $\cL C_a ^q(\cL U,E'\otimes E'' )$ denotes alternating cochains. For the isomorphism of the two kinds of Ce\v{c}h cohomology groups; see \cite[p.~35]{GR04} and Serre~\ci{Se66}. Since we are interested in the cohomological solutions with bounds, we fix our nation without requiring that the cochains  be alternating.
\end{rem}

Let $\|\cdot\|_{\Del_{n_q}^r}$  be the Hilbert space norm on $\cL C_h ^q(\cL U^{r})$ and set
$$
\|\zeta\|_{\cL U^{r}}^\bullet=\inf\{\|v\|_{\Del_{n_q}^{r}}\colon v\in  C_h ^q(\cL U^{r}),\e(v)=\zeta\}, \quad \zeta\in C_h ^q(\cL U^{r},E'\otimes E'' ).
$$
The inclusion $\cL C_h ^q(\cL U^{r},E'\otimes E'' )\hookrightarrow \cL C ^q(\cL U^{r},E'\otimes E'' )$ is continuous and compact (\cite[Thm.~3, p.~197]{GR04}).
We also define
\gan
 Z_h ^q(\cL U^{r}):=\e^{-1}(Z_h ^q(\cL U^{r}, E'\otimes E'' )),\\
\|\zeta\|_{\cL U^{r}}:=\inf\{\|v\|_{\Del^{n_q}_{r}}\colon v\in  Z_h ^q(\cL U^{r}),\e(v)=\zeta\}, \quad \forall\zeta\in Z_h ^q(\cL U^{r},E'\otimes E'' ),\\
 \ov v:=\e(v).
\end{gather*}
Then $Z_h ^q(\cL U^{r},E'\otimes E'' )$ is an isometric subspace of $\cL C_h ^q(\cL U^{r},E'\otimes E'' )$ via inclusion.
Let $
\{g_0,g_1,\dots\}$ be a monotone orthogonal base of $Z_h^1(\cL U^r)$ (\cite[p.~141, p.~201]{GR04}). An important feature of the monotone base is that the vanishing orders of $g_j$ at the origin satisfy
$$
\ord_0 g_0\leq\ord_0g_1\leq \cdots, \quad \lim_{i\to\infty}\ord_0g_i=\infty.
$$
By \cite[Thm.~1, p.~192 and p.~201]{GR04}, for a given $\nu$ there is an $\mu$ such that
 \ga\label{mu-nu}
 g_i(Z)=O(|Z|^\nu),\quad i>\mu, \quad Z\in\Del_r^{n_q}.
 \end{gather}
In fact, let the index set be  $\cL I=\{1,\dots, L\}$.
Set $\om((f_1,\dots, f_L))=\min\{(\all, Q)\colon f_{\all,Q}\neq0\}$ by using order $<$ on $\cL I\times \nn^m$ defined by $(\all, P)<(\beta,Q)$ if $| P|<|Q|$,  or if $| P|=|Q|$ and there is an $\ell$ such that $p_\ell<q_\ell$  and $p_{\ell'}=q_{\ell'}$ for all $\ell'>\ell$, or if   $P=Q$ and $\all<\beta$. Then the basis $\{g_j\}$ satisfies
\eq{omj}
\om(g_j)<\om(g_{j+1}).\nonumber
\eeq

We now return to the case $q=1$ with $n_q=2n$.
 In the sequel, $\{|t_{kj}'|\}=\{|t_{kj}'|\}_{\cL U^{r^*}}$ and $\{|t_{kj}''\}|=\{|t_{kj}''|\}_{\cL U^{r^*}}$.
 \begin{thm}[Donin-Grauert-Remmert]\label{DGR}
 Let $C$ be a compact complex manifold and let $\cL U^{r}$ $(r_*<r<r^*<1)$ be a family of open coverings of $C$
 as in \rla{lower-est}
 such that \rea{NE4} holds for all $k,j$.
 Let  $E=E'\otimes E''$ be a holomorphic vector bundle of positive rank $m$ over $C$ and fix a holomorphic base $e_j'$ $($resp. $e_j'')$ for $E'$ $($resp. $E'')$
over $U_j^{
r^*}$.  Suppose that $r_*<r''<r'<r<
r^*$
and $r'-r''\leq r^*-r$. Let $\theta
=r'/r$.  Let $\{g_0,g_1,\dots\}$ be a monotone orthogonal base of $Z_h^1(\cL U^r)$ as above.  Assume that
     $\mu,\nu$  satisfy \rea{mu-nu} and
 \ga\label{anut}
t:=\f{
C_{n}\kappa^{-1} }{(r'-r'')(r-r')^{2n}}\theta^{\nu}<1/2.
\end{gather}
There exist $\ov{g_{m_0}},\dots, \ov{g_{m_{\mu^*}}}$ such that their equivalence classes in $H^1(\cL U^r,E)$
 form a $\cc$-linear basis of subspace spanned by $\ov{g_{0}},\cdots, \ov{g_{\mu}}$
in $H^1(\cL U^r,E)$.  For any
$f\in Z_h^1(\cL U^{r'},E)$ there exists
$v\in C_h^0(\cL U^{r''},E )$ satisfying $f=\delta v+\sum_0^{
\mu^*} c_i \ov{g_{m_i}}$ with
\ga \label{dConstant+++}
|c_i| \leq \f{
C_{n}\kappa^{-1}A_{r}(E)}{r-r'}\|f\|_{\cL U^{r'}},\\
 \|v\|_{\cL U^{r''}}\leq \f{C_{n}\kappa^{-1} B_{
 r_-}(E)}{r-r'}\|f\|_{\cL U^{r'}}, \quad
 \forall r_-  
 \in[r',r),\label{dConstant+}
\end{gather}
\ga 
\label{gjcj}
g_j=\sum_{i=0}^{\mu^*} c_{ji}\ov{g_{m_i}}+\del 
\eta_j^*,
\quad\eta_j^*\in
C^0(\cL U^r,E),\\
A_r(E)=
|\{t_{kj}'\}||\{t_{kj}''\}|
\max_{0\leq i\leq
\mu_*}\sum_{j=0}^{
^{\mu}} |c_{ji}|, \quad
B_{
r_-}(E)=
|\{t_{kj}'\}||\{t_{kj}''\}|
\sum_{j=0}^{
\mu}
\|\{ \eta_j^*\}\|_{\cL U^{
r_-}}.
\label{A1B}
\end{gather}
 Furthermore, all $c_j=0$ when $f=0$ in $H^1(C, E)$.
\end{thm}
\begin{rem} The solution operator $f\to v$ may not be linear.
See a proof by Donin~\ci{Do71} to get a linear solution operator for which the constant $C_*$ results from a lemma of Schwartz.
\end{rem}
\begin{rem}
The previous theorem gives a solution $v$, defined on a smaller domain, to the equation $f=\delta v$ (i.e cohomological equations) whenever $f$ is $0$ in the first cohomology group. It also provides a bound of the solution in terms of the data. We emphasize that this bound depends on the bundle $E'\otimes E''$. In the applications we have in mind, we will have to consider a sequence of bundles $\{
S^m E''\}_m$,  and 
we will need to control the growth of these bounds as $m$ goes to infinite, similarly to the {\it small divisors} appearing in local dynamical systems.
\end{rem}
\begin{proof} Recall that $q=1$ and $n_1=2n$.
We may assume that $\|g_j\|_{\Del_{2n}^r}
=1$.
By the definition of $\mu,\nu$ and the monotone basis,   we have for any $v\in Z^1_h(\cL U^r)$,
 \ga\label{estim-mbasis}
 \|v-\sum_{j=0}^{\mu}(v,g_j)g_j\|_{\Del_{2n}^{r'}}\leq \f{C_n}{(r-r')^{2n}}(r'/r)^{\nu}\|v\|_{\Del_{2n}^r}
 \end{gather}
 where $C_n(r-r')^{-2n}$ is the constant $M$  in~\cite[Thm.~6, p.~191]{GR04}.

Replacing the smoothing lemma in~\cite[p.~200]{GR04}   by \rt{donin}, we derive some estimates following the proof of   the finiteness lemma in \cite[p.~201]{GR04}.  By assumption, we have
$$
t=\f{ C_{n}\kappa^{-1} }{
(r'-r'')(r-r')^{2n}}\theta^{\nu}<1/2,\quad \theta
=\f{r'}{r}<1.
$$

 Let $\zeta_0
 :=f\in Z_h^1(\cL U^{r'},E'\otimes E'' )$. By \rt{donin}, we have for some $\xi_0\in Z_h^1(\cL U^{r},E'\otimes E'' )$
 \gan
 \zeta_0=\xi_0+\del\eta_0,\\
 \|\xi_0\|_{\cL U^{r}}\leq
 t'\|\zeta_0\|_{\cL U^{r'}}, \quad \|\eta_0\|_{\cL U^{r''}}\leq t'\|\zeta_0\|_{\cL U^{r'}},
 \end{gather*}
 with $t':=\f{ C_{n}
|\{t_{kj}'\}||\{t_{kj}''\}|
}{\kappa (r'-r'')}$.
Let $\ov v$ denote $\e(v)$. Then $\xi_0=\ov v_0$ for some $v_0$ satisfying   $\|v_0\|_{\Del_{2n}^r}=\|\xi_0\|_{\cL U^r}$; see \cite[p.~198]{GR04}.
 Consider
 $$
 w_1=v_0-\sum_{j=0}^{\mu}(v_0,g_j)_{\Del^r_{2n}}g_j, \quad \zeta_1=\ov w_1.
 $$
According to \re{estim-mbasis}, we have
 $$
 \|\zeta_1\|_{\cL U^{r'}}\leq\|w_1\|_{\cL U^{r'}}\leq \f{C_n}{(r-r')^{2n}}(r'/r)^{\nu}\|v_0\|_{\Del^r_{2n}}\leq t\|\zeta_0\|_{\cL U^{r'}}.
 $$
 Therefore
 \gan
 \zeta_0=\sum_{j=0}^{\mu}(v_0,g_j)_{\Del^r_{2n}}\ov g_j+\del\eta_0+\zeta_1.
 \end{gather*}
 In general, we have
 \gan
 \zeta_\ell=\sum_{j=0}^{\mu}(v_\ell,g_j)_{\Del^r_{2n}}\ov g_j+\del\eta_\ell+\zeta_{\ell+1},\\
 \|v_\ell\|_{\Del^r_{2n}}=\|\xi_\ell\|_{\cL U^r}\leq
 t't^\ell\|\zeta_0\|_{\cL U^{r'}},\\
  \|\zeta_{\ell+1}\|_{\cL U^{r'}}\leq t\|\zeta_\ell\|_{\cL U^{r'}}\leq t^{\ell+1}\|\zeta_0\|_{\cL U^{r'}},\\
\|\eta_\ell\|_{\cL U^{r''}}\leq
 t'
t^\ell\|\zeta_0\|_{\cL U^{r'}}.
 \end{gather*}
  Then we have
 \gan 
 f=\zeta_0=\sum_{j=0}^{\mu}\sum_{\ell=0}^\infty(v_\ell,g_j)_{\Del^r_{2n}}\ov g_j+\del\sum_{\ell=0}^\infty\eta_\ell,\\
\sum_{\ell=0}^\infty|(v_\ell,g_j)|\leq \sum_{\ell=0}^\infty\|v_\ell\|_{\Del^r_{2n}}\leq
\f{t'}{1-t}\|\zeta_0\|_{\cL U^{r'}},\\
\sum_{\ell=0}^\infty\|\eta_\ell\|_{\cL U^{r''}}\leq \f{t'}{1-t}
\|\zeta_0\|_{\cL U^{r'}}.
 \end{gather*}

 So far we have followed the proof of the finiteness
 lemma in~\cite[p.~201]{GR04}. We now finish the proof of the theorem.   Let us first find the linearly independent elements $\ov{g_{i_0}},\dots, \ov{g_{i_{\mu_*}}}$.
 Assume first that all $\ov{g_{i}}=0$ in $H^1:=H^1(\cL U^r,E'\otimes E'')$. Then  $\del \eta_j=\ov{g_j}$ with $\eta_j\in C^0(\cL U^r,E)$.  Assume now that $\ov{g_{m_0}}\neq 0$ in $H^1$ for some $m_0$.
 Then we have two cases again: either
$\ov{g_i}=c_{i0}\ov{g_{m_0}}+\del  \eta_{i}$ on $\cL U^r$
for all $i\in\{0,\dots, \mu\}\setminus m_0$, or it fails for some $m_1$.  We repeat this to exhaust all elements so that
 \eq{A.63}
 \ov{g_{j}}=\del\eta_j^*+\sum_{i=0}^{\mu_*}c_{ji}\ov{g_{m_i}}, \quad \eta_{j}^*\in\cL C^0(\cL U^r,E),\quad
 0\leq j\leq \mu
 \eeq
 while
 $\ov{g_{m_0}},\dots,\ov{g_{m_{\mu_*}}}$ are linearly independent in $H^1$.
 (Note that the above expression means the trivial identity $\ov{g_j}=\ov{g_j}$ when $j$ is not in $\{m_0,\dots, m_{\mu^*}\}$.)
 We have obtained \re{gjcj} with the decomposition
 \gan
 f=\sum_{j=0}^{\mu^*}c_j\ov{g_{m_j}}+\del v,\\
 c_j=\sum_{\ell=0}^\infty(v_\ell,g_j)_{\Del^r_{2n}}+\sum_{i=0}^\mu c_{ij}\sum_{\ell=0}^\infty(v_\ell,g_i)_{\Del^r_{2n}},\\
 v=\sum_{i=0}^\mu\sum_{\ell=0}^\infty(v_\ell,g_i)_{\Del^r_{2n}}\eta^*_i +\sum_{\ell=0}^\infty\eta_\ell.
 \end{gather*}

The solution $\eta_j^*$ in \re{A.63} can be bounded in $U^{r_-}$ for any $r_-<r$. Of course we need to estimate $\eta_j^*$ on $\cL U^{r'}$. Thus, $r_-\geq r'$.
We have
\gan
\sum_{j=0}^\mu\sum_{\ell=0}^\infty  |(v_\ell,g_j)_{\Del^r_{2n}}c_{ji}|\leq \f{t'}{1-t}\sum_{j=0}^\mu |c_{ji}|
\|\zeta_0\|_{\cL U^{r'}},\\
\left\|\left\{\sum_{\ell=0}^\infty\eta_\ell+\sum_{\ell=0}^\infty\sum_{j=1}^{\mu}
 (v_\ell,g_j)_{\Del^r_{2n}}\eta^*_j\right\}\right\|_{\cL U^{r_-}}\leq \f{t'}{1-t}
 \left\{1+\sum_{j=0}^\mu \|\eta_j^*\|_{\cL U^{r_-}} \right\}
\|\zeta_0\|_{\cL U^{r'}}.
\end{gather*}
Set $A_r(E)=
|\{t_{kj}'\}||\{t_{kj}''\}|
\max_{i=0}^{\mu^*}\sum_{j=0}^\mu |c_{ji}|$ and $B_{r_-}(E)=
|\{t_{kj}'\}||\{t_{kj}''\}|
(1+\sum_{j=0}^\mu \|\eta_j^*\|_{\cL U^{r_-}})$.
We have obtained the required estimates.

Finally,
 let us assume that $f=0$ in $H^1(C,E)$ in order to show that all $c_j=0$ and thus $f=\del v$.
Since each $U^{r''}$ is Stein, we also have $f=0$ in $H^1(\cL U^r,E)$. Thus $f=\del\tilde v$ with $\tilde v\in \cL C^0(U^{r''},E)$.
We get $\del(\tilde v-v)=\sum_{j=0}^{\mu^*} c_j\ov{g_{m_j}}$. By the linear independence, we conclude that $c_j=0$.
We are done.\end{proof}

%

\begin{thm}\label{donin-sol-cohom}\label{estim-cohom-summary}Let $C$ be a compact complex manifold and let ${\cL U}^{r}$ $(
r_*\leq r\leq r^*<1)$ be
nested coverings of $C$  as in \rpa{regcover}. Let $\mu,\nu, r,r',r'',
 r_*,r^*$ be given in \rta{DGR}, which satisfy \rea{anut}.
    Let $f\in Z^1(\cL U^{r'}, E'\otimes E'')$. Suppose that $
 f=0$ in $H^1(C,E'\otimes E'')$. Then there exists a solution $\{u_j\}\in C^0(\cL U^{r'}, E'\otimes E'')$ such that $\del u=f$ and
\ga
\|u\|_{\cL U^{r'}}\leq K(E'\otimes E'') \|f\|_{\cL U^{r'}},\label{nonshrinking}\\
K(E'\otimes E''):=C(|\{t'_{kj}\}|_{\cL U^{r'}}+K_*(E'\otimes E'')
|\{t'_{kj}\}
|_{\cL U^{r'}}|\{t''_{kj}\}|_{\cL U^{r'}}),\\
\intertext{
where  $K_*(E'\otimes E'')$, defined by \rea{SD}, satisfies}
K_*(E'\otimes E'')
\leq\f{C_nB_{
r_-}(E'\otimes E'')}{(r-r')\kappa},
\label{uUr}
\end{gather}
where $\kappa$ and $B_{r_-}$ are defined by \rea{kap1} and \rea{A1B}.
The same conclusion holds if both sides are in   sup norms $|\cdot|_{\cL U^{r'}}$,
when $(r-r')\kappa$ is replaced by $((r-r')\kappa)^n$.
\end{thm}
\begin{rem}The main conclusion is that \re{nonshrinking} holds without shrinking the covering $\{U_i^{r'}\}$ on which $f$ is defined. The solution operator $f\mapsto u$ may  not  be linear.  
The small divisor conditions are carried by $B_{  
r_-}$ which is determined by \re{anut} and \re{A1B}, while the bounds in \rt{donin} as smoothing lemma does not involve small divisors.
\end{rem}
\begin{proof} By the Leray theorem, we know that $[f]=0$ in $H^1(\cL U^{r'},E)$.
 By \rt{DGR}, we have a solution $u\in C^0(\cL U^{r''},E)$ so that
\gan
f_{jk}=(\del u)_{jk}, \quad U_j^{r''}\cap U_k^{r''},\\
\|u\|_{\cL U^{r''}}\leq K\|f\|_{\cL U^{r'}}.
\end{gather*}
Then the conclusion follows from \rl{SD-p}.
%

When the super norm is used, we first obtain a solution $u=\{u_k\}$ for $\cL U
^{r^*}$ for $r^*=(r''+r')/2$, while \re{uUr} takes the form
$$
\|u\|_{\cL U^{r^*}}\leq K\|f\|_{\cL U^{r'}}\leq 
( \sqrt\pi r')^nK|f|_{\cL U^{r'}}.
$$
By 
$\dist(
\var_k(U_k^{r''}),\pd
\var_k(U_k^{r^*}))= r^*-r''$ and power series expansion,  we have
$|u|_{\cL U^{r''}}\leq (\sqrt\pi(r^*-r''))^{-n}\|u\|_{\cL U^{r^*}}.$
Then the conclusion follows from \rl{SD-p} again.
 \qedhere\end{proof}

\subsection{Existence of nested coverings}\label{nested}
In this subsection, our main goal is to
  construct nested coverings by using transversality theorems
 and analytic polyhedrons.
  We recall that $C_n$ is a $n$-dimensional compact complex manifold. We shall omit to mention its dimension in what follows.




We first deal with the transversality for a piecewise smooth  boundary of an analytic polyhedron and we then define the general position 
property of several analytic polyhedrons.
\begin{defn}\label{gen-p}
\bpp 
\item Let $M_j$ be a $C^1$ real hypersurface defined by $r_j=0$, where $r_j$ is a $C^1$ function in  an open set $\om_j$ of a complex manifold $C$ and $dr_j\neq0$ on $M_j$. We say that
$M_1,\dots, M_N$ are in the {\it general} position, if $dr_{i_0}\wedge\cdots\wedge dr_{i_q}\neq0$  at each point of $M_{i_0}\cap\cdots\cap M_{i_q}$ for any   $1\leq i_0<\dots<i_q\leq N$.
\item Let $\om$ be a   
{\it proper}
 open set of a complex manifold $C$ and let $f\in \cL O^N(\om)$.
 We say that
\eq{f1k1}
Q:=Q_N(f,\om):=\{z\in\om,|\;|f(z)|:=\max\{ |f_1(z)|,\dots,| f_N(z)|\}< 1\},
\eeq
 is an analytic $N$-polyhedron in $\om$ if $Q$  is non-empty and relatively compact in $\om$,  
 and $Q$ does not contain any compact connected component. We say that $Q$ is {\it generic}, if
\eq{dfi1}
(d|f_{i_1}|\wedge \cdots\wedge d|f_{i_\ell}|)(x)\neq0 \quad \forall x\in \{|f_{i_1}|=\cdots=|f_{i_\ell}|=1\}\cap \pd  Q
\eeq
for all $i_1<\dots<i_\ell$ and $1\leq\ell\leq N$.
\epp 
\end{defn}
We will apply transversality theorems. This requires us to  use open
submanifolds in $\cc^n$ which may not be
closed in $\cc^n$.
  Since $Q_N=Q_N(f,\om)$ does not contain compact connected component,   the closure of each connected component of $Q_N$ must intersect some $Q_N^i:=\{|f_i|=1\}\cap\om$.
We will call $Q_N^i$ a {\it face} of $Q_N$. Removing  each $Q_N^i$ from $\om$ if it does not intersect $\ov Q_N$, we get a new $\om$ such that $\ov Q_N$ intersects each $Q_N^i$. Applying the same procedure to $Q_N^{i_1\dots i_k}:=Q_N^{i_1}\cap\cdots\cap Q_N^{i_k}$, we may assume that the non-empty intersection of any number of $Q_N^1,\dots Q_N^N$ intersects $\ov Q_N$.
 By \re{dfi1}, the closed set $\ov Q_N$ does not intersect the closed subset of $\om$ defined by
$$
(d|f_{i_1}|\wedge \cdots\wedge d|f_{i_\ell}|)(x)=0 \quad |f_{i_1}|(x)=\cdots=|f_{i_\ell}|(x)=1.
$$
Removing the above sets from $\om$, we find a neighborhood $\om^*$ of $\ov Q_N$ such that if $Q_N^{i_1\dots i_k}$ with $i_1<i_2<\cdots<i_k$ intersects $\om^*$, then it intersects $\ov Q_N$ and it is  a codimension $k$ smooth submanifold in $\om^*$.
 For brevity we will call $\om^*$ a
{\it neat} neighborhood of $Q$. We will take $\om=\om^*$ without specifying $\om^*$.
 \begin{defn} Let $\om_i$ be open sets in $C$.
For $i=0,\dots, p$,  assume that   $\phi_i\in \cL O^{N_i}(\om_i)$ and
 $Q_{N_i}(\phi_i,\om_i)$ is an  analytic polyhedron in $\om_i$.
We say that they are in the {\it general position}, if all faces $Q_{N_i}^j$ for $1\leq j\leq N_i$ and $0\leq i\leq p$ are in general position.
More precisely, $\om^*_{N_i}\cap Q_{N_i}^j$ are in the general position, where each $\om^*_i$ is a neat neighborhood of $\ov{Q_{N_i}}$.
\end{defn}

Let us describe some elementary properties of generic analytic polyhedrons.
 If $Q_N(f,\om)$ is defined in $\om$ by \re{f1k1},  we denote for $\rho=(\rho_1,\dots,\rho_N)$
$$
Q_N^\rho(f,\om):=\{z\in\om\colon|f_j(z)|<\rho_j, j=1,\dots,N\}.
$$

 \le{genericP}Let $Q_{N_i}=Q_{N_i}(\phi_i,\om_i)$ be generic polyhedrons in $C$ for $0\leq i\leq p$. Suppose that
 $Q_{N_0},\cdots, Q_{N_p}$ are in the general position. Then
 $$
Q_{N_0+\cdots+N_p}((\phi_0,\dots,\phi_p),\om_0\cap\cdots\cap\om_{p})= Q_{N_0}\cap\cdots\cap Q_{N_p},
 $$
 if non-empty,
 is a generic $N_0+\cdots+N_p$ analytic polyhedron in $\om_{0\cdots p}:=\om_{0}\cap\cdots\cap \om_{p}$.
 \ele
\begin{proof} Let $N=N_0+\cdots+N_p$.  It is clear that
$Q:=Q_{N_0}\cap\cdots\cap Q_{N_p}=Q_{N}((\phi_0,\dots, \phi_{p}),
\om_{i_0\cdots i_p})$. Since $\ov Q\subset\cap\ov{Q_{N_i}}$, then $\ov Q$ is compact in $\om_{0\cdots p}$. Write
$(\phi_0,\dots, \phi_p)=(\psi_1,\cdots,
\psi_N)$.
Suppose that $x\in\pd Q$. Since $\ov Q$ is compact in $\om$, then there exist
$\mu_1<\cdots <\mu_m$ with $m\geq1$ such that
$|\psi_{\mu_i}(x)|=1$ and $|\psi_{j}(x)|<1$ for $j\neq \mu 
_\ell$. By the assumption of the general position,
we see that the faces of $Q$ are in the general position.
\end{proof}

 Let $X,Y$ be smooth real manifolds without boundary and $W$ a smooth submanifold of $Y$.
Following~\ci[p.~50]{GG73}, we say that a smooth mapping $h\colon X\to Y$ is {\it transversal} to $W$ at $x\in X$, denoted by
 $h\ptt W$ at $x$, if
either $h(x)\not\in W$ or
$$
T_{h(x)}W+dh(T_xX)=T_{h(x)}Y.
$$
Denote $h\ptt W$ on $A$ if $h\ptt W$ at  each $x\in A\subset X$.
When $h$ is the inclusion, we denote $h\ptt W$ on $A$ by    $X\ptt W$ on $A$.
 Finally,
extending Definition~\ref{gen-p} $(a)$,  we say that smooth
real submanifolds $W_0,\dots, W_k$ in $Y$ are in
{\it the general position}
 if for any $0\leq i_1<\cdots<i_m\leq k$ we have
\eq{wedgeEll}
\bigwedge_{\ell=1}^{k}\bigwedge_{j=1}^{d_{i_\ell}} dr_{i_\ell,j}  (y)\neq0,
\quad \forall y\in W_{i_1}\cap\cdots\cap W_{i_m},
\eeq
where $W_i\subset \om_i$ is defined by $r_{i,1}=\cdots=r_{i,d_i}=0$ with $dr_{
i,1}\wedge\cdots\wedge dr_{i,d_i}\neq0$
at each point of $W_i$. Thus $d_i$ is the codimension of $W_i$ in $\om_i$.
 It is clear that \re{wedgeEll} holds if and only if
\eq{edgTr}
W_{i_j}\ptt(W_{i_1}\cap\cdots\cap W_{i_{j-1}}) \ \text{at $y$},\quad \forall y\in W_{i_1}\cap\cdots\cap W_{i_k}, \ 0<j\leq m.
\eeq

 For an analytic
$N$-polyhedron $Q_N$
in $\om$ with faces
$Q_N^1,\dots, Q_N^N$,
 we call $Q_N^{i_1\cdots i_k}=Q_N^{i_1}\cap \cdots \cap Q_N^{i_k}$ with $i_1<\cdots <i_k$
 and
 $k\geq 1$
  an
{\it edge} of $Q$.  When $Q_N$ is generic,  a nonempty edge $Q_N^{i_1\cdots i_k}$     is a codimension $k$ submanifold in $\om$.
Let $\{Q_N^1\cdots, N_N^{N'}\}$ be the set of all edges, with the first $N$ edges being the faces.
\pr{gen-tra}Let $Q_{N_i}=Q_{N_i}(\phi_i,\om_i)$ be generic polyhedrons in $C$ for $0\leq i\leq p$
with $\om_i$ being a neat neighborhood of $\ov Q_{N_i}$.
Then $Q_{N_0},\dots, Q_{N_p}$
are in the general position if and only if any   $0\leq i_1<\dots<i_k\leq p$ and $1\leq j_\ell\leq N'_{i_\ell}$, the edges $Q^{j_1}_{N_{i_1}}, \cdots, Q^{j_k}_{N_{i_k}}$   are in the general position. Equivalently, each edge  $Q_{N_\ell}^s$ intersects transversally with each edge of the intersection of any number of $Q_{N_0},\dots, Q_{N_{\ell-1}}$, for $\ell=1,\dots, p$. \epr
\begin{proof}
Since each edge of a polyhedron is the intersection of its faces, it is clear that if $Q_{N_0},\dots, Q_{N_p}$ are in the general position, then the edges $Q^{j_1}_{N_{i_1}}, \cdots, Q^{j_k}_{N_{i_k}}$   are in the general position for $0\leq i_1<\cdots<i_k\leq p$.

 Conversely,   let $\phi_i=(\phi_{i,1},\dots, \phi_{i,N_i})$ and let $\psi_1,\dots, \psi_m$ be a subset of $\phi_{0,1},\dots,$ $\phi_{0, N_0}, \dots,$
 $\phi_{p,1},\dots$, $ \phi_{p,N_p}$. We emphasize that
 we do not assume that  the latter are distinct functions, although $\phi_{i,1}, \dots, \phi_{i, N_i}$ are distinct by the general position property of the faces of $Q_{N_i}$.
Suppose that $\psi_\ell$ is in $\{\phi_{i_\ell,1},\dots, \phi_{i_\ell,N_{i_\ell}}\}$.  We need to show that
\eq{wedges}
d|\psi_1|\wedge\cdots\wedge d|\psi_m|(x)\neq0
\eeq
if
for all $\ell$,  $|\psi_\ell|(x)=1$ and $x\in\ov Q_{N_{i_\ell}}$. Without loss of generality, we may assume that $i_1\leq i_2\leq \cdots\leq i_m$.
Thus
$$
(\psi_1,\dots, \psi_m)=(\tilde\psi_{\all_0},\dots, \tilde\psi_{\all_\ell}), \quad \all_1<\all_2<\cdots<\all_\ell
$$
with $\tilde\psi_{\all_\beta}$ being a non-empty subset of components of $\phi_{\all_\beta}$. Without loss of generality, we may assume that $\tilde\psi_{\all_\beta}=(\phi_{\all_\beta,1},\dots,\phi_{\all_\beta,\gaa_\beta})$ with $\gaa_\beta>0$.
Thus $|\phi_{\all_\beta,1}|=\cdots=|\phi_{\all_\beta,\gaa_\beta}|=1$ define an edge $W_{\all_\beta}$ of $Q_{\all_\beta}$.
 Then \re{wedges} is equivalent to
\eq{}\label{wedges+}\nonumber
\left(\bigwedge_{\del=1}^{\gaa_\ell}d|\phi_{\all_\ell, \del}|\right)\wedge\left(\bigwedge_{\ell'=1}^{\ell-1}\bigwedge_{\del=1}^{\gaa_{\ell'}}d|\phi_{\all_{\ell'},\del}|\right)(x)\neq0.
\eeq
The equivalence of \re{wedgeEll} and \re{edgTr} implies that \re{wedges} follows from the assumption that $W_{\all_\ell}\ptt(W_{\all_1}\cap\cdots\cap W_{\all_{\ell-1}})$,
for $\all_1<\all_2<\cdots<\all_\ell$.
\end{proof}
 \begin{lemma}[Golubitsky-Guillemin~{\cite[p.~53]{GG73}}]
  \label{gglemma}
  Let $X, B,$ and $Y $ be smooth manifolds  with $W$ a submanifold
of $Y$. Let  $\psi\colon B\to  C^\infty(X, Y)$ be a mapping $($not necessarily continuous$)$ and
define $\Psi\colon X \times B \to Y$ by $\Psi(x, b) = \psi(b)(x)$. Assume that $\Psi$ is smooth and that
$\Psi\ptt W$. Then the set $\{b\in B\mid \psi(b) \ptt W\}$ is dense in $B$. \ele

 \pr{regcover}
 Let $C$ be a compact complex manifold of dimension $n$. Let $\{U_i\colon
 i=1,\dots, m\}$ be a finite open covering of $C$.
 Assume that $\var_j$   is a biholomorphism
 from a neighborhood $\om_j$ of  the star $N(U_j)$ of $U_j$ onto $\hat \om_j\subset\cc^n$
 such that $U_j=\var_j^{-1}(\Del_n)=Q_n(\var_j,\om_j)$. There exists $\del>0$ satisfying the following:
\bpp
\item For each $j$, there
are a relatively compact open set
 $\tilde \om_j$ $($resp. $\tilde U_j)$ in $\om_j$ $($resp. $\tilde\om_j)$ and a dense open set $A_j$ of $\Del^\del_n$ such that
if $c_j\in A_j$, then   $\tilde\var_j:=\var_j-c_j$
is a biholomorphic mapping from
$\tilde U_j$ onto  $\Del_n$, and $\tilde U_1:=Q_n(\tilde\var_1,\tilde \om_1),\dots,$ $ \tilde U_m:=Q_n(\tilde\var_m,\tilde \om_m)$ are generic $n$-polyhedrons in the general position, where $\{\tilde U_1,\ldots \tilde U_m\}$ remains an open covering of $C$ and $\tilde \om_j$ is a neighborhood of $N(\tilde U_j)$.
In particular each $\tilde\var_j$, a translation of $\var_j$,  is injective on $\tilde \om_j$.
\item 
There is $0<r_*<1$ such that if $r_*\leq \rho_i\leq1$, then
$\tilde U_{i_0}^{\rho_0},\dots, \tilde U_{i_q}^{\rho_q}$ are generic $n$-polyhedrons in the general position, where $\tilde U_i^\rho:=\tilde\var_i^{-1}(\Del^\rho_n)$.
\epp
\epr
\begin{proof}  $(a)$
We will apply the transversality theorem for real submanifolds in $\cc^n$. Therefore, we will use
old coordinate charts $\var_j$ to map edges of polyhedrons $Q_j(\var_j,\om_j)$ into $\cc^n$.
Set $c_1=0,\tilde\var_1=\var_1,\tilde U_1=U_1$. Let $\hat W_1,\dots, \hat W_{L_0}$ be all edges of $\Del_n$.
  Let
$\tilde U_1^1,\dots, \tilde U_1^{N'}$ be
all edges of $\tilde U_1$. Set $\widetilde W^\ell_1=\var_2(
\om_2\cap \tilde U_{1}^{\ell})$.
Define
$$
\Psi\colon \cc^n\times \Del^\del_n\to Y:=\cc^n
$$
with $\Psi(x,b)=x+b$ and $\psi^b(x)=\Psi(x,b)$. Let $\psi^b|_{\widehat W_{\ell'}}$ be the restriction of $\psi^b$ to $\hat W_{\ell'}$.
    Applying  \rl{gglemma},
mainly the density assertion in the lemma,
finitely many times in which $W=\widetilde W^\ell_1$, we can find    $b_2\in\Del_n^\del$ such that
$$
\psi^{b_2}|_{\widehat W_{\ell'}} \ptt \widetilde W^{\ell}_1 \quad
\text{on $\var_2(\ov{\tilde U_1}\cap\ov{\om_2'})$}, \quad \forall \ell,\ell'
$$
where $\om_2'$ is a  relatively compact open subset of $\om_2$ which is independent of $\del$,   and $\ov{U_2}\subset\om_2'$.
We also remark that \re{gglemma} can be applied for finitely many times since $\var_2(\ov{\tilde U_1}\cap\ov{\om_2'})$ is compact.
Since $\ov{\tilde U_1}\cap\ov{U_2}$ is compact, then
\eq{twoEdg} \psi^{c_2}|_{\widehat W_{\ell'}} \ptt \widetilde W^\ell_1
\quad  %
\text{on $\var_2(\ov{\tilde U_1}\cap\ov{\om_2'})$}, \quad \forall \ell,\ell'
\eeq
when $|c_2-b_2|$ is sufficiently small.
Applying $\var_2^{-1}$ to \re{twoEdg} yields
\eq{twoEdg+}\var_2^{-1}( \psi^{c_2}|_{\widehat W_{\ell'}}) \ptt   (\om_2\cap\tilde U_1^\ell )\quad
\text{on $\ov{\tilde U_1}\cap\ov{\om_2'}$}, \quad \forall \ell,\ell'.
\eeq
 With $c_2$ being determined, set
 $$
 \tilde\var_2^{-1}=\var_2^{-1}(\id+c_2).
 $$
 Thus $\tilde\var_2=\var_2-c_2$. When $\del$ and $|c_2-b_2|$ are sufficiently small, we have $\ov{\tilde U_2}=\tilde\var_2^{-1}(\ov{\Del_n})\subset\om_2'$. Therefore, \re{twoEdg+} implies that every edge of $\tilde U_2$ intersects each edge of $\tilde U_1$.
 We have determined $\tilde U_{2}=\tilde\var_2^{-1}(\Del_n)$.

 We have verified $(a)$ when $m=2$. Let us assume that it also holds for $m\geq j$. By \rl{genericP}, each edge of
  a non-empty intersection of any number of
  $\tilde U_1,\dots, \tilde U_j$ is a smooth submanifold.
We remark the above
 transversality argument mainly uses the fact that $\var_2$ is a biholomorphism, while
each edge of $\tilde U_1$ is a smooth submanifold.

   To repeat
  the above argument for $m=2$ in details,
   we  list all edges of  all possible intersections of $\tilde U_1,\dots, \tilde U_{j}$
as $W'_1,\dots, W'_L$ so that each $W_j$ is an edge of some analytic polyhedron $U_j'$, where $U_j'$ is the intersection of some of $\tilde U_1,\dots, \tilde  U_{j'}$ which are in general position by the induction hypothesis
as mentioned above.  Therefore, by 
\rl{genericP},   each   $U_\ell'$ is generic.
Now we are in the situation of $m=2$ by considering the sets of two analytic polyhedrons $\{ U_\ell', U_{j+1}\}$ one by one for $\ell=1,\dots, j'$.
Here
 $U_{j+1}=\var_{j+1}^{-1}(\Del_n)$ with $\var_{j+1}$ being biholomorphic in
  a neighborhood of $N(U_{j+1})$.
Therefore, we can find $\tilde\var_{j+1}=\var_{j+1}-c_{j+1}$ such that each edge of $\tilde U_{j+1}$ intersects each $W_\ell'$ transversally on $\ov{\tilde U_{j+1}}\cap\ov{U_\ell'}$.

 The above argument shows the existence of $c_1,\dots, c_N$ in $\Del^\del_n$ when $\del$ is sufficiently small. The openness property on  $A_j$ is clear, since by shrinking $\tilde \om_j$ slightly the general position and generic properties are preserved under small perturbation of $c_j$. Then density of $A_j$ when $\del$ is sufficiently small can also be achieved; indeed when $c_j$ is sufficiently small, we may shrink $\om_j$ slightly and apply the above argument by replacing $\var_j-c_j$ with $\var_j$. Finally, $\{\tilde U_1,\dots, \tilde U_N\}$ still covers $C$ when $\del$ is sufficiently small.
We have verified $(a)$.

The assertion $(b)$ follows from $(a)$ and   
\rp{gen-tra}. Indeed, we first note that when $r_*$ is less than $1$, but it is sufficiently close to $1$, the 
$\pd Q^\rho(\tilde\var_j)$ is in a given neighborhood of $\pd Q(
\tilde \var_j,\tilde\om_j)$,
as $Q^\rho(\tilde\var_j,\tilde\om_j)$ does not have any  compact connected component.
By the relative compactness of
$Q_n(\tilde\var_i,\tilde \om_i)$, the condition \re{dfi1} with $f_j$ being replaced by $f_j/{\rho_j}$ and   
the general position condition remain true when $\rho_j$ are in $[r_*,1]$ when $r_*<1$ is sufficiently close to $1$.
 The proof is complete.
\end{proof}
 The following is a basic property of a generic analytic polyhedron.
\pr{component} Let $C$ be a compact complex manifold of dimension $n$.
Let $Q_N(f,\om)$ be a generic analytic $N$-polyhedron $C$ defined by \rea{f1k1}
and 
\rea{dfi1}.
 There exists $r_*\in(0,1)$ satisfying the following.
\bpp
\item
If $\rho=(\rho_1,\dots,\rho_N)$ and $\rho'=(\rho_1',\dots,\rho_N')$ satisfy
 $r_*\leq\rho_i'\leq\rho_i\leq1$, every connected component of $Q_N^{\rho}(f,\om)$ intersects
$Q_N^{\rho'}(f,\om)$
and  the latter is non-empty.
\item 
 There are finitely many
open sets $\om_j''$ in $C$ and smooth diffeomorphisms $\phi_j$ sending $\om_j''$ onto $\hat\om_j''$ in $\rr^{2n}$
such that $\{\om_j''\}$ covers $\pd Q_N(f,\om)$, and for any $p_0,p_1\in\phi_j(\om_j''\cap Q_N^{\rho}(f,\om))$ there is a smooth curve $\gaa$
in $\phi_j(\om_j''\cap Q_N^{\rho}(f,\om))$ connecting $p_0$ and $p_1$ with length $|\gaa|\leq C|p_1-p_0|$, where $C$ depends only on $\phi_j$ and $\om_j''$. \epp
\epr
\begin{proof}
(a)  Set $Q=Q_N(f,\om)$ and $Q^\rho=Q_N^\rho(f,\om)$.
For each $x\in\pd Q$, we find $\mu_1<\dots<\mu_m$ with
 $m\leq N$
such that
\ga\label{bdry-polyh}
|f_{\mu_i}(x)|=1, \quad i\leq m; \quad |f_j(x)|<1, \quad j\neq\mu_1,\dots, \mu_m.
\end{gather}
Note that $\{\mu_{1},\dots,\mu_{m}\}$ is uniquely determined by $x$.
By
the transversality condition \re{dfi1}, we have $m\leq
2n$. Choose an open set $\om'$ such that $x\in\om'\subset\om$ and
$$|f_i(z)|<1, \quad \forall z\in\ov{\om'}, \ i\neq\mu_{1},\dots,\mu_{m}.
$$
In particular, we have
$$
Q\cap\om'=\{z\in\om'\colon|f_{\mu_i}(z)|<1, \quad i=1,\dots, m\}.
$$
By \re{dfi1}, we  can  take $(
|f_{\mu_1}|,\dots, |f_{\mu_m}|)$ to be the first $m$ components of a
smooth diffeomorphism
 $\var\colon\om'\to \hat\om$, shrinking $\om'$ if necessary. Taking a smaller open subset $\om''$ of $\om'$ with $x\in\om''$, we may assume that
$$
t\zeta\in\hat\om, \quad \forall \zeta\in\hat\om'':=\var(\om''), \quad 1-\del\leq t\leq1,
$$
for some  $\del\in(0,1]$.

Since $\pd Q$ is compact,
there exists $\{x_j,\om_j'',  
 \om_j'\colon j=1,\dots, k\}$ satisfying the following:
\bpp
\item The $k$ is finite.
For each $j$,  we have that 
$x_j\in\om''_j\subset
\om_j'\subset\om$, $x_j\in\pd Q$, and
$\om_j'$ is an open subset of $\om$.  For each $j$, we have $m_j$ and $\mu_{j,1}<\ldots<\mu_{j,m_j}$, 
which are the numbers associated to $x_j$,  so that \re{bdry-polyh} holds for $x=x_j$.
$\{\om_1'',\dots\om_k''\}$ is an open covering of $\pd Q$.
\item  $|f_{\mu_{j,\ell}}(x_j)|=1$ for $\ell=1,\dots,m_{j}$ and
\ga\nonumber
\label{tdel}
M_j:=\sup_{z\in\ov{\om_j'}}\{|f_i(z)|\colon i\neq \mu_{j,1}, \dots, \mu_{j,m_j}\}<1,\\
\label{tdel-}\nonumber
 \om_j'\cap Q=\{z\in\om_j'\colon
|f_{\mu_{j,\ell}}(z)|<1,\ell=1,\dots, m_j\}.
\end{gather}
Here we set $M_j=0$ if $m_j=N$. 
\item The $(
|f_{\mu_{j,1}}|,\dots,| f_{\mu_{j,m_j}}|)$  are the first $m_j$ components of a
smooth diffeomorphism $\phi_{j}$   from $ \om_{j}$ onto a subset $\hat\om_{j}$ of $\cc^n$.  There exists $\del^*>0$ such that $\hat\om_{j}'':=\phi_j(\om_j'')$ satisfies
\eq{memj}
\{t\zeta\colon \zeta\in\hat\om_{j}''\}\subset\hat\om_{j},\quad
\forall j,\forall  t\in[1-\del^*,1].
\eeq
\epp
Indeed, let $\phi_j(x_j)=(1,\dots, 1,\tilde x_j)$ with $\tilde x_j\in\rr^{2n-m_j}$. We can take
\eq{exactHwj}
\hat\om_j''=(1-\del^*,1+\del^*)^{m_j}\times B^{\del''}_{2n-m_j}(\tilde x_j)
\eeq
where $B^{\del''}_{2n-m_j}(\tilde x_j)$ is the ball in $\rr^{2n-m_j}$ centered at $\tilde x_j$ with a sufficiently small  radius $\del''$. Note that
\eq{exactHwj+}
\phi_j(Q^\rho\cap\om_j'')=(1-\del^*,\rho_1)\times\cdots\times(1-\del^*,\rho_{m_j})\times B^{\del''}_{2n-m_j}(\tilde x_j).
\eeq
Define
$$M^*=\sup\{|f(z)|\colon z\in Q\setminus\cup_{j=1}^k\om_j''\}.
$$
 Then $M^*<1$. By the maximum principle, we have $|f|\leq M^*$ on $Q\setminus  \cup_{j=1}^k\om_j''$.
 Fix 
 $r_*$ so that
\ga\nonumber
1>r_*>\max\{1-\del^*,  M^*,
M_1,\dots, M_k\}.
\end{gather}

Suppose that $r_*\leq\rho_i'\leq\rho_i\leq 1$ for $i=1,\dots, N$.
Let $\Om$ be a connected component of $Q^\rho_N$.
Since $\Om$ does not have a compact connected component,
there exists $z^*\in \pd\Om$ satisfying $|f_{i}(z^*)|=\rho_{i}$ for some $i$.
Since $\rho_i > M^*$, then $z^*\in\om_j''$ for some $j$. Let us assume that $z^*\in\om_1''$, and $
(\mu_{1,1},\dots,\mu_{1,m_1})=(1,\dots, {m_1})$.
Thus $\phi_1=(|f_1|,\dots, |f_{m_1}|,\tilde f_{m_1+1},\dots,\tilde f
_{2n})$. We now replace $z^*$ by some $z_*\in\Om\cap\om_1''$.  We consider a path  defined by
$$
t\to\gaa(t):=\phi_{1}^{-1}(t\phi_{1}(z_*)), \quad 1-\del^*\leq t\leq 1.
$$
Note that by \re{memj}, $\gaa$ is well defined and is contained in $\om_1$. We now have
\ga\label{fLg}
|f_\ell(\gaa(t))|=
t|f_\ell(z_*)|\leq t\rho_\ell, \quad \ell\leq m_1.
\end{gather}
Since $\gaa(t)\in\om_1$, we also have
\ga\label{fLg+}
|f_\ell(\gaa(t))|\leq M_1< r_*, \quad \ell>m_1.
\end{gather}
 This shows that $\gaa(t)\in Q^{\rho}_N$. Since $\Om$ is a connected component of $Q^\rho_N$ and $\gaa(1)=z_*\in\Om$, we must have $\gaa(t)\in\Om$.  By the definition of $M_j$,
at $t=1-\del^*$ we have $t\rho_\ell\leq 1-\del^*<\rho_\ell'$.  Combining with \re{fLg}-\re{fLg+}, we get
$\gaa(1-\del^*)\in Q_N^{\rho'}$.

$(b)$
Since
$p_0,p_1$ are in the same
$\hat\om_j''$, the assertion also follows from the above construction of 
$\hat\om_j''$ via \re{exactHwj}-\re{exactHwj+} and the convexity of $\hat\om_j''$.
   \end{proof}

In summary,   by  \rp{regcover} we cover $C$ by generic analytic $n$-polyhedrons $U_i=\var_i^{-1}(\Del_n)$ ($i=1,\dots, m$), which are   in the general position. By \rl{genericP}, each $U_i\cap U_j$, if non-empty, is a generic analytic polyhedron.  Applying
\rp{component} $(a)$  to all non-empty $U_i\cap U_j$, we know that $\{U_i^r=\var_i^{-1}(\Del_n^r)\colon i=1,\dots, m\}$ for $r_*\leq r\leq 1$ is
a family of nested coverings.  Therefore, we can apply \rt{DGR} and \rt{donin-sol-cohom}.

\newcommand{\doi}[1]{\href{http://dx.doi.org/#1}{#1}}
\newcommand{\arxiv}[1]{\href{https://arxiv.org/pdf/#1}{arXiv:#1}}

  \def\MR#1{\relax\ifhmode\unskip\spacefactor3000 \space\fi%
  \href{http://www.ams.org/mathscinet-getitem?mr=#1}{MR#1}}

\bibliographystyle{abbrv}

\begin{bibdiv}
\begin{biblist}

\end{biblist}
\end{bibdiv}


\begin{bibdiv}
\begin{biblist}

\bib{AG05}{article}{
      author={Ahern, P.},
      author={Gong, X.},
       title={A complete classification for pairs of real analytic curves in
  the complex plane with tangential intersection},
        date={2005},
        ISSN={1079-2724},
     journal={J. Dyn. Control Syst.},
      volume={11},
      number={1},
       pages={1\ndash 71},
  url={https://doi-org.ezproxy.library.wisc.edu/10.1007/s10883-005-0001-7},
      review={\MR{2122466}},
}

\bib{Ar76}{article}{
      author={Arnol'd, V.~I.},
       title={Bifurcations of invariant manifolds of differential equations,
  and normal forms of neighborhoods of elliptic curves},
        date={1976},
        ISSN={0374-1990},
     journal={Funkcional. Anal. i Prilo\v{z}en.},
      volume={10},
      number={4},
       pages={1\ndash 12},
      review={\MR{0431285}},
}

\bib{Ar88}{book}{
      author={Arnol'd, V.~I.},
       title={Geometrical methods in the theory of ordinary differential
  equations},
     edition={Second},
      series={Grundlehren der Mathematischen Wissenschaften [Fundamental
  Principles of Mathematical Sciences]},
   publisher={Springer-Verlag, New York},
        date={1988},
      volume={250},
        ISBN={0-387-96649-8},
  url={https://doi-org.ezproxy.library.wisc.edu/10.1007/978-1-4612-1037-5},
        note={Translated from the Russian by Joseph Sz\"{u}cs [J\'{o}zsef M.
  Sz\H{u}cs]},
      review={\MR{947141}},
}

\bib{Br71}{article}{
      author={Brjuno, A.~D.},
       title={Analytic form of differential equations. {I}, {II}},
        date={1971},
        ISSN={0134-8663},
     journal={Trudy Moskov. Mat. Ob\v{s}\v{c}.},
      volume={25},
       pages={119\ndash 262; ibid. 26 (1972), 199\ndash 239},
      review={\MR{0377192}},
}

\bib{CMS03}{article}{
      author={Camacho, C.},
      author={Movasati, H.},
      author={Sad, P.},
       title={Fibered neighborhoods of curves in surfaces},
        date={2003},
        ISSN={1050-6926},
     journal={J. Geom. Anal.},
      volume={13},
      number={1},
       pages={55\ndash 66},
         url={https://doi-org.ezproxy.library.wisc.edu/10.1007/BF02930996},
      review={\MR{1967036}},
}

\bib{CLPT18}{article}{
      author={Claudon, B.},
      author={Loray, F.},
      author={Pereira, J.~V.},
      author={Touzet, F.},
       title={Compact leaves of codimension one holomorphic foliations on
  projective manifolds},
        date={2018},
     journal={Ann. Sci. \'{E}cole Norm. Sup. (4)},
      volume={51},
       pages={1457\ndash 1506},
}

\bib{CG81}{incollection}{
      author={Commichau, M.},
      author={Grauert, H.},
       title={Das formale {P}rinzip f\"{u}r kompakte komplexe
  {U}ntermannigfaltigkeiten mit {$1$}-positivem {N}ormalenb\"{u}ndel},
        date={1981},
      series={Ann. of Math. Stud.},
      volume={100},
   publisher={Princeton Univ. Press, Princeton, N.J.},
       pages={101\ndash 126},
      review={\MR{627752}},
}

\bib{Do71}{article}{
      author={Donin, I.~F.},
       title={Cohomology with estimates for coherent analytic sheaves over
  complex spaces},
        date={1971},
     journal={Mat. Sb. (N.S.)},
      volume={86(128)},
       pages={339\ndash 366},
      review={\MR{0299829}},
}

\bib{Fi18}{article}{
      author={Fischer, E.},
       title={\"{U}ber die {D}ifferentiationsprozesse der {A}lgebra},
        date={1918},
        ISSN={0075-4102},
     journal={J. Reine Angew. Math.},
      volume={148},
       pages={1\ndash 78},
  url={https://doi-org.ezproxy.library.wisc.edu/10.1515/crll.1918.148.1},
      review={\MR{1580952}},
}

\bib{GG73}{book}{
      author={Golubitsky, M.},
      author={Guillemin, V.},
       title={Stable mappings and their singularities},
   publisher={Springer-Verlag, New York-Heidelberg},
        date={1973},
        note={Graduate Texts in Mathematics, Vol. 14},
      review={\MR{0341518}},
}

\bib{GS16}{article}{
      author={Gong, X.},
      author={Stolovitch, L.},
       title={Real submanifolds of maximum complex tangent space at a {CR}
  singular point, {I}},
        date={2016},
        ISSN={0020-9910},
     journal={Invent. Math.},
      volume={206},
      number={2},
       pages={293\ndash 377},
  url={https://doi-org.ezproxy.library.wisc.edu/10.1007/s00222-016-0654-8},
      review={\MR{3570294}},
}

\bib{Gr62}{article}{
      author={Grauert, H.},
       title={\"{U}ber {M}odifikationen und exzeptionelle analytische
  {M}engen},
        date={1962},
        ISSN={0025-5831},
     journal={Math. Ann.},
      volume={146},
       pages={331\ndash 368},
         url={https://doi-org.ezproxy.library.wisc.edu/10.1007/BF01441136},
      review={\MR{0137127}},
}

\bib{GR84}{book}{
      author={Grauert, H.},
      author={Remmert, R.},
       title={Coherent analytic sheaves},
      series={Grundlehren der Mathematischen Wissenschaften [Fundamental
  Principles of Mathematical Sciences]},
   publisher={Springer-Verlag, Berlin},
        date={1984},
      volume={265},
        ISBN={3-540-13178-7},
  url={https://doi-org.ezproxy.library.wisc.edu/10.1007/978-3-642-69582-7},
      review={\MR{755331}},
}

\bib{GR04}{book}{
      author={Grauert, H.},
      author={Remmert, R.},
       title={Theory of {S}tein spaces},
      series={Classics in Mathematics},
   publisher={Springer-Verlag, Berlin},
        date={2004},
        ISBN={3-540-00373-8},
  url={https://doi-org.ezproxy.library.wisc.edu/10.1007/978-3-642-18921-0},
        note={Translated from the German by Alan Huckleberry, Reprint of the
  1979 translation},
      review={\MR{2029201}},
}


\bib{Gr66}{article}{
      author={Griffiths, P.~A.},
       title={The extension problem in complex analysis. {II}. {E}mbeddings
  with positive normal bundle},
        date={1966},
        ISSN={0002-9327},
     journal={Amer. J. Math.},
      volume={88},
       pages={366\ndash 446},
         url={https://doi-org.ezproxy.library.wisc.edu/10.2307/2373200},
      review={\MR{0206980}},
}

\bib{GH94}{book}{
      author={Griffiths, P.},
      author={Harris, J.},
       title={Principles of algebraic geometry},
      series={Wiley Classics Library},
   publisher={John Wiley \& Sons, Inc., New York},
        date={1994},
        ISBN={0-471-05059-8},
         url={https://doi-org.ezproxy.library.wisc.edu/10.1002/9781118032527},
        note={Reprint of the 1978 original},
      review={\MR{1288523}},
}

\bib{HR64}{article}{
      author={Hironaka, H.},
      author={Rossi, H.},
       title={On the equivalence of imbeddings of exceptional complex spaces},
        date={1964},
        ISSN={0025-5831},
     journal={Math. Ann.},
      volume={156},
       pages={313\ndash 333},
         url={https://doi-org.ezproxy.library.wisc.edu/10.1007/BF01361027},
      review={\MR{0171784}},
}

\bib{Hi81}{article}{
      author={Hirschowitz, A.},
       title={On the convergence of formal equivalence between embeddings},
        date={1981},
        ISSN={0003-486X},
     journal={Ann. of Math. (2)},
      volume={113},
      number={3},
       pages={501\ndash 514},
         url={https://doi-org.ezproxy.library.wisc.edu/10.2307/2006994},
      review={\MR{621013}},
}

\bib{Ho65}{article}{
      author={H\"ormander, L.},
       title={{$L^{2}$} estimates and existence theorems for the {$\bar
  \partial $}\ operator},
        date={1965},
        ISSN={0001-5962},
     journal={Acta Math.},
      volume={113},
       pages={89\ndash 152},
         url={https://doi.org/10.1007/BF02391775},
      review={\MR{0179443}},
}


\bib{Ho90}{book}{
      author={H\"ormander, L.},
       title={An introduction to complex analysis in several variables},
     edition={Third},
      series={North-Holland Mathematical Library},
   publisher={North-Holland Publishing Co., Amsterdam},
        date={1990},
      volume={7},
        ISBN={0-444-88446-7},
      review={\MR{1045639}},
}

\bib{Hw19}{article}{
      author={{Hwang}, J.-M.},
       title={{An application of Cartan's equivalence method to Hirschowitz's
  conjecture on the formal principle}},
        date={2019},
        volume={189},
        pages={ 945\ndash 978},
        journal={Annals of Math.},
        number={3}
}

\bib{IP79}{article}{
      author={Ilyashenko, Yu.~S.},
      author={Pjartli, A.~S.},
       title={Neighborhoods of zero type imbeddings of complex tori},
        date={1979},
        ISSN={0321-2971},
     journal={Trudy Sem. Petrovsk.},
      number={5},
       pages={85\ndash 95},
      review={\MR{549623}},
}

\bib{IL05}{article}{
      author={Iooss, G.},
      author={Lombardi, E.},
       title={Polynomial normal forms with exponentially small remainder for
  analytic vector fields},
        date={2005},
        ISSN={0022-0396},
     journal={J. Differential Equations},
      volume={212},
      number={1},
       pages={1\ndash 61},
  url={https://doi-org.ezproxy.library.wisc.edu/10.1016/j.jde.2004.10.015},
      review={\MR{2130546}},
}

\bib{Ko62}{article}{
   author={Kodaira, K.},
   title={A theorem of completeness of characteristic systems for analytic
   families of compact submanifolds of complex manifolds},
   journal={Ann. of Math. (2)},
   volume={75},
   date={1962},
   pages={146--162},
   issn={0003-486X},
   review={\MR{0133841}},
   doi={10.2307/1970424},
}

\bib{KS59}{article}{
      author={Kodaira, K.},
      author={Spencer, D.~C.},
       title={A theorem of completeness of characteristic systems of complete
  continuous systems},
        date={1959},
        ISSN={0002-9327},
     journal={Amer. J. Math.},
      volume={81},
       pages={477\ndash 500},
         url={https://doi-org.ezproxy.library.wisc.edu/10.2307/2372752},
      review={\MR{0112156}},
}


\bib{Ko15}{article}{
      author={Koike, T.},
       title={Toward a higher codimensional {U}eda theory},
        date={2015},
        ISSN={0025-5874},
     journal={Math. Z.},
      volume={281},
      number={3-4},
       pages={967\ndash 991},
  url={https://doi-org.ezproxy.library.wisc.edu/10.1007/s00209-015-1516-6},
      review={\MR{3421649}},
}

\bib{Ko88}{article}{
      author={Kosarew, S.},
       title={Ein allgemeines {K}riterium f\"{u}r das formale {P}rinzip},
        date={1988},
        ISSN={0075-4102},
     journal={J. Reine Angew. Math.},
      volume={388},
       pages={18\ndash 39},
  url={https://doi-org.ezproxy.library.wisc.edu/10.1515/crll.1988.388.18},
      review={\MR{944181}},
}

\bib{LS10}{article}{
      author={Lombardi, E.},
      author={Stolovitch, L.},
       title={Normal forms of analytic perturbations of quasihomogeneous vector
  fields: rigidity, invariant analytic sets and exponentially small
  approximation},
        date={2010},
        ISSN={0012-9593},
     journal={Ann. Sci. \'{E}c. Norm. Sup\'{e}r. (4)},
      volume={43},
      number={4},
       pages={659\ndash 718},
         url={https://doi-org.ezproxy.library.wisc.edu/10.24033/asens.2131},
      review={\MR{2722512}},
}

\bib{LTT17}{article}{
      author={{Loray}, F.},
      author={{Thom}, O.},
      author={{Touzet}, F.},
       title={{Two dimensional neighborhoods of elliptic curves: formal
  classification and foliations}},
        date={2017-04},
     journal={arXiv e-prints},
      eprint={1704.05214},
}

\bib{MR78}{article}{
      author={Morrow, J.},
      author={Rossi, H.},
       title={Submanifolds of {${\bf P}^{n}$} with splitting normal bundle
  sequence are linear},
        date={1978},
        ISSN={0025-5831},
     journal={Math. Ann.},
      volume={234},
      number={3},
       pages={253\ndash 261},
         url={https://doi.org/10.1007/BF01420647},
      review={\MR{0492406}},
}

\bib{Ne89}{article}{
      author={Neeman, A.},
       title={Ueda theory: theorems and problems},
        date={1989},
     journal={Mem. Amer. Math. Soc.},
      volume={81},
      number={415},
}

\bib{NS60}{incollection}{
      author={Nirenberg, L.},
      author={Spencer, D.~C.},
       title={On rigidity of holomorphic imbeddings},
        date={1960},
   booktitle={Contributions to function theory ({I}nternat. {C}olloq.
  {F}unction {T}heory, {B}ombay, 1960)},
   publisher={Tata Institute of Fundamental Research, Bombay},
       pages={133\ndash 137},
      review={\MR{0125982}},
}

\bib{ohsawa-book}{book}{
	author={Ohsawa, T.},
	title={$L^2$ approaches in several complex variables},
	series={Springer Monographs in Mathematics},
	note={Development of Oka-Cartan theory by $L^2$ estimates for the
		$\overline\partial$ operator},
	publisher={Springer, Tokyo},
	date={2015},
	pages={ix+196},
	isbn={978-4-431-55746-3},
	isbn={978-4-431-55747-0},
	review={\MR{3443603}},
	doi={10.1007/978-4-431-55747-0},
}

\bib{Po86}{article}{
      author={P\"{o}schel, J.},
       title={On invariant manifolds of complex analytic mappings near fixed
  points},
        date={1986},
        ISSN={0723-0869},
     journal={Exposition. Math.},
      volume={4},
      number={2},
       pages={97\ndash 109},
      review={\MR{879908}},
}

\bib{Ru02}{article}{
      author={R\"{u}ssmann, H.},
       title={Stability of elliptic fixed points of analytic area-preserving
  mappings under the {B}runo condition},
        date={2002},
        ISSN={0143-3857},
     journal={Ergodic Theory Dynam. Systems},
      volume={22},
      number={5},
       pages={1551\ndash 1573},
  url={https://doi-org.ezproxy.library.wisc.edu/10.1017/S0143385702000974},
      review={\MR{1934150}},
}

\bib{Sa82}{article}{
      author={Savel'ev, V.~I.},
       title={Zero-type imbedding of a sphere into complex surfaces},
        date={1982},
        ISSN={0201-7385},
     journal={Vestnik Moskov. Univ. Ser. I Mat. Mekh.},
      number={4},
       pages={28\ndash 32, 85},
      review={\MR{671883}},
}

\bib{Se66}{article}{
      author={Serre, J.-P.},
       title={Prolongement de faisceaux analytiques coh\'{e}rents},
        date={1966},
        ISSN={0373-0956},
     journal={Ann. Inst. Fourier (Grenoble)},
      volume={16},
      number={fasc. 1},
       pages={363\ndash 374},
         url={http://www.numdam.org/item?id=AIF_1966__16_1_363_0},
      review={\MR{0212214}},
}

\bib{Sh89}{article}{
      author={Shapiro, H.~S.},
       title={An algebraic theorem of {E}. {F}ischer, and the holomorphic
  {G}oursat problem},
        date={1989},
        ISSN={0024-6093},
     journal={Bull. London Math. Soc.},
      volume={21},
      number={6},
       pages={513\ndash 537},
         url={https://doi-org.ezproxy.library.wisc.edu/10.1112/blms/21.6.513},
      review={\MR{1018198}},
}

\bib{St00}{article}{
      author={Stolovitch, L.},
       title={Singular complete integrability},
        date={2000},
        ISSN={0073-8301},
     journal={Inst. Hautes \'{E}tudes Sci. Publ. Math.},
      number={91},
       pages={133\ndash 210 (2001)},
         url={http://www.numdam.org/item?id=PMIHES_2000__91__133_0},
      review={\MR{1828744}},
}

\bib{Ue82}{article}{
      author={Ueda, T.},
       title={On the neighborhood of a compact complex curve with topologically
  trivial normal bundle},
        date={1982/83},
        ISSN={0023-608X},
     journal={J. Math. Kyoto Univ.},
      volume={22},
      number={4},
       pages={583\ndash 607},
         url={https://doi-org.ezproxy.library.wisc.edu/10.1215/kjm/1250521670},
      review={\MR{685520}},
}
\bib{Yoccoz-ast}{article}{
	author		= {Yoccoz, J.-C.},
	title		= {Petits diviseurs en dimension 1},
	journal		= {Ast\'{e}risque},
	year		= {1995},
	volume		= {231}
}
\bib{Zh05}{book}{
      author={Zhu, K.},
       title={Spaces of holomorphic functions in the unit ball},
      series={Graduate Texts in Mathematics},
   publisher={Springer-Verlag, New York},
        date={2005},
      volume={226},
        ISBN={0-387-22036-4},
      review={\MR{2115155}},
}

\end{biblist}
\end{bibdiv}

\end{document}